\documentclass[11pt]{article}
\usepackage[utf8]{inputenc}
\usepackage[british]{babel}
\usepackage{csquotes}
\usepackage{amsmath,amsthm,amssymb,amsfonts}
\usepackage{mathtools}
\usepackage{mathrsfs}

\usepackage[shortlabels]{enumitem}
\usepackage[overload]{empheq}

\usepackage[unicode,colorlinks=true,pagebackref=false,linkcolor=blue,citecolor=green,urlcolor=black]{hyperref}

\usepackage{algorithm}
\usepackage{algorithmic}
\usepackage{dsfont}
\usepackage[justification=centering,singlelinecheck=false]{caption}
\usepackage{subcaption}
\usepackage{graphicx}

\allowdisplaybreaks

\theoremstyle{plain}
\newtheorem{lemma}{Lemma}[section] 

\theoremstyle{definition}
\newtheorem{definition}[lemma]{Definition}

\theoremstyle{remark}
\newtheorem{remark}{Remark}[section] 

\theoremstyle{plain}
\newtheorem{theorem}[lemma]{Theorem}

\theoremstyle{plain}
\newtheorem{corollary}[lemma]{Corollary}

\theoremstyle{plain}
\newtheorem{assumption}[lemma]{Assumption}

\theoremstyle{plain}
\newtheorem{proposition}[lemma]{Proposition}

\theoremstyle{remark}

\numberwithin{equation}{section}

\DeclareMathOperator*{\argmin}{arg\,min}
\DeclareMathOperator*{\esssup}{ess\,sup}

\title{Policy Optimization for Continuous-time Linear-Quadratic Graphon Mean Field Games}

\author{Philipp Plank\thanks{Department of Mathematics, Imperial College London,  London,  UK  ({\tt p.plank24@imperial.ac.uk, yufei.zhang@imperial.ac.uk})} \and Yufei Zhang\footnotemark[1]}
\date{ }

\begin{document}

\maketitle

\begin{abstract}
     Multi-agent reinforcement learning, despite its popularity and empirical success, faces significant scalability challenges in large-population dynamic games. Graphon mean field games (GMFGs) offer a principled framework for  approximating such games while capturing heterogeneity among players. In this paper, we propose and analyze a policy optimization framework for continuous-time, finite-horizon linear-quadratic GMFGs. Exploiting the structural properties of GMFGs, we design an efficient policy parameterization in which each player's policy is represented as an affine function of their private state, with a shared slope function   and  player-specific intercepts. We develop a bilevel optimization algorithm that alternates between policy gradient updates for best-response computation under a fixed population distribution, and distribution updates using the resulting policies. We prove linear convergence of the policy gradient steps to best-response policies and establish global convergence of the overall algorithm to the Nash equilibrium. The analysis relies on novel landscape characterizations over   infinite-dimensional policy spaces. Numerical experiments demonstrate the convergence and robustness of the proposed algorithm under varying graphon structures, noise levels, and action frequencies.

\end{abstract}

\medskip
 

\medskip



\section{Introduction}

Multi-agent reinforcement learning (MARL) has achieved substantial success in finding Nash equilibria in dynamic 
non-cooperative games, with applications in domains such as computer and board games \cite{openai_dota_2019, silver_go}, self-driving vehicle navigation \cite{shalev-shwartz_safe_2016}, and real-time bidding  \cite{jin2018real}. Despite these empirical achievements, 
  MARL faces significant scalability issues as the number of players increases,  owing to the exponential growth in the complexity of the joint strategy space. 
  As a result, developing efficient and theoretically grounded MARL algorithms for computing equilibrium strategies in large-population dynamic games remains a formidable challenge.

A principled approach to addressing these scalability challenges is through the mean field game (MFG) approximation.
Classical  MFG  theory considers settings   with a large number of \emph{homogeneous} agents who interact weakly and \emph{symmetrically} \cite{huangMFG, lasry_mean_2007}. In this regime, by a law of large numbers, the equilibria of finite-player games can be approximated by those of a limiting MFG, where a single representative agent interacts with the population distribution. This approximation greatly simplifies the analysis and computation, leading to provably convergent MARL algorithms for large-population games
 (see, e.g., \cite{angiuli_unified_2022,  pmlr-v130-cui21a, elie_convergence_2020, guo_learning_2019,
 hu2024mf,  mfg_rl_statisticharder,
 huang_stat_effic_mfrl, laurière2024learningmeanfieldgames, wang2021global, pmlr-v139-xie21g, yardim_mfrl_tractable}). 
However, the strong assumptions of homogeneity and symmetry inherent in classical MFG theory limit its applicability to more realistic scenarios involving heterogeneous agents.

The graphon mean field game (GMFG) framework has recently emerged as a powerful tool for analyzing large-population games involving heterogeneous agents. In contrast to classical MFGs, which rely on strong symmetry assumptions among agents, GMFGs capture heterogeneity by modeling agent interactions through large networks and studying their asymptotic behavior via the graphon limit, i.e., a measurable function that describes the limiting structure of interaction graphs (see \cite{lovász2012large} for an overview of graphon theory). This framework allows for the analysis of games where each agent   interacts differently depending on their position within the network.

Most existing work on GMFGs focuses on analytical aspects, assuming known continuous-time system dynamics and cost functions (see e.g., 
\cite{aurell2021stochasticgraphongamesii, Caines_2021, carmona2019stochasticgraphongamesi, gao2021lqggraphonmeanfield, lacker_label-state_2022, 
2023parisegraphongames, neuman_stochastic_2024, foguen-tchuendom_infinite_2024}).
These studies have established  the existence of Nash equilibria of GMFGs and   demonstrated that the equilibria in GMFGs   provide accurate approximations to equilibria in finite-player games when the number of agents is large.   

In contrast, the development of provably efficient learning algorithms for GMFGs remains relatively underexplored. In particular, existing algorithms are largely confined to discrete-time models with finite state and/or action spaces, primarily due to their mathematical tractability (see, e.g., \cite{
cui2022learninggraphonmeanfield, 
fabian2023learningsparsegraphonmean, zhang2023learningregularizedmonotonegraphon, zhang2023learningregularizedgraphonmeanfield}).

\paragraph{Our work.}
 In this work, we propose and analyze the first learning algorithm for continuous-time  GMFGs. 
 We focus on the finite-horizon linear-quadratic  GMFG (LQ-GMFG), a fundamental model in which each agent controls a linear stochastic differential equation to minimize a quadratic cost over a fixed time horizon. Both the state dynamics and cost functions involve heterogeneous interactions with the graphon aggregate of the population.
 Specifically, 

\begin{itemize}
\item 
We  consider general square-integrable 
    graphons, which are   crucial for capturing the limiting behavior of finite-player games with sparse interactions \cite{borgs_lp_2019, borgs_lp_2018, fabian2023learningsparsegraphonmean}. 
    \item We parameterize each player’s policy as an affine function of their private state, with time-dependent slope and intercept terms. 
     Importantly, the same slope function is shared across all players, while the intercept is player-specific and captures each agent’s dependence on the population. 
     This parameterization exploits the structural properties of GMFGs and  significantly reduces the policy complexity compared to conventional approaches, particularly for   large-scale games (see Remark \ref{rmk:complexity_policy}).

\item We propose a bilevel policy optimization algorithm for computing the NE policy (Algorithm \ref{algo}). The outer loop employs a mean field oracle to compute the graphon aggregate induced by the current policy. In the inner loop, each player applies a policy gradient (PG)  method to compute their best response given the fixed graphon aggregate. These PG updates are derived from a novel decomposition of each player's cost function (Theorem \ref{theorem:cost_decomposition}), which enables a \emph{player-independent} update  of the slope parameter and substantially reduces computational complexity in large-population settings (see Remark \ref{rmk:gradient_normalization}).

\item 
We establish new cost landscape properties over the infinite-dimensional policy space using advanced   functional analysis tools. Specifically, we show that along the gradient iterates, the cost function satisfies the Polyak-{\L}ojasiewicz condition and is almost Lipschitz smooth with respect to the slope parameter, as well as  being strongly convex and Lipschitz smooth with respect to the intercept term. Leveraging these regularity properties, we prove that the inner PG method converges globally to the best response policy at a linear rate (Theorems \ref{theorem:K_CONVERGENCE} and \ref{theorem:G_CONVERGENCE}). Furthermore, under a suitable contractivity condition, we prove that the bilevel policy optimization algorithm converges to the NE policy (Theorem \ref{theorem:CONVERGENCE_FINAL}).

To the best of our knowledge, this is the first linear convergence
result for PG methods with  \emph{affine policies} for finite-horizon  continuous-time LQ control 
problems, and the first provably convergent policy optimization algorithm for continuous-time GMFGs. 

\item 
Building on our theoretical analysis, we develop practically implementable policy optimization methods that operate on discrete time grids. A key feature of our approach is the proper scaling of discrete-time gradients with respect to the action frequency, ensuring robust   performance across varying timescales (see Remark \ref{rmk:scaling}). 

\item We conduct extensive numerical experiments to demonstrate the effectiveness of the proposed methods, where policy gradients are estimated from observations using either pathwise differentiation or zeroth-order optimization techniques. The results validate our theoretical findings and show that the algorithm achieves robust linear convergence across a broad spectrum of graphon interaction structures, noise levels, and action frequencies.

\end{itemize}
 
\paragraph{Our approaches.}

A crucial technical step in our analysis  involves establishing  a novel decomposition of each player’s cost function in  Theorem \ref{theorem:cost_decomposition}. This decomposition  is essential to the design of   a convergent PG algorithm for computing each player's best response under a fixed graphon aggregate. Constructing such PG methods requires analyzing LQ control problems with \emph{affine policies}.

While there is a substantial body of work on policy optimization for LQ control with linear policies, the case of affine policies has received far less attention. To the best of our knowledge, affine policy optimization has only been studied in \cite{carmona_linear-quadratic_2019, frikha_full_2024, wang2021global} for infinite-horizon LQ problems with stationary policies under specific structural conditions.

The main difficulty lies in the fact that the cost function depends jointly on the slope and intercept parameters of the affine policy. This coupling significantly complicates the loss landscape, making it difficult to ensure uniform landscape properties along the PG iterates. To address this, prior works impose structural assumptions that either fully decouple the contributions of the slope and intercept terms, allowing them to be optimized separately \cite{carmona_linear-quadratic_2019, frikha_full_2024}, 
or make the intercept term analytically solvable \cite{wang2021global}. 

Unfortunately, the structural assumptions imposed in \cite{carmona_linear-quadratic_2019, frikha_full_2024, wang2021global} do not extend to general LQ GMFGs considered in this paper. Moreover, the finite-horizon cost criterion introduces additional challenges, as it requires optimization over infinite-dimensional spaces of time-dependent policies, in contrast to the finite-dimensional spaces of stationary policies typically considered in infinite-horizon settings.

To   overcome the above challenge,
we build on the   cost decomposition   in Theorem \ref{theorem:cost_decomposition} and  introduce a sequential optimization approach, wherein we first optimize the slope parameter, followed by the intercept parameter conditioned on the optimized slope.
To address the lack of coercivity in the cost function for finite-horizon problems (see \cite[Proposition 2.4]{giegrich_convergence_2022}), we normalize the gradient of the slope parameter using the state covariance matrix, rather than the state second moment matrices used in prior work \cite{fazel_global_2018, giegrich_convergence_2022, hambly_policy_2021}. This yields a player-independent gradient update for the slope parameter that is uniformly bounded.

Moreover, we apply advanced   tools from functional analysis to characterize  the  regularity   of the decomposed cost   with respect to the intercept parameter (Propositions \ref{proposition:J2_strongconvexity} and \ref{proposition:LSmooth}). These  properties  hold uniformly with respect to the slope parameter, population distribution, and  player identity.
The full   convergence analysis of Algorithm \ref{algo} to the NE policy requires careful control of error propagation in both the mean field oracle and the policy parameters throughout their sequential updates.

\paragraph{Most related works.}

There is an extensive body of research on policy optimization algorithms in RL, and it is beyond the scope of this work to cover it comprehensively. Therefore, we focus on the subset of RL literature that is most relevant to our study.

\vspace{2mm}
\noindent \textit{Policy gradient methods.} 
Policy gradient method and its variants  \cite{naturalpolgrad, konda1999actor, silver_dpg, sutton_pg_rl_fa} are 
 the cornerstone of recent successes in RL \cite{sutton2018reinforcement}.  Convergence analyses of policy optimization methods have primarily focused on discrete-time Markov decision processes   (see, e.g., \cite{fazel_global_2018, hambly_policy_2021, kerimkulov2024fisherraogradientflowentropyregularised, pmlr-v119-mei20b} and references therein).

Analysis of policy gradient methods for continuous-time problems are fairly limited. For single-agent control problems, policy gradient methods have been analyzed in the infinite-horizon setting with stationary policies, both for LQ  problems \cite{bu_policy_2020, zhang_convergence_2024} and for more general drift-controlled systems \cite{sethi2024entropy}. In the finite-horizon setting, convergence results have been established for LQ problems with time-dependent policies \cite{giegrich_convergence_2022}. Extensions of policy gradient methods to the  settings involving large populations of homogeneous agents have also been studied both in cooperative control problems \cite{carmona_linear-quadratic_2019, wang2021global} and in LQ MFGs \cite{wang2021global}. Notably, these works are restricted to infinite-horizon formulations and finite-dimensional, stationary policies.
 
In contrast to the above literature, our work  considers 
  finite-horizon GMFGs    with heterogeneous player interactions, and   infinite-dimensional, time-dependent policies.

\vspace{2mm}
\noindent \textit{Learning in MFGs and GMFGs.} 
Numerous learning algorithms have been developed for MFGs and the corresponding finite-player games with homogeneous players (see, e.g., \cite{angiuli_unified_2022, carmona_deep_2023, pmlr-v130-cui21a, elie_convergence_2020, guo_learning_2019, pmlr-v139-xie21g, mirrorasc_mfg}). Convergence is typically guaranteed under suitable contraction conditions \cite{ANAHTARCI2020104744, guo_learning_2019, mirrorasc_mfg} or monotonicity assumptions \cite{hu2024mf}. Sample efficiency   of RL for MFGs is studied in \cite{hu2024mf, mfg_rl_statisticharder, huang_stat_effic_mfrl, yardim_mfrl_tractable}. We refer the reader to the survey article \cite{laurière2024learningmeanfieldgames} for a comprehensive review and additional references.

In contrast, learning algorithms for GMFGs with heterogeneous player interactions remain relatively underexplored. To the best of our knowledge, all existing works are limited to discrete-time settings with finite state and/or action spaces \cite{cui2022learninggraphonmeanfield, fabian2023learningsparsegraphonmean, zhang2023learningregularizedmonotonegraphon, zhang2023learningregularizedgraphonmeanfield}, and they establish convergence results under contractivity conditions \cite{cui2022learninggraphonmeanfield, zhang2023learningregularizedgraphonmeanfield} or monotonicity assumptions \cite{fabian2023learningsparsegraphonmean, zhang2023learningregularizedmonotonegraphon}. Most algorithmic studies on GMFGs focus on bounded graphons, which correspond to limits of finite-player games with dense interaction structures. An exception is \cite{fabian2023learningsparsegraphonmean}, which proposes an online mirror descent algorithm for unbounded graphons, enabling the modeling of sparse interactions. Furthermore, most of the existing literature assumes full knowledge of the graphon and access to exact gradients. Among them, \cite{zhang2023learningregularizedgraphonmeanfield} is the only work that addresses model estimation under an unknown player interaction structure.

\paragraph{Notation.}
For a Euclidean space $E$, we denote by  $\langle \cdot, \cdot \rangle$ the standard scalar product, and 
by $\lvert \cdot \vert$  the induced norm. 
Let $I_d$ be the $d\times d $ identity matrix.
For a matrix $A \in \mathbb{R}^{d \times k}$, we denote by $A^\top$   the transpose of $A$,
by $\mathrm{tr}(A)$ the trace of $A$,
and 
by 
$\lambda_{\mathrm{max}}(A)$ and $\lambda_{\mathrm{min}}(A)$   the largest and smallest eigenvalue of $A$, respectively. 
We denote  by  $\langle \cdot, \cdot \rangle_F$
the Frobenius scalar product of a matrix,
and 
by $\lVert \cdot \rVert_2$ the spectral norm.
 $\mathbb{S}^d, \Bar{\mathbb{S}}_{+}^d$ and $\mathbb{S}_{+}^d$ denote  the space of symmetric, positive semidefinite,  and positive definite matrices, respectively. $\mathbb{S}^d$ is endowed with the Loewner partial order such that for $A, B \in \mathbb{S}^d$,  $A \succcurlyeq B$ if $A - B \in \Bar{\mathbb{S}}_{+}^d$. In particular, $A \succcurlyeq 0$ for $A \in \Bar{\mathbb{S}}_{+}^d$ and $A \succ 0$ for $A \in \mathbb{S}_{+}^d$. 

 For a given Euclidean space $E$, we introduce the following function spaces: \\
$L^2([0,T], E)$ is the space of measurable function $f: [0,T] \to E$ with finite $L^2$-norm induced by the inner product $\langle \cdot, \cdot \rangle_{L^2}$:
\begin{equation*}
    \langle f,g \rangle_{L^2} \coloneqq \int_0^T (f(t))^\top g(t) \, dt, \quad  \forall f,g \in L^2([0,T], E). 
\end{equation*}
 $L^\infty([0,T], E)$ is  the space of measurable functions $f: [0,T] \to E$ which are bounded almost everywhere, i.e., $\lVert f \rVert_{L^\infty} \coloneqq \esssup_{t \in [0,T]}{\lvert f(t) \rvert} < \infty$.
$\mathcal{C}([0,T], E)$   is  the space  of continuous   functions $f:[0,T]\to E$.  $L^2([0,T], L^2([0,1], E))$ is the Bochner-Lebesgue space 
of 
Bochner measurable functions  $f: [0,T]\to  L^2([0,1], E)$
satisfying 
\begin{equation*}
    \lVert f \rVert_{L_B^2} \coloneqq \left( \int_0^T \lVert f(t) \rVert_{L^2([0,1])}^2 \, dt \right)^{1/2}<\infty.
\end{equation*}
 $\mathcal{C}([0,T], L^2([0,1], E))$ 
 is the space of continuous functions 
 $f: [0,T]\to L^2([0,1], E)$ satisfying 
   $\lVert f \rVert_{\mathcal{C}} \coloneqq \sup_{t \in [0,T]}{\lVert f(t) \rVert_{L^2([0,1])}}<\infty$. Note that  the spaces $(L^2([0,T], L^2([0,1], E)), \lVert \cdot \rVert_{L_B^2})$ and $(\mathcal{C}([0,T], L^2([0,1], E)), \lVert \cdot \rVert_{\mathcal{C}})$ are Banach spaces. 
   We denote by $\lVert \cdot \rVert_{\mathrm{op}}$   the   operator norm
   for an operator 
   between Banach spaces, and by  $\mathrm{id}$
   the identity operator on $L^2([0,T], E)$.

\section{Problem formulation and main results}
\label{section:problem_formulation}

This section 
formulates the linear-quadratic graphon mean field game (LQ-GMFG),
proposes a policy optimization algorithm to compute its NE,
and presents the main convergence results.

\subsection{Mathematical setup of LQ-GMFGs}

 This section presents the mathematical formulation of a continuous-time LQ-GMFG. The game consists of a continuum of heterogeneous, competitive, and rational players, each represented by a node in an infinite network. Each player's state evolves according to a linear stochastic differential equation, where the drift coefficient is   controlled and influenced heterogeneously by the overall population state distribution. Each player optimizes a quadratic objective over a finite time horizon. An NE in this game is attained when the distribution of players' optimal state dynamics matches the given population state distribution.

\paragraph{LQ-GMFG and its NE.}

Let $T>0$ be a given terminal time,
and     $I \coloneqq [0,1]$ be  the index set for   the continuum of players,
equipped with the Borel $\sigma$-algebra. 
 
 We first introduce the initial conditions and  noises  for   players' state processes.
Let   
$\mathcal P(\mathbb R^d)$
be the space of probability measures
on $\mathbb R^d$
equipped with the weak topology, 
and  $\nu: I\to \mathcal P(\mathbb R^d)$ be a   measurable function 
where 
$\nu(\alpha)$ has a second moment
for all $\alpha \in I$.
The probability measure $\nu(\alpha)$ models 
 the 
  distribution of player $\alpha$'s initial state. 
Let   $(\Omega, \mathcal F, \mathbb P)$ 
be a probability space, and  
$(\xi^\alpha, B^\alpha)_{\alpha\in I}$   
be 
  a collection of  $\mathbb R^d\times \mathcal{C}([0,T], \mathbb R^d)$-valued random variables 
  defined 
on   $I\times \Omega$,
such that  
$  (\xi^\alpha, B^\alpha )_{\alpha\in I}$
are
independent    random variables,
and 
  for all $\alpha \in I$,
$\xi^\alpha$ has the distribution $\nu(\alpha)$
and $B^\alpha$ is a Brownian motion independent of $\xi^\alpha$ 
 under the measure $\mathbb P$.
 The random variables 
 $\xi^\alpha$ 
  and $B^\alpha$  represent the initial condition and the idiosyncratic noise of the state process for player $\alpha$, respectively. 
To simplify the notation,
we   drop the subscript $\alpha$ when referring to a family of random variables or processes indexed by $\alpha\in  I $, e.g., $\xi= (\xi^\alpha)_{\alpha\in  I }$.

{  

Each player's state dynamics follow a linear SDE, where the evolution depends on their own control as well as   a graphon aggregate   of the population’s state means.
More precisely, 
let 
$\mu = (\mu^\alpha)_{\alpha \in I} \in \mathcal{C}([0,T], L^2(I, \mathbb{R}^d))$
be a given flow of the population’s state means, and define 
$Z=(Z^\alpha)_{\alpha \in I}\in \mathcal{C}([0,T], L^2(I, \mathbb{R}^d))$ by
\begin{equation}
\label{eq:graphon_aggregate}
    Z_t^\alpha \coloneqq 
    \int_I W(\alpha, \beta) \mu_t^\beta \, d\beta, \quad t\in [0,T],
\end{equation}
where $W: I\times I\to \mathbb R$ is a symmetric measurable    function
satisfying $\|W\|_{L^2(I^2)}<\infty$
, referred to as a \emph{graphon}.
For each $\alpha, \beta\in I$, 
$W(\alpha,\beta)$ represents the interdependence between players $\alpha$ and $\beta$,
and $Z^\alpha$
  reflects  how player $\alpha$'s  state dynamics depends on the population distribution.  
  }

\begin{remark}[\textbf{Unbounded graphons}]
    Note that, rather than restricting to bounded graphons $W: I \times I \to [0,1]$ as is common in most of the GMFG literature (see, e.g., \cite{aurell2021stochasticgraphongamesii, Caines_2021, gao2021lqggraphonmeanfield, Gao_2021}), we allow for general $L^2$
    graphons, which are   crucial for capturing the limiting behavior of $N$-player games with sparse interactions \cite{borgs_lp_2019, borgs_lp_2018, fabian2023learningsparsegraphonmean}.
\end{remark}

Given the graphon aggregate $Z$, 
 for each $\alpha \in I$,
 let 
 $\mathbb F^\alpha$ be the filtration generated by
 $(\xi^\alpha, B^\alpha)$, and 
 $\mathcal A^\alpha$ be the
set of  square integrable  $\mathbb F^\alpha$-progressively measurable processes
$U^\alpha: \Omega\times [0,T]\to \mathbb R^k$,
  representing the set of admissible
 controls of player $\alpha$. 
 For each 
 $U^\alpha\in \mathcal A^\alpha$,
 player $\alpha$'s state process is governed by the following dynamics
\begin{equation}
\label{eq:state_dynamics}
    dX_t^\alpha = \left(A(t) X_t^\alpha + B(t) U_t^\alpha + \Bar{A}(t) Z_t^\alpha\right) \, dt + D(t) \, dB_t^\alpha,
    \quad 
    t\in [0,T]; 
     \quad X_0^\alpha = \xi^\alpha, 
\end{equation}
where  $A, \Bar{A}, D \in \mathcal{C}([0,T], \mathbb{R}^{d \times d})$ and 
$B\in \mathcal{C}([0,T], \mathbb{R}^{d \times k})$ are given functions.
Player $\alpha$ considers minimizing the following quadratic cost functional 
\begin{equation}
    \label{eq:quadratic_cost_initial}
    \begin{split}
        J^\alpha(U^\alpha, Z^\alpha)&  \coloneqq \mathbb{E}\Bigg[ \int_0^T \left[(X_t^\alpha - H(t) Z_t^\alpha)^\top Q(t) (X_t^\alpha - H(t) Z_t^\alpha) + (U_t^\alpha)^\top R(t) U_t^\alpha \right] \, dt \\
        &\quad + (X_T^\alpha - \bar{H} Z_T^\alpha)^\top \bar{Q} (X_T^\alpha - \bar {H} Z_T^\alpha) \Bigg]
    \end{split}
\end{equation}
over all $U^\alpha \in \mathcal A^\alpha$, 
where 
$X^\alpha$ satisfies the dynamics
\eqref{eq:state_dynamics}
controlled by  $U^\alpha$ with a fixed graphon aggregate $Z^\alpha$,
and 
$Q \in \mathcal{C}([0,T], \overline{\mathbb{S}}_{+}^d)$, $R \in \mathcal{C}([0,T], \mathbb{S}_{+}^k)$, 
$H \in \mathcal{C}([0,T], \mathbb{R}^{d \times d})$,
$\bar{Q} \in \overline{\mathbb{S}}_{+}^d$ and $\bar{H} \in \mathbb{R}^{d \times d}$ are given coefficients. 
To simplify the notation, 
  the time variable of all model coefficients  will be  dropped when there is no risk of confusion.

Note that
for ease of exposition,
we assume 
  the coefficients in \eqref{eq:state_dynamics} and \eqref{eq:quadratic_cost_initial} are continuous. However,  similar analysis and results extend to measurable coefficients satisfying appropriate integrability conditions   as in \cite{giegrich_convergence_2022}. 
  Moreover,
  we assume that the mean field dependence in both the state dynamics and the cost functional is induced by the same graphon $W$, 
  but  the results naturally extend to settings where these dependencies are governed by different graphons.

In this LQ-GMFG,  player $\alpha$ optimizes \eqref{eq:quadratic_cost_initial} by considering the
graphon aggregate $Z^\alpha$ in  \eqref{eq:state_dynamics} and \eqref{eq:quadratic_cost_initial} 
(or equivalently the population means $(\mu^\alpha)_{\alpha \in I}$ in \eqref{eq:graphon_aggregate})
as given. 
An NE is achieved when the graphon aggregate generated by the players' optimal state processes matches the prescribed graphon aggregate  $Z^\alpha$. 
 
\begin{definition}

\label{def:NE}

We say $U^*=(U^{*,\alpha})_{\alpha\in  I}
\in (\mathcal A^\alpha)_{\alpha \in I} $ and $Z^*=(Z^{*,\alpha})_{\alpha\in  I}\in \mathcal{C}([0,T], L^2( I,\mathbb R^d))$ is an NE of the LQ-GMFG
if 
\begin{enumerate}[(1)]
\item
\label{item:optimality}
$U^*$ is optimal when $Z^*$  is the given graphon aggregate, i.e., 
for all $\alpha\in  I$,
$J^\alpha(U^{*,\alpha},Z^{*,\alpha})\le J^\alpha(U^\alpha,Z^{*,\alpha})$
for all $U^\alpha\in \mathcal A^{\alpha}$.

\item
\label{item:consistency}
$(U^*, Z^*)$ satisfies the consistency condition, i.e., 
    for all $t\in [0,T]$ and $\alpha\in  I$,  $\mu^* = (\mu^{*,\alpha})_{\alpha \in I} \in \mathcal{C}([0,T], L^2([0,1], \mathbb{R}^d))$ and $Z_t^{*,\alpha}= 
    \int_I W(\alpha, \beta) \mu_t^{*,\beta} \, d\beta$ with $\mu_t^{*,\alpha} = \mathbb{E}[X_t^{*,\alpha}]$, where $X^{*,\alpha}$ satisfies 
\eqref{eq:state_dynamics} with  
    $U^\alpha =U^{*,\alpha}$
    and $Z^\alpha= Z^{*,\alpha}$.
\end{enumerate}
\end{definition}

\paragraph{Characterization of NEs.}

The NEs of the LQ-GMFG can be characterized through a coupled forward-backward   system (see e.g., \cite{gao2021lqggraphonmeanfield, xu2024linearquadraticgraphonmeanfield}).
More precisely, 
define  the graphon operator  $\mathbb{W}: L^2(I, \mathbb{R}^d) \to L^2(I, \mathbb{R}^d)$    such that  
\begin{equation}
    \label{eq:graphon_operator}
    \mathbb{W}[f](\alpha) \coloneqq \int_I W(\alpha, \beta) f(\beta) \, d\beta, \quad \forall f \in L^2(I, \mathbb{R}^d),
\end{equation}
where $W$ is the graphon in \eqref{eq:graphon_aggregate}. Since $W\in L^2(I^2)$, the graphon operator $\mathbb{W}$ is a bounded linear operator.
Consider the following  infinite-dimensional   forward-backward  system:
for all $t\in [0,T]$ and $\alpha\in I$, 
\begin{subequations}
\label{eq:semiexplicit_ODEs}
\begin{align}[left = \empheqlbrace\,]
    \begin{split}
        &\frac{\partial  Z_t^{*,\alpha}  }{\partial t}  = \left(A  - B  R^{-1}  B^\top P_t^* \right) Z_t^{*,\alpha} + \Bar{A}\mathbb{W}[Z_t^*](\alpha) - B R^{-1}  B^\top \mathbb{W}[S_t^*](\alpha), 
       \end{split}  \\
      \begin{split}    \label{eq:PI_riccati}
      \frac{\partial P_t^*}{\partial t} + A^\top P_t^* + P_t^* A - P_t^* B   R^{-1} B^\top P_t^* + Q = 0, 
      \end{split}\\ 
         \begin{split}
            &\frac{\partial S_t^{*,\alpha}  }{\partial t}   = -\left(A  - B  R^{-1} B^\top P_t^* \right)^\top S_t^{*,\alpha} + (Q  H - P_t^* \Bar{A}) Z_t^{*,\alpha}, 
    \end{split} \\
      \begin{split}
            &Z_0^{*,\alpha} = \mathbb{W}[\mathbb{E}[\xi]](\alpha), \quad P_T^* = \bar{Q},
            \quad S_T^{*,\alpha} = \bar{Q} \bar{H} Z_T^{*,\alpha}, 
    \end{split} 
  \end{align}
\end{subequations}
where    the time variable of all coefficients is  
dropped.
Above and hereafter, the  ODE systems and their solutions will be understood in the mild sense; see \cite{gao2021lqggraphonmeanfield}.

The following proposition characterizes  NEs  of the LQ-GMFG through   solutions of \eqref{eq:semiexplicit_ODEs},
as established in 
\cite[Proposition 2]{gao2021lqggraphonmeanfield}.

\begin{proposition}
    \label{proposition:existenceANDuniqueness}
   Suppose   that the system
   \eqref{eq:semiexplicit_ODEs} has a unique solution   $P^* \in \mathcal{C}([0,T], \mathbb{S}^d)$,
$Z^*\in \mathcal{C}([0,T], L^2(I, \mathbb{R}^d))$ and $S^* \in  \mathcal{C}([0,T], L^2(I, \mathbb{R}^d))$.
Then the LQ-GMFG has  a unique NE  $(U^*,Z^*)$, 
with   $U^*=(U^{*,\alpha})_{\alpha \in I}$
given by 
$U_t^{*,\alpha} = K^*_t X_t^{*,\alpha} + G_t^{*,\alpha}$
for all $\alpha \in I$ and a.e.~$t\in [0,T]$, where
\begin{equation}
    \label{eq:systemcontrol}
    K^*_t =- R^{-1} B^\top P_t^*,
    \quad
    G^{*,\alpha}_t
    =- R^{-1} B^\top S_t^{*,\alpha},
\end{equation}
and 
$X^{*,\alpha}$ satisfies the dynamics 
\begin{equation}
\label{eq:equil_state}
    dX_t^{*,\alpha} = \left(A  X_t^{*,\alpha} + B  U_t^{*,\alpha}  + \Bar{A}  Z_t^{*,\alpha}\right) \, dt + D  \, dB_t^\alpha,
    \quad 
    t\in [0,T]; 
     \quad X_0^{*,\alpha} = \xi^\alpha. 
\end{equation}
 
\end{proposition}

\begin{remark}[\textbf{Structure of NE policies}]
\label{rmk:NE_policy}
 Proposition \ref{proposition:existenceANDuniqueness} shows that at the NE, each player utilizes an affine policy that depends   on their own state but is independent of other players' states. Moreover, all players' policies share the same slope function, with each having a player-specific intercept term that captures the individual player's dependence on the equilibrium graphon aggregate. This observation is essential for designing efficient policy gradient methods to solve LQ-GMFGs;
  see Remarks \ref{rmk:complexity_policy}
  and \ref{rmk:gradient_normalization}. 
\end{remark}

Under appropriate conditions on the coefficients, \cite{gao2021lqggraphonmeanfield} establishes the well-posedness of \eqref{eq:semiexplicit_ODEs} using the Banach fixed-point theorem, which in turn implies the well-posedness of NEs for LQ-GMFGs.
In the sequel, we design policy gradient   algorithms to compute this Nash equilibrium and establish their convergence under similar conditions on the model coefficients.
The precise conditions will be stated in Assumption \ref{assumption:convergence} before   presenting the convergence results.

\subsection{Policy optimization algorithm    for LQ-GMFGs}

This section proposes a policy optimization  method to compute the NE of the LQ-GMFG. The algorithm employs affinely parameterized policies and seeks the NE policy through a bilevel optimization procedure: it iteratively updates all players’ policies via gradient descent with a fixed graphon aggregate, and then updates the graphon aggregate based on the latest policy parameters.

 \paragraph{Policy class.}

 We consider the following class  of affine   policies,
motivated by  the form of the NE policy  given  in \eqref{eq:systemcontrol}
(see Remark \ref{rmk:NE_policy}): 
\begin{equation}
\label{eq:policy_parameterization}
 \mathcal V  \coloneqq 
 \left\{\phi_\theta=(\phi^\alpha_\theta)_{\alpha\in I}
 \,\middle\vert\, 
\begin{aligned}
&\phi^{\alpha}_\theta(t,x)= K_t x+G^\alpha_t,
\quad 
 \forall
(t,x)\in  [0,T]\times \mathbb R^d, 
 \alpha\in I, 
\\
&
\theta= \big(K,(G^\alpha)_{\alpha\in I}\big) \in L^2([0,T], \mathbb{R}^{k \times d}) \times L^2([0,T], L^2(I, \mathbb{R}^k))
 \end{aligned}
 \right\}.
\end{equation}
As  $R \in \mathcal{C}([0,T], \mathbb{S}_{+}^k)$, 
one can easily verify that 
if   \eqref{eq:semiexplicit_ODEs} has a unique solution,
then the NE policy given by \eqref{eq:systemcontrol}
lies in the class $ \mathcal V$. 
In the sequel, we   identify $\phi\in \mathcal V$ 
with  its parameter 
$\theta= \big(K,(G^\alpha)_{\alpha\in I}\big)
\in L^2([0,T], \mathbb{R}^{k \times d}) \times L^2([0,T], L^2(I, \mathbb{R}^k))$. 
We call $K$  and $G^\alpha$ the slope and  intercept parameters of player $\alpha$'s policy, respectively. 

\begin{remark}[\textbf{Efficient  policy parameterization}]
\label{rmk:complexity_policy}
The policy class \eqref{eq:policy_parameterization} exploits the structure of GMFGs and yields a more efficient  policy parameterization 
than conventional approaches, especially 
for  large-scale games with many players and high-dimensional state spaces. 

To illustrate the difference,   consider a finite-horizon $N$-player LQ game where each player controls a $d$-dimensional state process. 
In this setting,  the conventional  approach    would   parameterize    player $i$'s policy,  $   i=1,\ldots, N $,  in  a linear  form $\phi^i(t,x) = K^i_t x$, where $K^i: [0,T] \to \mathbb{R}^{k \times Nd}$ is a
 player-specific 
 feedback function      depending    on all players' states     (see, e.g., \cite{hambly_nplayer}). The resulting  parameter space is  $L^2([0,T], \mathbb{R}^{k \times Nd})^N \cong L^2([0,T], \mathbb{R}^{N^2kd})$,
 whose complexity grows quadratically with respect to  the number of players $N$,
  with the state dimension $d$  acting as a multiplicative constant.
 
 In contrast, with the parameterization in \eqref{eq:policy_parameterization}, each player's policy 
 performs feedback control 
  based only on their individual state,    determined by the feedback function $K$, while the influence of other players is captured by the player-specific intercept term $G^\alpha$. Crucially, the feedback function $K$ is shared across all players. Therefore, 
  for the $N$-player system obtained by discretizing the index set  $I$ of the LQ-GMFG,
  the parameter space is $L^2([0,T], \mathbb{R}^{k \times d}) \times L^2([0,T], \mathbb{R}^k)^N \cong L^2([0,T], \mathbb{R}^{k(d+N)})$. 
The resulting  complexity grows linearly with respect to $N$ and, importantly, the growth rate is independent of the state dimension $d$, leading to a   more efficient policy parameterization, particularly for large   $d$ and $N$.

Furthermore, the shared feedback function $K$ 
allows for designing  learning algorithms that apply a common gradient update to $K$ across all players,
which  further enhances computational efficiency in searching for the NE policy; see  Remark
\ref{rmk:gradient_normalization} for more details.

\end{remark}

\paragraph{Optimization costs for    best response policies.}  

To search an NE  policy within the policy class $\mathcal V$,  an essential step is to compute each    player's best response policy for a fixed graphon aggregate.

Specifically, given a graphon aggregate $Z=(Z^\alpha)_{\alpha \in I}\in \mathcal{C}([0,T], L^2(I, \mathbb{R}^d))$,
player $\alpha \in I$ considers  minimizing the following cost functional 
 \begin{align}
          \label{eq:cost}
 \begin{split}
       &J^\alpha(K, G^\alpha, Z^\alpha) 
       \\
        &= \mathbb{E}\Bigg[  \int_0^T \left[ \left(X_t^{\theta,\alpha} - H Z_t^\alpha \right)^\top Q  \left(X_t^{\theta,\alpha} - H Z_t^\alpha \right)  
+ \left(K_t X_t^{\theta,\alpha} + G_t^\alpha \right)^\top R  \left(K_t X_t^{\theta,\alpha} + G_t^\alpha \right) \right] \, dt 
         \\
        &\quad  + \left(X_T^{\theta,\alpha} - \bar{H} Z_T^\alpha \right)^\top \bar{Q} \left(X_T^{\theta,\alpha} - \bar{H} Z_T^\alpha \right) \Bigg]
 \end{split}       
\end{align}
over all ${\theta}\coloneqq (K,G^\alpha)\in L^2([0,T], \mathbb{R}^{k \times d})  \times  L^2([0,T], \mathbb{R}^{k})$,
where for each ${\theta}\coloneqq (K,G^\alpha)$, $X^{\theta,\alpha} $
satisfies the following dynamics:
\begin{equation}
\label{eq:state_player_alpha}
    dX_t= \left((A +B  K_t) X_t  + B  G^\alpha_t   + \Bar{A}  Z_t^{\alpha}\right) \, dt + D  \, dB_t^\alpha,
    \quad 
    t\in [0,T]; 
     \quad X_0 = \xi^\alpha. 
\end{equation}
 
 Note that we aim to  optimize the cost functional \eqref{eq:cost} over affine policies that involve both the slope   $K$ and the intercept $G^\alpha$.
This is in contrast to most existing works on PG  methods for LQ control problems, which   focus on   linear policies without an intercept term (i.e., $G^\alpha \equiv 0$) (see, e.g., \cite{fazel_global_2018, giegrich_convergence_2022, hambly_policy_2021} and references therein).
Allowing for general affine policies introduces significant challenges in the landscape analysis, owing to the nonlinear coupling between $K$ and $G^\alpha$ in the cost functional $J^\alpha $.
In particular, it becomes necessary to carefully characterize how $G^\alpha$ affects the geometry of the map $K \mapsto J^\alpha(K, G^\alpha, Z^\alpha)$, and conversely, how $K$ influences the landscape of the map $G^\alpha \mapsto J^\alpha(K, G^\alpha, Z^\alpha)$.

To characterize the   contributions of $K$ and $G^\alpha$, 
we establish a key decomposition of    the cost functional $J^\alpha$  into   $J^\alpha_1(K)+J^\alpha_2(K, G^\alpha, Z^\alpha)$.
The term $J^\alpha_1$ depends only on $K$, while the term $J^\alpha_2$ depends on both $K$ and $G^\alpha$, 
but its minimum with respect to $G^\alpha$ is independent of $K$.
This decomposition is crucial for   both the design and the convergence analysis of the policy optimization   algorithms (see Remark \ref{rmk:cost_decomposition}).

To this end, 
for a given ${\theta}\coloneqq (K,G^\alpha)$, let  $X^{\theta,\alpha} $ be the corresponding state process satisfying 
\eqref{eq:state_player_alpha}. 
Define the state mean vector $\mu^{\theta,\alpha}_t \coloneqq \mathbb E[X^{\theta, \alpha}_t]$ for all $t\in [0,T]$, which satisfies 
\begin{equation}
\label{eq:mean}
    \begin{split}
        &\frac{\partial \mu_t}{\partial t} = (A  + B  K_t) \mu_t + B  G_t^\alpha + \Bar{A}  Z_t^\alpha, \quad
            t\in [0,T]; \quad 
         \mu_0= \mathbb{E}[\xi^\alpha].
    \end{split}
\end{equation}
Similarly, define the state covariance process $\vartheta_t^{\theta, \alpha} \coloneqq \operatorname{Cov}[X_t^{\theta,\alpha}] = \mathbb{E}[X^{\theta, \alpha}_t(X^{\theta, \alpha}_t)^\top ] - \mu_t^{\theta,\alpha}(\mu_t^{\theta,\alpha})^\top$ for all $t\in [0,T]$, which satisfies 
\begin{equation}
\label{eq:variance}
    \frac{\partial \vartheta_t}{\partial t} = (A  + B  K_t) \vartheta_t  + \vartheta_t  (A  + B  K_t)^\top + D D^\top, \quad 
    t\in [0,T]; \quad \vartheta_0 =  \operatorname{Cov}[\xi^\alpha].
\end{equation}
In the sequel, we   write $\vartheta^{\theta, \alpha}=\vartheta^{K, \alpha}$
to highlight the fact that $\vartheta^{\theta, \alpha}$ is independent of $G^\alpha$.

Using the processes $\mu^{\theta,\alpha}$ and   $\vartheta^{K, \alpha}$,
define the following cost functionals $J_1^\alpha$ and $J_2^\alpha$ by
\begin{align}
\begin{split}
\label{eq:J1_cost}
   & J_1^\alpha(K) \coloneqq \int_0^T \mathrm{tr}\left((Q + K_t^\top R  K_t)
     \vartheta_t^{K,\alpha} \right) \, dt + \mathrm{tr}(\bar{Q} \vartheta_T^{K,\alpha}),
\end{split}
\\
\begin{split}
\label{eq:J2_cost}
 & J_2^\alpha(K,G^\alpha,Z^\alpha) 
        \\
        &  \coloneqq  \int_0^T \left[ \left(\mu_t^{\theta, \alpha} - H Z_t^\alpha\right)^\top Q  \left(\mu_t^{\theta, \alpha} - H  Z_t^\alpha\right) + \left(K_t \mu_t^{\theta, \alpha} + G_t^\alpha\right)^\top R  \left(K_t \mu_t^{\theta, \alpha} + G_t^\alpha\right) \right] \, dt \\
        &\quad + \left(\mu_T^{\theta, \alpha} - \Bar{H} Z_T^\alpha\right)^\top \bar{Q} \left(\mu_T^{\theta, \alpha} - \Bar{H} Z_T^\alpha\right).
\end{split}
\end{align}
Note that it is essential to formulate 
the cost functional
  $J_1^\alpha$   using  the state covariance $\vartheta^{K,\alpha}$,
instead of  using  the  state second moment matrix $\Sigma^{\theta, \alpha} = \mathbb{E}[X^{\theta, \alpha}_t (X^{\theta, \alpha}_t)^\top]$ as in  most of the existing literature 
for LQ control problems with  linear policies (e.g., \cite{fazel_global_2018, giegrich_convergence_2022, hambly_policy_2021}).
This is  because, in the present setting with affine policies, $\Sigma^{\theta, \alpha}$ depends on the nonzero intercept parameter $G^\alpha$, 
and hence does not yield a cost functional   depending only on $K$.

The following theorem expresses  the cost functional $J^\alpha$ using $J^\alpha_1$ and $J^\alpha_2$.
The proof  
  is given   in Section \ref{subsection:decomp_proof}.

\begin{theorem}
\label{theorem:cost_decomposition}
    Let $\alpha \in I$ and $Z^\alpha \in \mathcal{C}([0,T], \mathbb{R}^d)$.
    For all 
 $K\in L^2([0,T], \mathbb{R}^{k \times d})$ and $ G^\alpha\in L^2([0,T], \mathbb{R}^{k})$,
\begin{equation}
\label{eq:cost_decomposition}
 J^\alpha(K,G^\alpha, Z^\alpha) =  J_1^\alpha(K) + J_2^\alpha(K,G^\alpha,Z^\alpha),
\end{equation} 
and $\inf_{G^\alpha \in L^2([0,T], \mathbb{R}^{k})} J_2^\alpha(K,G^\alpha,Z^\alpha)$ is independent of $K$. 
 Consequently, if
  $$
  K^*\in \argmin_{K\in L^2([0,T], \mathbb{R}^{k \times d})} J_1^\alpha(K),
  \quad \text{and} \quad
G^{*, Z, \alpha} \in \argmin_{G^\alpha \in L^2([0,T], \mathbb{R}^{k})}J_2^\alpha(K^*,G^\alpha,Z^\alpha),
$$
 then 
 $
 J^\alpha(K,G^\alpha, Z^\alpha) \ge J^\alpha(K^*,G^{*, Z, \alpha}, Z^\alpha) 
 $ 
for all 
 $(K,G^\alpha)\in L^2([0,T], \mathbb{R}^{k \times d})\times L^2([0,T], \mathbb{R}^{k})$.

       \end{theorem}

\begin{remark}[\textbf{Cost Decomposition}]

\label{rmk:cost_decomposition}

 The cost decomposition 
\eqref{eq:cost_decomposition}
 provides a foundation for designing a sequential policy gradient   algorithm to compute the slope and intercept parameters of each player's best response policy.  
 Specifically, one first applies a gradient descent algorithm to minimize $J^\alpha_1 $,   obtaining an approximate slope parameter. Then, fixing this slope,  a second gradient descent algorithm is used to minimize 
$J^\alpha_2 $  with respect to  $G^\alpha $, thereby determining the intercept parameter.

To the best of our knowledge, this is the first work to establish the cost decomposition \eqref{eq:cost_decomposition} for finite-horizon LQ-GMFGs with heterogeneous players.
In contrast, similar decompositions have previously been developed only for infinite-horizon problems with homogeneous agents,   under the restriction to \emph{finite-dimensional} stationary policies and additional structural assumptions on the model. Specifically, prior work either assumes a full decoupling of the contributions of $K$ and $G^\alpha$ to   the total cost $J$, allowing them to be optimized independently \cite{carmona_linear-quadratic_2019,frikha_full_2024}, or relies on cases where the minimizer of $J_2$  can be solved explicitly \cite{wang2021global}.
In this work, we remove these structural assumptions and analyze finite-horizon problems over infinite-dimensional policy spaces, allowing for player-dependent and time-dependent policies.

\end{remark}

\paragraph{Policy optimization  algorithms for LQ-GMFG.}
 
 The    algorithm  proceeds  by alternating between the following two steps: (i) for a given population distribution, apply a sequential policy gradient (PG) update to approximate each player's best response policy
 (see Remark \ref{rmk:cost_decomposition}); (ii) update the population distribution  using the approximate best response policies obtained from the previous iteration.
For simplicity,   we present the   algorithm assuming  that agents have full knowledge of the model coefficients,
and   refer the reader to Section \ref{section:numerics} for model-free extensions. 

 Fix an arbitrary population state mean $\mu=(\mu^\alpha)_{\alpha \in I}\in  \mathcal{C}([0,T], L^2(I, \mathbb{R}^d))$,
and  let $Z=(Z^\alpha)_{\alpha \in I}\in  \mathcal{C}([0,T], L^2(I, \mathbb{R}^d))$ be the corresponding 
  graphon aggregate defined by \eqref{eq:graphon_aggregate}.

  The algorithm first employs preconditioned gradient descent on the cost functional 
$J^\alpha_1$    to compute the slope parameter of each player's best response policy.
 More precisely, 
for each initial guess $K^{(0)} \in L^\infty([0,T], \mathbb{R}^{k \times d})$
and stepsize $\eta_K>0$, consider 
$(K^{(\ell)})_{\ell\in \mathbb N_0} \subset  L^2([0,T], \mathbb{R}^{k \times d})$
such that for all $\ell\in \mathbb{N}_0$,
\begin{equation}
    \label{eq:P1_graddescent}
    K_t^{(\ell+1)} \coloneqq K_t^{(\ell)} - \eta_K 
 (\nabla_K J_1^\alpha(K^{(\ell)}))_t (\vartheta_t^{(\ell),\alpha})^{-1}, 
 \quad \textnormal{a.e.~$t\in [0,T]$},
\end{equation}
where $\nabla_K J_1^\alpha(K^{(\ell)})$ is the  (Fr\'echet) derivative of $J^\alpha_1$ at $K^{(\ell)}$
and $\vartheta^{{(\ell)},\alpha}$ is the associated state covariance defined by \eqref{eq:variance} with $K=K^{(\ell)}$.
 The gradient   $\nabla_K J_1^\alpha$ can be expressed using the model coefficients as follows (Lemma \ref{lemma:K_cost_derivative}):
\begin{align}
\label{eq:J1_K}
 \nabla_K J_1^\alpha(K^{(\ell)}) = 2 (B^\top P^{(\ell)} + R K^{(\ell)})\vartheta^{{(\ell)},\alpha},
\end{align}
where  
 $P^{(\ell)} \in \mathcal{C}([0,T], \mathbb{S}^d)$ satisfies for all $t\in [0,T]$,
\begin{equation}
\label{eq:ODE_P}
    \frac{\partial P_t}{\partial t} + (A + B  K^{(\ell)}_t)^\top P_t + P_t (A  + B  K^{(\ell)}_t) + Q  + (K^{(\ell)}_t)^\top R K^{(\ell)}_t = 0, \quad P_T = \bar{Q}.
\end{equation}

The update in \eqref{eq:P1_graddescent} normalizes the exact gradient $\nabla_K J^\alpha_1$    by the state covariance. This differs from existing PG algorithms in LQ control problems \cite{fazel_global_2018, giegrich_convergence_2022} or finite-player LQ games \cite{hambly_nplayer}, where the exact gradient is normalized using the state's second moment matrix.
This normalization step is crucial for both the algorithm efficiency and its convergence analysis, as discussed in the following remark.

\begin{remark}[\textbf{Gradient normalization in the update 
 of $K$}]
\label{rmk:gradient_normalization}

By   \eqref{eq:J1_K}, 
 normalizing the vanilla     gradient  $\nabla_K J_1^\alpha(K^{(\ell)}) $ 
 with the covariance matrix $\vartheta^{{(\ell)},\alpha}$ yields a gradient direction that 
 is independent of the player index $\alpha$.
 Consequently, the   update \eqref{eq:P1_graddescent}
 of the parameter  $K$
  is player-independent, unlike existing PG methods that require separate updates for each player's feedback matrix (see e.g.,  \cite{hambly_nplayer}). This shared update rule significantly enhances the algorithm  efficiency, especially in large-scale games   involving many  players.

Normalizing the gradient with the state covariance matrix is crucial for establishing the convergence   of the PG algorithm. 
Specifically, it  ensures that the iterates 
$(K^{(\ell)})_{\ell\in \mathbb N_0}$  remain uniformly bounded for all $\ell$  (Proposition \ref{prop:BOUND_K}), which subsequently guarantees that   all essential landscape properties of the optimization  objectives
 hold uniformly along the policy updates.
Unlike existing works that consider  finite-dimensional policy classes, it is challenging
to guarantee   these uniform landscape properties   for   unnormalized PG   algorithms in the present setting
with infinite-dimensional policy classes. This is due to    the lack of coercivity in the optimization objective,  as noted in \cite[Remark 2.3]{giegrich_convergence_2022}.

\end{remark}

After executing \eqref{eq:P1_graddescent} for 
multiple iterations, we then apply gradient descent to 
 $J^\alpha_2$ 
  to update the intercept parameters of all players' best response policies.
More precisely, 
for each initial guess  $G^{(0)}=(G^{(0),\alpha})_{\alpha\in I} \in L^2([0,T], L^2(I, \mathbb{R}^k))$ 
and stepsize $\eta_G>0$, consider 
$(G^{(\ell)})_{\ell\in \mathbb N_0} \subset  L^2([0,T], L^2(I, \mathbb{R}^k))$
such that for all $\ell\in \mathbb{N}_0$ and $\alpha \in I$,
\begin{equation}
    \label{eq:P2_graddescent}
    G_t^{(\ell+1), \alpha} \coloneqq G_t^{(\ell), \alpha} - \eta_G (\nabla_G J_2^\alpha(K^{(L)}, G^{(\ell), \alpha}, Z^\alpha))_t,
     \quad \textnormal{a.e.~$t\in [0,T]$},
\end{equation}
where  $\nabla_G J_2^\alpha $ is the  (Fr\'echet) derivative of $J^\alpha_2$ with respect to $G^\alpha$,
  $K^{(L)}$ is the  approximate slope parameter obtained from \eqref{eq:P1_graddescent} after $L$ iterations,
and $Z^\alpha$ is the fixed graphon aggregate. 

The gradient   $\nabla_G J_2^\alpha $ admits the following analytic representation in terms of model parameters 
(see Lemma \ref{lemma:derivative_G}):
\begin{equation}
    \label{eq:grad_J2}
\nabla_G J_2^\alpha(K^{(L)}, G^{(\ell), \alpha}, Z^\alpha) =  B^\top \zeta^{(\ell), \alpha} + 2R (K^{(L)} \mu^{(\ell), \alpha} + G^{(\ell), \alpha}),
\end{equation}
where 
$\mu^{(\ell), \alpha}$ is the   mean 
of player $\alpha$'s state process 
under the policy $(t,x)\mapsto  K^{(L)}_t x+G^{(\ell), \alpha}_t$
and the  graphon aggregate $Z^\alpha$ (see \eqref{eq:mean}), 
and   $\zeta^{(\ell),\alpha} \in \mathcal{C}([0,T], \mathbb{R}^d)$ satisfies
\begin{equation}
    \label{eq:ode_zeta_gradG}
    \begin{split}
        &\frac{\partial \zeta_t^\alpha}{\partial t} = -\big( (A+ B K_t^{(L)})^\top \zeta_t^\alpha + 2Q(\mu_t^{(\ell), \alpha} - H Z_t^\alpha) + 2 (K_t^{(L)})^\top R  (K_t^{(L)} \mu_t^{(\ell), \alpha} + G_t^{(\ell), \alpha)} \big), \\
        &\zeta_T^\alpha = 2 \bar{Q} (\mu_T^{(\ell), \alpha} - \Bar{H} Z_T^\alpha).
    \end{split}
\end{equation}
Note that  
although  the graphon aggregate   remains fixed during the policy updates, both $\mu^{(\ell), \alpha}$ and  $\zeta^{(\ell),\alpha} $   evolve due to the update of $G^{(\ell)}$ in \eqref{eq:P2_graddescent}.

Finally, 
after  executing \eqref{eq:P2_graddescent} for multiple iterations, we update the population mean (and consequently the graphon aggregate) using the approximate slope and intercept parameters. Rather than specifying an exact update rule, we assume access to a mean field oracle:
$$
\Hat{\Phi}: L^2([0,T], \mathbb{R}^{k \times d}) \times L^2([0,T], L^2(I, \mathbb{R}^k)) \to 
\mathcal C
([0,T],   L^2 (I, \mathbb{R}^d)),
$$ 
which  for any given policy $(K, G)$,
approximates the corresponding population mean
satisfying for all $\alpha \in I$,
\begin{equation}
\label{eq:mean_policy}
    \begin{split}
        &\frac{\partial \mu_t}{\partial t} = (A  + B  K_t) \mu_t + B  G_t^\alpha + \Bar{A}  \mathbb W[\mu_t](\alpha), \quad
            t\in [0,T]; \quad 
         \mu_0= \mathbb{E}[\xi^\alpha].
    \end{split}
\end{equation}
In the model-free case, the oracle can be realized by generating multiple state trajectories under the given policy and taking their empirical average.
 The precise condition of the oracle $\Hat{\Phi}$ 
 for the convergence analysis will be given 
   in Theorem \ref{theorem:CONVERGENCE_FINAL}.

 Algorithm \ref{algo}   summarizes  the   policy optimization  algorithm  presented above.

\begin{algorithm}[H]
\caption{Policy Optimization Algorithm for LQ-GMFGs}
\label{algo}
\begin{algorithmic}[1]
\STATE \textbf{Input:} Number of   outer   iterations $N$, 
numbers of     updates $(L^K_n)_{n=0}^{N-1}$ for $K$, 
numbers of    updates $(L^G_n)_{n=0}^{N-1}$ for $G$,
stepsizes  $\eta_K,\eta_G>0$, 
initial policy parameters $K^{(0)}$ and $G^{(0)}=(G^{(0), \alpha})_{\alpha \in I}$,
and initial population mean  $\mu^{(0)}=(\mu^{(0), \alpha})_{\alpha\in I}$.

\FOR{$n = 0, \dots, N - 1$}
    \STATE $Z^{(n),\alpha} \coloneqq \int_{I} W(\alpha, \beta) \mu^{(n), \beta} \, d\beta$   for all $\alpha \in I$.
    \STATE $K^{(n), (0)} \coloneqq K^{(n)}$ and 
 $G^{(n), (0)} \coloneqq G^{(n)}$.
    \FOR{$\ell = 1, \dots, L_n^K$}
        \STATE $K^{(n), (\ell)} = K^{(n), (\ell-1)} - \eta_K
         \nabla_K J_1^\alpha(K^{(n), (\ell-1)}) \big(\vartheta^{K^{(n), (\ell-1)} ,\alpha}\big)^{-1} $ with an arbitrary    $\alpha\in I$.
    \ENDFOR
    \STATE $K^{(n+1)} \coloneqq K^{(n), (L_n^K)}$
    \FOR{$\ell = 1, \dots, L_n^G$}
        \STATE $G^{(n), (\ell), \alpha} = G^{(n), (\ell-1), \alpha} - \eta_G   
              \nabla_{G} J^\alpha_2  (K^{(n+1)}, G^{(n), (\ell-1), \alpha}, Z^{(n),\alpha})$ for all $\alpha \in I$.
    \ENDFOR
    \STATE $G^{(n+1), \alpha} \coloneqq G^{(n), (L_n^G), \alpha}$ for all $\alpha \in I$.
    \STATE $\mu^{(n+1)} \coloneqq \Hat{\Phi}[K^{(n+1)}, G^{(n+1)}]$.
\ENDFOR
\STATE \textbf{Return:} $(K^{(N)}, G^{(N)}, \mu^{(N)})$
\end{algorithmic}
\end{algorithm}

Algorithm \ref{algo} extends the bilevel optimization framework  in \cite{wang2021global},
developed for infinite horizon MFGs with homogeneous players and stationary policies, to the   finite horizon GMFGs with heterogeneous players and time-dependent policies. 
 In contrast to \cite{wang2021global}, which requires an analytical solution for the best response intercept, Algorithm \ref{algo} uses a policy gradient method to approximate the intercept, as no closed-form minimizer exists for 
$J^\alpha_2$   in our setting.

\subsection{Convergence  of policy optimization   methods for LQ-GMFGs}

This section presents the linear convergence of Algorithm \ref{algo} to the NE policy. Specifically, we first establish   the convergence of the PG updates  for the best response slope parameter (Theorem \ref{theorem:K_CONVERGENCE}) and the intercept parameter 
(Theorem \ref{theorem:G_CONVERGENCE}). Then 
 we present the convergence to the NE policies and the corresponding graphon aggregate
(Theorem \ref{theorem:CONVERGENCE_FINAL}).

For the convergence analysis, 
observe that 
as $R \in \mathcal{C}([0,T], \mathbb{S}_{+}^k)$,
there exist    $\overline{\lambda}^R\ge 
\underline{\lambda}^R> 0$   such that 
$\lambda_{\mathrm{max}} (R(t))\le  \overline{\lambda}^R $ 
and 
$\lambda_{\mathrm{min}} (R(t))\ge  \underline{\lambda}^R $ for all $t\in [0,T]$.
We further  impose the following uniform boundedness and non-degeneracy conditions on  the initial state covariance matrices.
 
\begin{assumption}
     \label{assumption:RandCOV}
    
     There exists 
 $    \overline{C}_0^\vartheta, \underline{C}_0^\vartheta > 0$ such that 
   for all $\alpha\in  I$,
  the covariance matrix  
 $\vartheta_0^\alpha
 = \operatorname{Cov}[\xi^\alpha] $ 
 of 
 player $\alpha$'s initial state 
  $\xi^\alpha $ satisfies 
  $\lvert \vartheta_0^\alpha \rvert \le \overline{C}^\vartheta$ and $\lambda_{\mathrm{min}}(\vartheta_0^\alpha) \ge \underline{C}_\vartheta$.
 
 \end{assumption}

\paragraph{Convergence of the slope parameter}

As $R \in \mathcal{C}([0,T], \mathbb{S}_{+}^k)$,
the Riccati equation \eqref{eq:PI_riccati} has a unique solution 
$P^* \in \mathcal{C}([0,T], \mathbb{S}^d)$ (\cite[Chapter 6, Theorem 7.2]{yong_stochastic_1999}),
and 
 the NE slope parameter $K^*$  in  \eqref{eq:systemcontrol} is well-defined. 
 The following theorem shows that the iterates $(K^{(\ell)})_{\ell\in \mathbb N_0}$ generated by \eqref{eq:P1_graddescent}
  converge to $K^*$ with a linear rate.

\begin{theorem}
\label{theorem:K_CONVERGENCE}
Suppose Assumption \ref{assumption:RandCOV} holds.
Let $K^{(0)} \in L^\infty([0,T], \mathbb{R}^{k \times d})$,
and $(K^{(\ell)})_{\ell\in \mathbb N_0}\subset L^\infty([0,T], \mathbb{R}^{k \times d})$ be  defined by \eqref{eq:P1_graddescent} with 
$\eta_K>0$. There exists $C_1^K, C_2^K, C_3^K > 0$ such that for 
$        \eta_K \in \left(0, C^K_1 \right)
$, 
    \begin{equation*}
        \lVert K^{(\ell)} - K^* \rVert_{L^2([0,T])}^2 \le C_3^K \left(1 - \eta_K C_2^K \right)^{\ell} \lVert J_1^\cdot(K^{(0)}) - J_1^\cdot(K^*) \rVert_{L^2(I)},
    \end{equation*}
    where $J^\alpha_1$, $\alpha \in I$, is defined in \eqref{eq:J1_cost}.
\end{theorem}
The precise  expressions of  the constants $C_1^K, C_2^K$ and $C_3^K$ are given in \eqref{eq:constant_convergence_K}.  The constants depend on $\underline{\lambda}^R, \overline{\lambda}^R$ and on the uniform bounds of the eigenvalues of the covariance matrices of the state processes controlled by $(K^{(\ell)})_{\ell\in \mathbb N_0}$.

The proof of Theorem \ref{theorem:K_CONVERGENCE} adapts the analysis for the single-agent LQ control problem in \cite{giegrich_convergence_2022} to our setting.
There are two key differences between our framework and \cite{giegrich_convergence_2022}. First, \eqref{eq:J1_cost} involves optimization over the covariance matrix $\vartheta^{K,\alpha}$, whereas \cite{giegrich_convergence_2022} optimizes the state second moment matrix $\Sigma^{\theta, \alpha} = \mathbb{E}[X^{\theta, \alpha}_t (X^{\theta, \alpha}_t)^\top]$. Second, and more importantly, the cost functional $J_1^\alpha$ is player-dependent because the state dynamics' covariance matrix depends on $\alpha$ through the initial state $\xi^\alpha$. Consequently, we must ensure that the landscape properties of the cost $K \mapsto J_1^\alpha(K)$ (such as gradient dominance and smoothness) hold uniformly across policy iterates and the player space, which is essential to obtain a player-independent linear convergence rate.  

Our analysis with time-dependent policies is also more intricate than that of the mean field game with stationary policies studied in \cite{wang2021global}. This complexity arises not only from player heterogeneity but also because $K \mapsto J_1^\alpha(K)$ satisfies a \emph{non-uniform almost smoothness condition,
depending on the current iterate $K$ through the state covariance matrix
(Lemma \ref{lemma:P1_ITERATION_CD})}.
We overcome this difficulty by proving
$(K^{(\ell)})_{\ell\in \mathbb N_0}$
is uniformly bounded 
(Proposition \ref{prop:BOUND_K})
and hence
the loss $  J_1^\alpha $ has   uniform landscape properties along the policy  updates.

\paragraph{Convergence of the  intercept parameter.} 

We proceed to analyze the convergence 
of \eqref{eq:P2_graddescent}
to the minimizer of $J_2^\alpha$.
For this analysis, 
we fix 
a slope parameter $K^{(\ell)}\in L^\infty([0,T], \mathbb R^{k\times d})$
obtained from \eqref{eq:P1_graddescent} 
with an initial guess
$K^{(0)}\in L^\infty([0,T], \mathbb R^{k\times d})$
and a stepsize $\eta_K\in (0,C^K_1)$,
after $\ell$ iterations for some $\ell\in \mathbb N_0$.

The following two propositions establish a novel uniform landscape property of the cost functional $  J^\alpha_2 $ with respect to $G^\alpha$. Specifically, 
for each  player index
$\alpha\in I$ and   graphon aggregate 
$Z^\alpha \in \mathcal{C}([0,T], \mathbb R^d)$, 
we prove that the map
\begin{equation}
\label{eq:J_2_ell_alpha}
L^2([0,T], \mathbb{R}^k) \ni G \mapsto 
J_2^{\ell,\alpha}(G)\coloneqq J_2^\alpha(K^{(\ell)}, G, Z^\alpha)\in \mathbb R  
\end{equation}  
is strongly convex  and  Lipschitz smooth, \emph{uniformly with respect to the slope parameter \( K^{(\ell)} \), the graphon aggregate \( Z^\alpha \), and  
 the player index \( \alpha \in I \).}

\begin{proposition}
\label{proposition:J2_strongconvexity}
    There exists  $m > 0$ such that for all 
    $\alpha \in I$,
    $\ell\in \mathbb N_0$, $Z^\alpha\in \mathcal{C}([0,T],\mathbb R^d)$ and 
    $G, G' \in L^2([0,T], \mathbb{R}^k)$,
      \begin{equation*}
        J_2^{\ell, \alpha}(G') - J_2^{\ell, \alpha}(G) \ge \langle G' - G, \nabla_G J_2^{\ell, \alpha}(G) \rangle_{L^2} + \frac{m}{2} \lVert G'  - G  \rVert_{L^2([0,T])}^2,
    \end{equation*}
    with the gradient 
    $\nabla_G J_2^{\ell, \alpha}(G) 
    =\nabla_G J_2^\alpha(K^{(\ell)}, G, Z^\alpha) $   
      defined as in \eqref{eq:grad_J2}.
      The value of  $m $ is given in Lemma \ref{lemma:J22_stronglyconvex}.
    
\end{proposition}

\begin{proposition}
    \label{proposition:LSmooth}
    There exists   $L > 0$ such that
    for all $\alpha \in I$, $\ell\in \mathbb N_0$, $Z^\alpha\in \mathcal{C}([0,T],\mathbb R^d)$ and 
    $G, G' \in L^2([0,T], \mathbb{R}^k)$,
    \begin{equation*}
        J_2^{\ell, \alpha}(G') - J_2^{\ell, \alpha}(G) \le \langle G' - G, \nabla_G J_2^{\ell, \alpha}(G) \rangle_{L^2} + \frac{L}{2} \lVert G'  - G  \rVert_{L^2([0,T])}^2.
    \end{equation*}  
\end{proposition}

The proof of Proposition \ref{proposition:J2_strongconvexity}
requires 
sophisticated functional analytic techniques. Due to the linearity of the map  
$L^2([0,T],\mathbb R^k)\ni  G\mapsto 
\mu^{(K^\ell, G),\alpha}\in \mathcal C([0,T],\mathbb R^d)$
(cf.~\eqref{eq:mean}), 
  the key step of the proof is to show that
\begin{align*}
 \begin{split}
 & G\mapsto   \int_0^T   \left(K^{(\ell)}_t \mu_t^{(K^{(\ell)}, G),\alpha} + G_t\right)^\top R  \left(K^{(\ell)}_t \mu_t^{(K^{(\ell)}, G),\alpha} + G_t\right)   \, dt
\end{split}
\end{align*}
is  strongly convex, uniformly with respect to $\ell$ and $\alpha$ (see Lemma \ref{lemma:J22_stronglyconvex}). We prove this  by viewing  
$G\mapsto 
\mu^{(K^\ell, G),\alpha}$ as a Volterra operator, and establishing a uniform bound of its resolvent
operator (Lemma \ref{lemma:resolvent}). 
  A key ingredient of this argument is the uniform boundedness of the iterates $(K^{(\ell)})_{\ell\in \mathbb N_0}$ (Proposition \ref{prop:BOUND_K}). 

 Using Propositions \ref{proposition:J2_strongconvexity}
 and 
 \ref{proposition:LSmooth}, 
 the following theorem establishes  the linear convergence of \eqref{eq:P2_graddescent}
 to the minimizer of 
$ G \mapsto J_2^\alpha(K^{(\ell)}, G, Z^\alpha) $.

\begin{theorem}
    \label{theorem:G_CONVERGENCE}
Let $m, L>0$ be the constants 
in 
Propositions \ref{proposition:J2_strongconvexity}
 and 
 \ref{proposition:LSmooth}, respectively. 
Let 
$G^{(0)}=(G^{(0),\alpha})_{\alpha \in I} \in  L^2([0,T], L^2(I,\mathbb R^{k}))$, 
$\eta_G\in (0,2/L)$,
 and 
 for each $i\in \mathbb N_0$,
 let 
 $G^{(i+1)}=(G^{(i+1),\alpha})_{\alpha \in I}\in   L^2([0,T], L^2(I,\mathbb R^{k}))$ be such that 
 $G^{(i+1), \alpha} = G^{(i), \alpha} - \eta_G  \nabla_G J_2^{\ell,\alpha}( G^{(i), \alpha})$ 
 with
 $J_2^{\ell,\alpha}$
 defined by \eqref{eq:J_2_ell_alpha}.
 Then for all  
 $\alpha \in I$
 and $i\in \mathbb N_0$,
    \begin{equation*}
        \lVert G^{(i), \alpha} - G^{*,\ell, \alpha} \rVert_{L^2([0,T])}^2 \le (1 - \eta_G m)^i \lVert G^{(0), \alpha} - G^{*,\ell, \alpha} \rVert_{L^2([0,T])}^2.
    \end{equation*}
where $G^{*, \ell, \alpha}=\argmin_{G\in L^2([0,T],\mathbb R^{k})}J_2^{\ell, \alpha}(G)$.
Consequently, for all $i\in \mathbb N_0$,
   \begin{equation*}
        \lVert G^{(i)} - G^{*,\ell} \lVert_{L_B^2}^2 \le (1 - \eta_G m)^i \lVert G^{(0)} - G^{*,\ell} \lVert_{L_B^2}^2,
        \quad \textnormal{with $G^{*,\ell}=(G^{*,\ell,\alpha})_{\alpha \in I}$.}
    \end{equation*}
\end{theorem}

To the best of our knowledge,
this is the first result on the linear convergence of PG methods with \emph{affine policies}
for finite-horizon drift-controlled LQ problems. 
The constants 
  $L$ and $m$ in general
  depend exponentially on the  coefficients $A, B$, the time horizon $T$ and the uniform bound for the slope parameter iterates $(K^{(\ell)})_{\ell\in \mathbb N_0}$. 
Such an exponential dependence arises from using Gronwall's inequality
to estimate the state moments, and using   the Volterra resolvent bound in  
\cite[Theorem 3.1]{dechevski_sharp_1994}
for the cost landscape analysis.

\paragraph{Convergence  to   the NE policy.} 

Leveraging  the error bounds of \eqref{eq:P1_graddescent} and \eqref{eq:P2_graddescent} for the best response policy
established in 
 Theorems \ref{theorem:K_CONVERGENCE} and  \ref{theorem:G_CONVERGENCE}, 
we now   analyze the convergence of Algorithm \ref{algo} to the NE policy of the LQ-GMFG. 

To facilitate the convergence analysis, 
we impose the following smallness 
condition on the model coefficients. 

\begin{assumption}
\label{assumption:convergence}

Let   
$C_0^K>0$ be such that 
$\lVert K^{(\ell)} \rVert_{L^\infty} \le C_0^K$ for all $\ell\in \mathbb N_0$, where 
$ (K^{(\ell)})_{\ell\in \mathbb N_0}$ is   defined by \eqref{eq:P1_graddescent} with
$K^{(0)} \in L^\infty([0,T], \mathbb{R}^{k \times d})$ and 
$\eta_K>0$ (cf.~Proposition \ref{prop:BOUND_K}).
    It holds that  
    \begin{equation}
    \label{eq:M1}
        M_1 \coloneqq T(\lVert A \rVert_{L^\infty} + \lVert B \rVert_{L^\infty} C_0^K + \lVert \Bar{A} \rVert_{L^\infty} \lVert W \rVert_{L^2(I^2)}) < 1,
    \end{equation}
and that  
    \begin{equation*}
        \begin{split}
            M_2 &\coloneqq \frac{T \lVert B \rVert_{L^\infty}^2 \left( \lvert \Bar{Q} \rvert \lvert \Bar{H} \rvert + T \left( \lVert P^* \rVert_{L^\infty} \lVert \Bar{A} \rVert_{L^\infty} + \lVert Q \rVert_{L^\infty} \lVert H \rVert_{L^\infty} \right) \right) \lVert W \rVert_{L^2(I^2)}}{\underline{\lambda}^R (1 - T (\lVert A \rVert_{L^\infty} + \lVert B \rVert_{L^\infty} \lVert K^* \rVert_{L^\infty} + \lVert \Bar{A} \rVert_{L^\infty} \lVert W \rVert_{L^2(I^2)}))^2}<1,
        \end{split}
    \end{equation*}
    where
   $P^*\in \mathcal C([0,T],\mathbb S^d)$ is the solution to  \eqref{eq:PI_riccati}
and 
$K^*\in L^\infty([0,T], \mathbb R^{k\times d})$
is defined by
\eqref{eq:systemcontrol}.

\end{assumption}

\begin{remark}
    Assumption \ref{assumption:convergence} is analogue to the contraction assumption imposed in   \cite[Lemma 4]{gao2021lqggraphonmeanfield}
    to ensure the existence and uniqueness of NEs for LQ-GMFGs.
    Indeed, by  \cite[Lemma 4]{gao2021lqggraphonmeanfield}, a sufficient condition for the well-posedness of  NEs is  
 \begin{equation}  
\label{eq:contraction_cond}
       \begin{split}
           &\bar{M}_1  +  \bar{M}_2 < 1,
       \end{split}
   \end{equation}
   where
   \begin{equation*}
        \begin{split}
            &\bar{M_1} \coloneqq T (\lVert A \rVert_{L^\infty} + \lVert B \rVert_{L^\infty} \lVert K^* \rVert_{L^\infty}+ \lVert \Bar{A} \rVert_{L^\infty} \lVert W \rVert_{L^2(I^2)}), \\
            &\bar{M_2} \coloneqq \frac{T \lVert B \rVert_{L^\infty}^2 \left( \lvert \Bar{Q} \rvert \lvert \Bar{H} \rvert + T \left( \lVert P^* \rVert_{L^\infty} \lVert \Bar{A} \rVert_{L^\infty} + \lVert Q \rVert_{L^\infty} \lVert H \rVert_{L^\infty} \right) \right) \lVert W \rVert_{L^2(I^2)}}{\underline{\lambda}^R (1 - T (\lVert A \rVert_{L^\infty} + \lVert B \rVert_{L^\infty} \lVert K^* \rVert_{L^\infty}))}.
        \end{split}
   \end{equation*}
This ensures that the mapping from a given graphon aggregate to the graphon aggregate induced by all players' best response policies is a contraction.
  Assumption \ref{assumption:convergence}  is slightly stronger than \eqref{eq:contraction_cond}, because, in each outer iteration, Algorithm \ref{algo} updates the graphon aggregate using approximate slope parameters generated by the PG update \eqref{eq:P1_graddescent}, rather than the exact best-response policy (see Line 13 of Algorithm \ref{algo}). Note that similar contraction assumptions have been adopted to  analyze NEs of continuous-time GMFGs (e.g., \cite[(H8)]{Caines_2021}, \cite[Equation (36)]{gao2021lqggraphonmeanfield}, \cite[Equation (21)]{Gao_2021}, \cite[Assumption (A2)]{foguen-tchuendom_infinite_2024}) and to  design  learning algorithms for discrete-time GMFGs (e.g., \cite[Proposition 2]{cui2022learninggraphonmeanfield}, \cite[Assumption 2]{huo2023reinforcementlearningsbmgraphon}, \cite[Assumption 4.1]{zhang2023learningregularizedgraphonmeanfield}, \cite[Assumption 5.1]{zhou2024graphonmeanfieldgames}).
\end{remark}

The following theorem presents  the global convergence   of Algorithm \ref{algo} to the NE of LQ-GMFG.
Recall that 
$(K^*, G^*) \in L^\infty([0,T], \mathbb{R}^{k \times d}) \times L^2([0,T], L^2(I, \mathbb{R}^k))$ 
  are the equilibrium policy parameters 
  defined in \eqref{eq:systemcontrol},
  and 
    $\mu^* \in \mathcal{C}([0,T], L^2(I, \mathbb{R}^d))$ is the population state mean of the associated      state processes (see  \eqref{eq:equil_state}). 

{ 

\begin{theorem}
    \label{theorem:CONVERGENCE_FINAL}
    Suppose Assumptions \ref{assumption:RandCOV} and \ref{assumption:convergence} hold. Let $\mu^{(0)} \in \mathcal{C}([0,T], L^2(I, \mathbb{R}^d))$ be the  initial population state mean and let $(K^{(0)}, G^{(0)}) \in L^\infty([0,T], \mathbb{R}^{k \times d}) \times L^2([0,T], L^2(I, \mathbb{R}^k))$ be the initial policy parameters.

    Let $\varepsilon \in (0,1)$. Choose the number of outer iterations  $N \in \mathbb{N}$ in Algorithm \ref{algo} as 
    \begin{equation*}
        N \ge -\frac{1}{\log{(M_2)}} \log{\left( \frac{2 \lVert \mu^{(0)} - \mu^* \lVert_{\mathcal{C}}}{\varepsilon} \right)} 
    \end{equation*}
    with  $M_2 \in (0,1)$  given in Assumption \ref{assumption:convergence},
    the stepsizes $\eta_K$ and $\eta_G$ for the gradient descent updates in Algorithm \ref{algo} as   in Theorems \ref{theorem:K_CONVERGENCE} and \ref{theorem:G_CONVERGENCE}, respectively,
    and 
$(\omega_n)_{n=0}^{ N-1}
\subset (0,\infty)$ such that 
    \begin{equation}
    \label{eq:sum_condition_omega}
        \sum_{n = 0}^{N-1} M_2^{N-1-n} \omega_n \le \frac{\varepsilon}{2}.
    \end{equation}
    Consider 
the policy iterates 
$(K^{(n)}, G^{(n)})_{n=0}^{N}$ generated by 
Algorithm \ref{algo}
as follows: 
for each $n=0, \ldots, N-1 $, given 
 $K^{(n)}$, $ G^{(n)}$ and 
 $\mu^{(n)}$,
set 
$L_n^K,  L_n^G\in \mathbb N $, 
depending polynomially on $\log{(1/\omega_n)}$, 
to 
obtain 
$(K^{(n+1)}, G^{(n+1)})$
(cf.~Steps 5-12 of Algorithm \ref{algo}),
and  obtain 
$\mu^{(n+1)} $
from 
the mean field oracle $\Hat{\Phi}$ 
(cf.~Step 13 of
Algorithm \ref{algo}) such that
    \begin{equation}
    \label{eq:mu_estimator_assumption}
        \lVert \mu^{(n+1)} - \mu^{(n+1),*} \rVert_{\mathcal{C}} \le \frac{\omega_n}{3},
    \end{equation}
    where $\mu^{(n+1),*}$ is the  
    population state mean
    associated with the policy $(K^{(n+1)}, G^{(n+1)})$ (cf.~\eqref{eq:mean_policy}).
    
    Then it holds that 
\begin{equation*}
        \lVert K^{(N)} - K^* \rVert_{L^2([0,T])} \le \varepsilon, \quad \lVert G^{(N)} - G^{*} \rVert_{L_B^2} \le \varepsilon M_G,
        \quad 
         \lVert \mu^{(N)} - \mu^* \rVert_{\mathcal{C}} \le \varepsilon,
    \end{equation*}
 with a constant $M_G \geq 1$, defined in \eqref{eq:M_G}, which is independent of $\varepsilon$ and decreases as $\lVert W \rVert_{L^2(I^2)}$ decreases.
 
\end{theorem}

By Theorem \ref{theorem:CONVERGENCE_FINAL}, the number of gradient descent iterations required in Algorithm \ref{algo} to approximate the NE policy up to an error of 
$\varepsilon$  is of order
$\sum_{n = 0}^{\lceil M \log{(1/\varepsilon)} \rceil} \log{ (1/\omega_n ) } $,  
with
$(\omega_n)_{n=0}^{ N-1}
$
satisfying 
\eqref{eq:sum_condition_omega},
and  $M > 0$ 
being a   constant independent of $\varepsilon$. 
This aligns with the   complexity of the   algorithm proposed in \cite{wang2021global} for infinite-horizon MFGs.
Setting a constant sequence    $(\omega_n)_{n=0}^{ N-1}
$ 
yields an overall iteration complexity of $(\log(1/\varepsilon))^2$ for Algorithm \ref{algo}.

The lower bounds for the number of gradient iterations  $L_n^K$ and $L_n^G$ are provided in Section \ref{subsection:convergence_analsis_to_NE}. These bounds depend polynomially on $\log(1/\omega_n)$,  as well as on the final cost associated with the previous policy iterate, and the constants  in Theorems \ref{theorem:K_CONVERGENCE} and \ref{theorem:G_CONVERGENCE}.  
Condition \eqref{eq:mu_estimator_assumption} on the mean field oracle can be satisfied either by solving the population mean ODE system when the model parameters are known, or by estimating the population mean using sufficiently many state trajectories in a model-free setting.

The proof of Theorem 
\ref{theorem:CONVERGENCE_FINAL}
relies on a careful analysis of error propagation in both the population state mean and the policy parameters through their sequential updates. The analysis requires controlling four sources of error:
(i) the optimization error in the slope parameter;
(ii) the optimization error in the intercept parameter, given a fixed slope and graphon aggregate;
(iii) the mean field estimation error under fixed policy parameters; and
(iv) the convergence of the population state mean to the NE.
Each of these error terms is estimated by introducing suitable integral operators and analyzing their stability with respect to their inputs.

}

\section{Numerical experiments}
\label{section:numerics}

{
This section examines the convergence and robustness of Algorithm \ref{algo} in  LQ-GMFGs  under various graphon interaction structures.  
Algorithm~\ref{algo} is implemented by approximating continuous-time policies with piecewise constant ones,
and    estimating policy gradients using either pathwise differentiation or zeroth-order optimization techniques.   
Our numerical results   demonstrate that:  
\begin{itemize}
    \item Algorithm \ref{algo} achieves    linear convergence rates for both policy slope and intercept  parameters. 
    The convergence rate of the slope parameter is independent of the underlying graphon, whereas the convergence of the intercept parameter depends on the graphon's $L^2$-norm, 
in accordance with  
    Theorem \ref{theorem:CONVERGENCE_FINAL}. 
    \item 
   The discrete-time analogue of Algorithm \ref{algo}, with gradients scaled according to the policy discretization mesh size, achieves robustness across different discretization mesh sizes and noise magnitudes   in the state process.
    \item The model-free extension of Algorithm \ref{algo}, where policy gradients are estimated using suitable zeroth-order optimization techniques based on sampled cost functions, achieves similar linear convergence rates.

\end{itemize}
}

\subsection{Experiment setup}

In the numeral  experiments, 
we fix  
the  following model parameters
in \eqref{eq:state_dynamics}-\eqref{eq:quadratic_cost_initial}: 
$T=1$, 
$d= k=1$, 
$A \equiv -0.25$,  $B \equiv 0.5$, 
$\Bar{A} \equiv  0.25$,
$Q\equiv 0.25$,
$R \equiv 0.5$,
$ H \equiv 1$,
$\Bar{Q} = 0.05$, 
$\Bar{H} = 1$,
and 
$  \xi^\alpha \sim \mathcal{N}(0.5, 0.01)$ for all $\alpha\in I$.
The noise coefficient $D$ is set to $0.25$ for most experiments, except in Figure \ref{fig:num_experim_MSE_noise}, where 
 it is varied to examine its impact on the performance of Algorithm \ref{algo}.
We consider the following graphons:  the complete bipartite graphon $W_{\mathrm{bp}}$, the threshold graphon $W_{\mathrm{th}}$, the half graphon $W_{\mathrm{hf}}$ and the uniform-attachment graphon $W_{\mathrm{ua}}$.
These graphons are defined below and visualized in Figure    \ref{fig:graphons}.
 
\begin{alignat*}{2}
        &W_{\mathrm{bp}}(\alpha, \beta) \coloneqq 
        \begin{cases} 
            1, & \text{if } \min{(\alpha, \beta)} < 0.5 < \max{(\alpha, \beta),} \\ 
            0, & \text{otherwise,}
        \end{cases}  \quad && W_{\mathrm{th}}(\alpha, \beta) \coloneqq
        \begin{cases} 
            1, & \text{if } \alpha + \beta \le 1, \alpha \neq \beta, \\ 
            0, & \text{otherwise,}
        \end{cases}\\
        &W_{\mathrm{hf}}(\alpha, \beta) \coloneqq 
        \begin{cases} 
            1, & \text{if } \lvert \alpha - \beta \rvert \ge 0.5, \\ 
            0, & \text{otherwise,}
        \end{cases}   
       && W_{\mathrm{ua}}(\alpha, \beta) \coloneqq 
        \begin{cases} 
            1 - \max{(\alpha, \beta)}, & \text{if } \alpha \neq \beta, \\ 
            0, & \text{otherwise. }
        \end{cases} 
\end{alignat*}

\begin{figure}[H]
    \centering
    \begin{subcaptionbox}{Complete bipartite graphon\label{fig:graphon_bp}}[.23\linewidth]
    {   \includegraphics[scale=0.23]
        {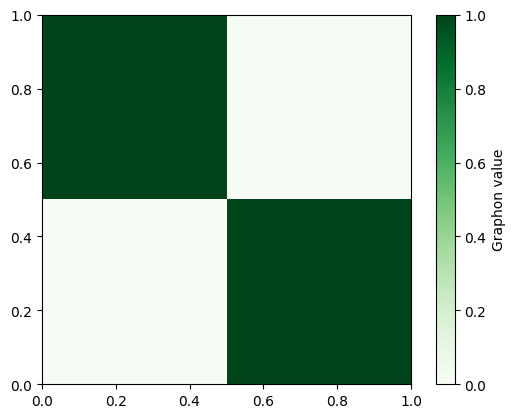}}
    \end{subcaptionbox}
    \begin{subcaptionbox}{Threshold graphon
        \label{fig:graphon_threshold}}[.23\linewidth]
        {\includegraphics[scale=0.23]
        {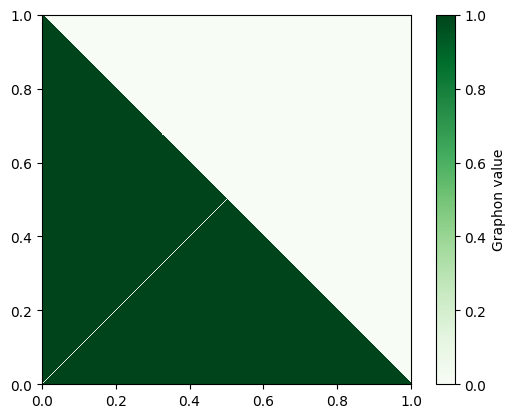}}
     \end{subcaptionbox}
    \begin{subcaptionbox}{Half graphon}[.23\linewidth]
        {\includegraphics[scale=0.23]
        {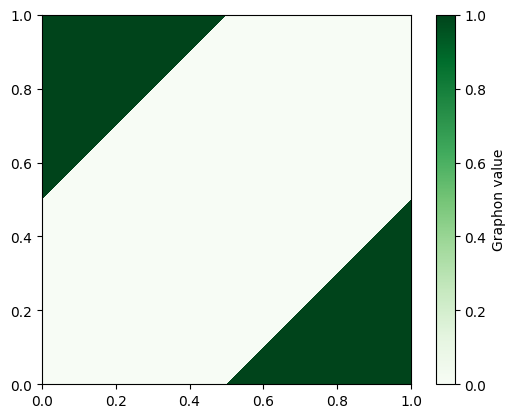}} 
    \end{subcaptionbox}
    \begin{subcaptionbox}{Uniform-attachment graphon\label{fig:ua_graphon}}[.23\linewidth]
     {\includegraphics[scale=0.23]
     {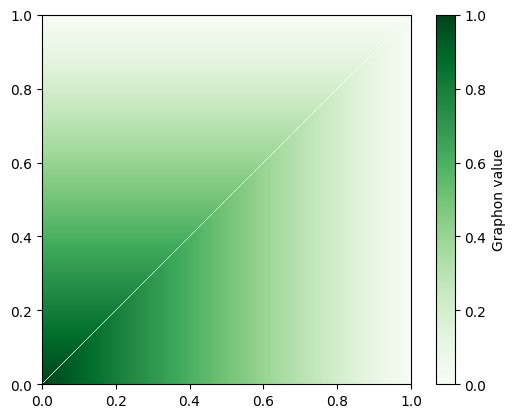}}
    \end{subcaptionbox}

    \caption{Graphons for numerical experiments}
    \label{fig:graphons}
\end{figure}

 Algorithm \ref{algo}
 is implemented by approximating 
 the continuous policy space \eqref{eq:policy_parameterization}
with the space of piecewise constant policies. 
 Specifically, 
we consider a uniform time mesh  
$(\tau_i)_{i=0}^{N_{\mathrm{policy}} + 1}\subset [0,T]$ 
with $\tau_i=i \Delta \tau$,
for some 
$\Delta \tau>0$ and $N_{\mathrm{policy}}\in \mathbb N$ such that $
(N_{\mathrm{policy}} + 1)\Delta \tau =T$,
and consider piecewise constant policy slope and intercept parameters on this grid (see e.g.,   \cite{giegrich_convergence_2022}).   
We further discretize the player space $I$ 
into  $\{0, 0.1, 0.2, \dots, 0.9, 1\}$,
with $N_{\mathrm{player}}=11$ players. This yields the following discretization of the policy space \eqref{eq:policy_parameterization}:
for all $t\in [0,T]$
and $j \in \{1,2 \dots, N_{\mathrm{player}} \}$,
\begin{equation*}
    K_t = \sum_{i = 0}^{N_{\mathrm{policy}}} \mathds{1}_{[\tau_{i}, \tau_{i+1})}(t) K_{\tau_i}, \quad \text{and} \quad G_t^{\alpha_j} = \sum_{i = 0}^{N_{\mathrm{policy}}} \mathds{1}_{[\tau_{i}, \tau_{i+1})}(t) G_{\tau_i}^{\alpha_j},
\end{equation*}
where
$\{K_{\tau_i},  G_{\tau_i}^{\alpha_j}\mid 
  0\le i\le   N_{\mathrm{policy}},
 1\le j \le N_{\mathrm{player}}
 \}
$ are the  policy   parameters    to be updated by the algorithm. 

We initialize Algorithm \ref{algo}
with the initial policy parameters $K^{(0)} \equiv -1$ and $G^{(0), \alpha_j} \equiv 1$  for all $1\le j \le N_{\mathrm{player}}$, 
the initial mean field $\mu^{(0)}\equiv 0$,
   the number of outer iterations $N=15$,  and 
 the number of gradient updates    
$L^K_n=L^G_n =10$ for all $n=0, \ldots, N $.
At each gradient update, the performance of the current policy iterate is evaluated by approximating  $J_1$ and $J_2$ defined in \eqref{eq:J1_cost} and \eqref{eq:J2_cost}
with a finer time mesh. More precisely, 
given the current policy parameters $(K,G)$ and graphon aggregate $Z$,
we consider the following  Euler-Maruyama approximation of  
   the state process \eqref{eq:state_dynamics}:  
for each $j=1,\ldots  N_{\mathrm{player}}$
and $i=0,\ldots, N_{\textrm{time}}-1$, set
\begin{equation}
\label{eq:state_euler}
    X_{t_{i+1}}^{\alpha_j} = X_{t_i}^{\alpha_j} + \left( (A + B K_{t_i}) X_{t_i}^{\alpha_j} + B G_{t_i}^{\alpha_j} + \Bar{A} Z_{t_i}^{\alpha_j} \right) \Delta t + D \Delta B_{t_i}^{\alpha_j},
    \quad 
    X_{t_0}^{\alpha_j} = \xi^\alpha,
\end{equation}
where 
$\Delta t=1/N_{\textrm{time}}$
and $(\Delta B_{t_i}^{\alpha_j})_{i = 0,1,\dots,N_{\textrm{time}}-1}^{j = 1,2,\dots,N_{\mathrm{player}}}$ are independent zero-mean Gaussian random variables with variance $\Delta t$.
Using  $N_{\mathrm{sample}}$ sampled 
trajectories of 
\eqref{eq:state_euler}, 
we estimate 
the player-dependent state mean $\Hat{\mu}^{\alpha_j}$ for all $j \in \{1,  \ldots, N_{\mathrm{player}} \}$ and the player-independent state variance $\Hat{\vartheta}$ through empirical approximations. 
The cost functionals $J_1$ and $J_2$ are then approximated by
\begin{equation}
\label{eq:simulated_cost}
    \begin{split}
        &\Hat{J}_1(K) \coloneqq  \Delta t \sum_{i = 0}^{N_{\textrm{time}}-1} \left( Q + R K_{t_i}^2 \right) \Hat{\vartheta}_{t_{i}} + \Bar{Q} \Hat{\vartheta}_{t_{N_{\textrm{time}}}}, \\
        &\Hat{J}_2^{\alpha_j}(G^{\alpha_j})\coloneqq \Hat{J}_2^{\alpha_j}(K, G^{\alpha_j}, Z^{\alpha_j}) = \Delta t \sum_{i = 0}^{N_{\textrm{time}}-1} \left( Q \left(\Hat{\mu}_{t_i}^{\alpha_j} - H Z_{t_i}^{\alpha_j} \right)^2 + R \left( K_{t_i} \Hat{\mu}_{t_i}^{\alpha_j} + G_{t_i}^{\alpha_j} \right)^2 \right) \\
        &\quad\quad\quad\quad\quad\quad\quad\quad\quad\quad\quad\quad\quad\quad+ \Bar{Q} \left( \Hat{\mu}_{t_{N_{\textrm{time}}}}^{\alpha_j} - \Bar{H} Z_{t_{N_{\textrm{time}}}}^{\alpha_j} \right)^2.
    \end{split}
\end{equation}
In the sequel, we set $N_{\mathrm{sample}} = 10^5$ and $N_{\mathrm{time}} = 120$ such that the time step $\Delta t=1/120$ is less than or equal to the finest policy discretization mesh size.
The approximate cost functions \eqref{eq:simulated_cost} will then be used for estimating   policy gradients, 
either  
using pathwise differentiation (Section \ref{sec:numerical_automatic}) or using zeroth-order optimization techniques 
(Section \ref{sec:numerical_model_free}).

Finally, to implement the mean field oracle $\hat \Phi$ in  Step  13 of
Algorithm \ref{algo}, 
we first solve the ODE \eqref{eq:mean_policy}
with  the latest policy parameters for 
the state mean $(\mu_{t_i}^{\alpha_j})_{i=0,\ldots, N_{\rm time}}^{j=1,\ldots, N_{\rm player}}$, and then define the updated graphon aggregate  by  
\begin{equation*}
    Z_{t_i}^{\alpha_j} = \frac{1}{N_{\mathrm{player}}} \sum_{k \neq j} W(\alpha_j, \alpha_k) \mu_{ t_i}^{\alpha_k},
    \quad \forall i=0,\ldots, N_{\rm time}, j=1,\ldots, N_{\rm player},
\end{equation*}
where $W$ is the pre-specified graphon taken from $\{ W_{\mathrm{bp}},   W_{\mathrm{th}},  W_{\mathrm{hf}},  W_{\mathrm{ua}}\}$.

\subsection{Gradient estimation using pathwise   differentiation} 
\label{sec:numerical_automatic}

In this section,
we compute the policy gradients by applying pathwise differentiation to the associated cost functional with respect to the policy parameters.
Specifically, for each 
$\ell \in \mathbb{N}_0$,
given 
the policy parameters
$(K^{(\ell)}, G^{(\ell)})$,
we compute the associated costs $\Hat{J}_1 $ and $\Hat{J}^{\alpha_j}_2  $
as in   \eqref{eq:simulated_cost},
and update the policy parameters by: 
for all $i=0,\ldots, N_{\mathrm{policy}}$ and 
$j=1,\ldots, N_{\mathrm{player}}$, 
\begin{equation}
    \label{eq:updates_numexp}
    K_{\tau_i}^{(\ell + 1)} = K_{\tau_i}^{(\ell)} - \frac{\eta_K}{\Delta \tau \, \Hat{\vartheta}_{\tau_{i}}} \widehat{\nabla_{K_{\tau_i}}} J_1(K^{(\ell)}), \quad   G_{\tau_i}^{\alpha_j, (\ell + 1)} = G_{\tau_i}^{\alpha_j, (\ell)} - \frac{\eta_G}{\Delta \tau} \widehat{\nabla_{G_{\tau_i}}} J_2^{\alpha_j}(G^{\alpha_j, (\ell)}),
\end{equation}
where 
$\Delta \tau>0$ is the policy discretization mesh size, 
$\eta_K = \eta_G = 0.1$, and 
$\widehat{\nabla_{K_{\tau_i}}} J_1(K^{(\ell)})$
and $\widehat{\nabla_{G_{\tau_i}}} J_2^{\alpha_j}(G^{\alpha_j, (\ell)})$ are 
computed using the automatic differentiation in PyTorch. 

 \begin{remark}[\textbf{Scaling gradients with
 discretization
 timescales}]
 \label{rmk:scaling}
    Note that, as in the single-agent setting studied in \cite{giegrich_convergence_2022}, the discrete-time gradient is scaled by the policy discretization mesh size $\Delta \tau$ in the update \eqref{eq:updates_numexp}.
This scaling ensures that, as the policy discretization mesh size tends to zero, the update \eqref{eq:updates_numexp} with piecewise constant policies remains consistent with the continuous-time policy update analyzed in Theorem \ref{theorem:CONVERGENCE_FINAL}. As a result, the update \eqref{eq:updates_numexp}  achieves a \emph{mesh-independent} convergence rate,  
comparable to that of 
Algorithm \ref{algo} with 
 continuous-time policies, as demonstrated in Figure \ref{fig:num_experim_MSE_mesh}; see \cite[Section 2.4]{giegrich_convergence_2022} for a rigorous justification of this mesh-independent convergence in the single-agent setting. 
 \end{remark}

\paragraph{Convergence of the policy parameters with different graphons.}

To examine the convergence 
of Algorithm \ref{algo}, 
we evaluate  the error between the approximate  policies  and the NE policy \eqref{eq:systemcontrol},
whose coefficients $(K^*,G^*)$  
  are  obtained 
 by numerically solving the forward-backward system \eqref{eq:semiexplicit_ODEs} using an explicit Runge-Kutta method
  on the finest time grid with $\Delta t=1/120$,
  and a  fine  discretization of the player space with $N_{\rm player} =161$. 
 In particular, 
 for  any  policy parameters
 $(K_{t_i})_{i=0}^{120}$
 and  $(G^{\alpha_j}_{t_i})_{  i=0,\ldots, 120, j=1,\ldots, 11}$
 generated by Algorithm \ref{algo}, we define their     
root mean squared errors (RMSEs)   by 
  \begin{equation*}
  \textrm{RMSE}(K)=  \left(\frac{1}{121} \sum_{i = 0}^{120} ({K}_{t_i} - K_{t_i}^*)^2\right)^{1/2}, \quad  
  \textrm{RMSE}(G)=
  \left(\frac{1}{121 \cdot 161} \sum_{j = 1}^{161} \sum_{i = 0}^{120} (\Hat{G}_{t_i}^{  \beta_j} - G_{t_i}^{*, \beta_j})^2\right)^{1/2},\end{equation*}
where $\hat G$ is a piecewise constant extension of $G$ over the player space $I=[0,1]$. 

Figure \ref{fig:num_experim_MSE} presents the RMSEs of the policy parameters throughout the execution of Algorithm~\ref{algo},
for different graphon interactions (with the  
 policy discretization mesh size fixed at $\Delta \tau  = 1 /30$).
Recall that there are 10 outer iterations, each comprising 15 policy gradient updates, resulting in a total of 150 gradient iterations.
One can clearly observe that both policy parameters converge linearly to the NE policy, even when only 11 players are used to approximate the player space.

Moreover,
Figure \ref{fig:mse_K} shows that the convergence of the slope parameter is nearly identical across all graphons used in the experiment. This confirms the benefit of normalizing the gradient direction using the covariance matrix, which results in a gradient update that is independent of the underlying graphon (see Remark 
\ref{rmk:gradient_normalization}).

For the intercept parameter, Figure \ref{fig:mse_G} shows that the error initially decays at a common linear rate, before reaching a plateau that depends on the underlying graphon.
The different plateaus are likely due to the interplay between the connectivity and denseness of the graphons, as well as discretization errors in the player space.
In particular, if one measures the denseness 
of a graphon using its 
$L^2$-norm, then Figure~\ref{fig:num_experim_MSE} shows that a larger $L^2$-norm corresponds to a higher RMSE in the final intercept parameter
(the $L^2$-norms of the graphons are 
$\lVert W_{\mathrm{th}} \rVert_{L^2(I^2)} = 1/\sqrt{2}$,
 $\lVert W_{\mathrm{bp}} \rVert_{L^2(I^2)} = 1/\sqrt{2}$,
  $\lVert W_{\mathrm{hf}} \rVert_{L^2(I^2)} = 1/2$
  and $\lVert W_{\mathrm{ua}} \rVert_{L^2(I^2)} = 1/\sqrt{6}$). 
This observation is consistent with the theoretical result in Theorem~\ref{theorem:CONVERGENCE_FINAL}, where the constant $M_G$ in  the final error of the intercept parameter decreases as  $\lVert W  \rVert_{L^2(I^2)}$  becomes smaller.

\begin{figure}[H]
    \centering
    
    \begin{subcaptionbox}{RMSE of the slope parameter $K$\label{fig:mse_K}}[.48\linewidth]
     {\includegraphics[scale=0.48]{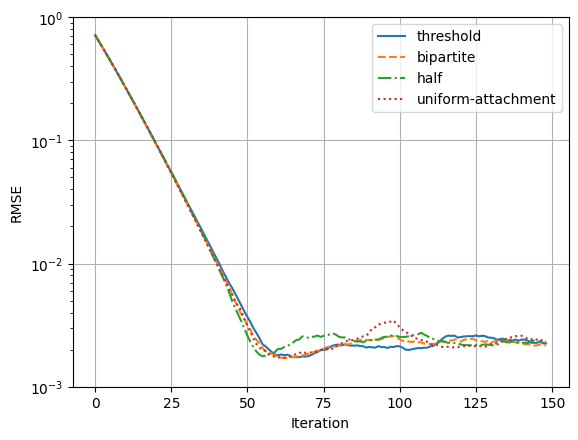}}
    \end{subcaptionbox}
    \begin{subcaptionbox}{RMSE of the intercept parameter $G$\label{fig:mse_G}}[.48\linewidth]
        {\includegraphics[scale=0.48]{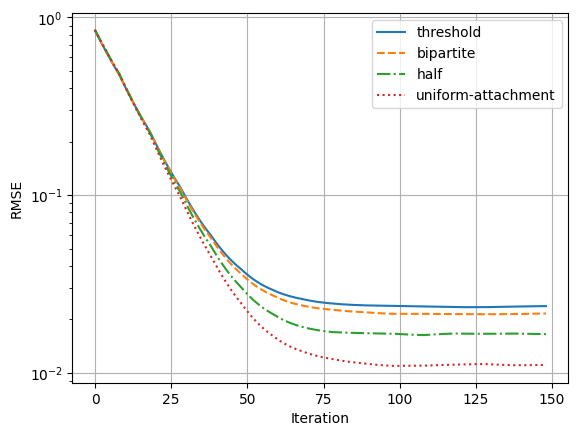}}
    \end{subcaptionbox}

    \caption{Convergence of Algorithm \ref{algo} with different graphons}
    \label{fig:num_experim_MSE}
\end{figure}

\paragraph{Robustness in policy discretization.} 

We then fix the graphon to be the  uniform-attachment graphon,
and examine the    performance of Algorithm \ref{algo} for different   policy discretization mesh sizes 
$\Delta \tau \in \{1/15, 1/30, 1/60, 1/120 \}$ 
(the mesh size for simulating the state process is fixed as $\Delta t = 1/120$).

Figure~\ref{fig:num_experim_MSE_mesh} shows  that Algorithm~\ref{algo} achieves a robust linear convergence   for both policy parameters, regardless of the policy discretization mesh size $\Delta \tau$,  until the optimization error is overtaken by the discretization error.
As highlighted in Remark~\ref{rmk:scaling}, scaling the policy gradient by the discretization mesh size is crucial for ensuring mesh-independent convergence, as standard PG methods typically exhibit performance degradation with increasing action frequency (see~\cite{park2021time}).
The final error of the intercept parameter is  stable across different   discretization mesh sizes, whereas   the slope parameter exhibits a stronger dependence on the discretization level.

\begin{figure}[H]
    \centering
    
    \begin{subcaptionbox}{RMSE of the slope parameter $K$\label{fig:mse_K_mesh}}[.48\linewidth]
     {\includegraphics[scale=0.48]{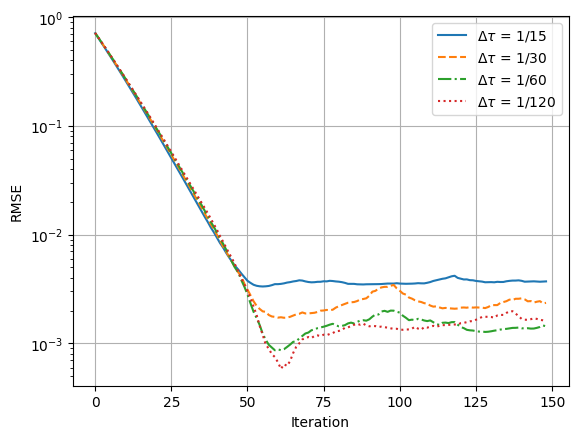}}
    \end{subcaptionbox}
    \begin{subcaptionbox}{RMSE of the intercept parameter $G$\label{fig:mse_G_mesh}}[.48\linewidth]
        {\includegraphics[scale=0.48]{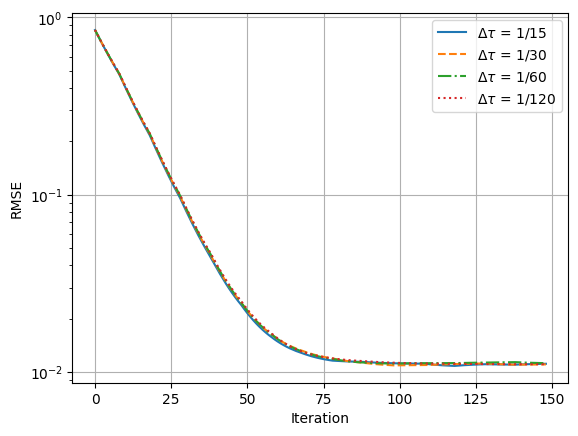}}
    \end{subcaptionbox}

    \caption{Convergence of Algorithm \ref{algo} with different policy discretization mesh sizes}
    \label{fig:num_experim_MSE_mesh}
\end{figure}

\paragraph{Robustness in the  magnitude  of 
system
noise.}

Figure \ref{fig:num_experim_MSE_noise} shows the convergence behavior of Algorithm~\ref{algo} by varying magnitudes of system noise  $D \in \{0.001, 0.01, 0.25, 1, 2\}$, with
the uniform-attachment graphon and 
the policy   mesh size fixed at $\Delta \tau = 1/30$. 
Both policy parameters 
achieve similar 
linear convergence rates for different noise levels.  
This contrasts with the PG algorithm for $N$-player games proposed in \cite{hambly_nplayer}, which only converges when the system noise is sufficiently large.
The seemingly higher  numerical instability at larger iteration numbers in Figures \ref{fig:mse_K_noise} and \ref{fig:mse_G_noise} arises from the small magnitude of the errors, which makes the fluctuations appear more pronounced on the logarithmic scale. This variance could be mitigated by increasing the number of samples.

\begin{figure}[!ht]
    \centering
    
    \begin{subcaptionbox}{RMSE of the slope parameter $K$ \label{fig:mse_K_noise}}[.48\linewidth]
     {\includegraphics[scale=0.48]{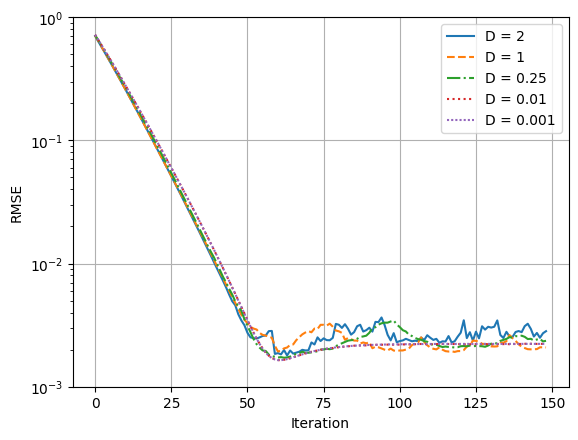}}
    \end{subcaptionbox}
    \begin{subcaptionbox}{RMSE of the intercept parameter $G$ \label{fig:mse_G_noise}}[.48\linewidth]
        {\includegraphics[scale=0.48]{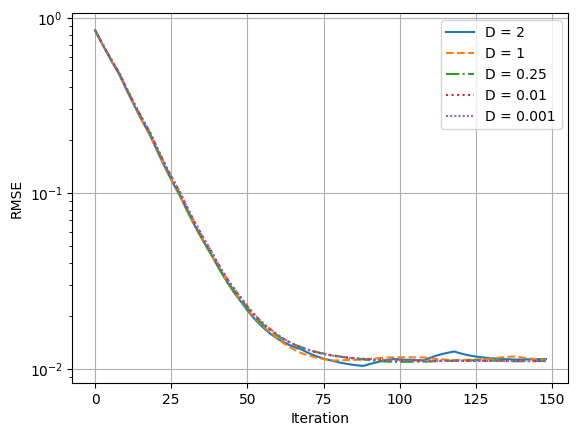}}
    \end{subcaptionbox}

    \caption{Convergence   of Algorithm \ref{algo} with different noise levels}
    \label{fig:num_experim_MSE_noise}
\end{figure}

\subsection{Gradient estimation using 
model-free
zeroth-order optimization}
\label{sec:numerical_model_free}

In this section, 
we implement Algorithm~\ref{algo} using a model-free approach to compute the policy gradients,
without relying on direct access to the model coefficients.
  Specifically, we assume access to an oracle that returns the cost functionals $\hat{J}_1$ and $\hat{J}_2$ for given policy parameters and graphon aggregates (as in \cite{carmona_linear-quadratic_2019}), and estimate the gradients using a standard  zeroth-order optimization method 
by   perturbing the policies with independent uniform noises
as in \cite{fazel_global_2018, hambly_nplayer, hambly_policy_2021} 
 (see also \cite{berahas_theoretical_2022}). The detailed implementation of this model-free gradient estimation procedure is provided in Appendix \ref{appendix:gradient_est}.
For simplicity the terminal cost  is set to $0$, i.e., $\Bar{Q} = 0$,
the policy   mesh size is set to  $\Delta \tau = 1/30$, and 
the uniform-attachment graphon is used. 

\begin{figure}[!ht]
    \centering
    
    \begin{subcaptionbox}{RMSE of the slope parameter $K$\label{fig:mse_K_unkown}}[.48\linewidth]
     {\includegraphics[scale=0.48]{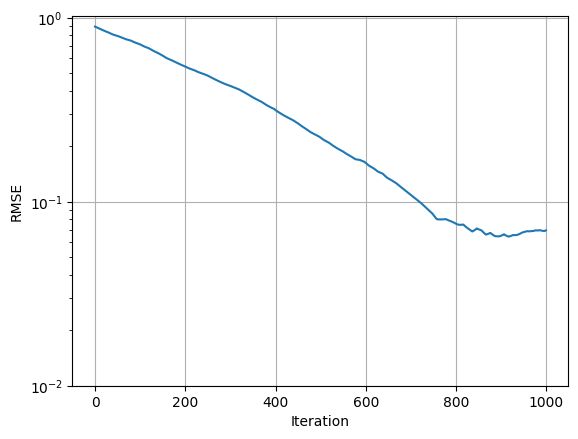}}
    \end{subcaptionbox}
    \begin{subcaptionbox}{RMSE of the intercept parameter $G$\label{fig:mse_G_unkown}}[.48\linewidth]
        {\includegraphics[scale=0.48]{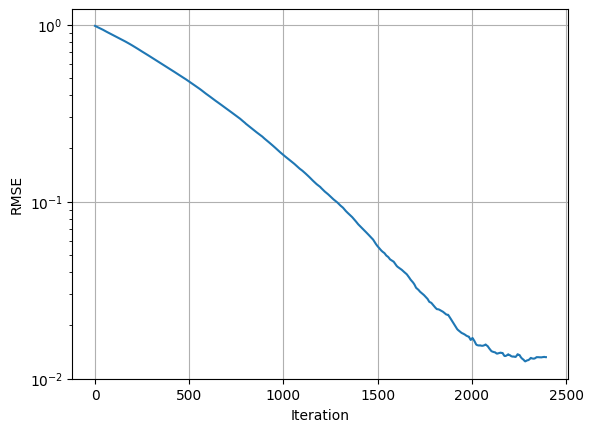}}
    \end{subcaptionbox}

    \caption{Convergence of Algorithm~\ref{algo} with zeroth-order gradient estimations}
    \label{fig:num_experim_MSE_UNKOWN}
\end{figure} 

We take the same piecewise constant  policy parameterization  as in Section \ref{sec:numerical_automatic}, and 
update the policies  according to  \eqref{eq:updates_numexp} with smaller  stepsizes $\eta_K = \eta_G = 0.01$ to enhance algorithm stability.
The number of  outer iterations in Algorithm \ref{algo} is set to   $240$ and 
at each outer iteration,
both policy parameters are updated using 10 inner gradient descent steps.
 We examine the algorithm convergence 
 by computing the RMSEs of the policy parameters   as in Section \ref{sec:numerical_automatic}.

Figure \ref{fig:num_experim_MSE_UNKOWN} shows  the linear convergence  
of   the model-free extension of Algorithm \ref{algo}  
for both policy parameters. Figure \ref{fig:mse_K_unkown} displays the convergence of the slope parameter over the first 1000 iterations, beyond which the error plateaus. In contrast, Figure \ref{fig:mse_G_unkown} shows that the intercept error continues to decrease until around iteration 2250, ultimately reaching an accuracy comparable to that achieved with automatic differentiation for gradient computation in Section \ref{sec:numerical_automatic}.

\section{Convergence analysis}
\label{section:convergence_analysis}

The convergence analysis is structured as follows: We first show the decomposition of the cost functional in Section \ref{subsection:decomp_proof}. Then we provide the convergence analysis for the slope parameter $K$ to the optimal slope $K^*$ in Section \ref{subsection:convergence_analysis_K}. Afterwards, we show in Section \ref{subsection:convergence_analysis_G} the convergence of the intercept parameter $G$ for fixed slope $K$ and graphon aggregate $Z$. Finally, in Section \ref{subsection:convergence_analsis_to_NE} we combine the results and carefully analyze the full convergence to the NE for the policy parameters and the population state mean.

\subsection{Decomposition of the cost functional}
\label{subsection:decomp_proof}

The next proposition characterizes the optimal value  and intercept parameter for fixed slope and graphon aggregate. More precisely, 
given $K\in L^2([0,T], \mathbb R^{k\times d})$
and $Z^\alpha \in \mathcal{C}([0,T], \mathbb R^d)$,
consider the following minimization problem: 
\begin{equation}
\label{eq:J2_controlproblem}
    \left\{
    \begin{aligned}
        &\inf\limits_{G^\alpha \in L^2([0,T], \mathbb{R}^{k})}{ J_2^\alpha(K, G^\alpha, Z^\alpha) } \\
        &\text{s.t.} \quad \frac{\partial \mu_t^\alpha}{\partial t} = (A + B K_t) \mu_t^\alpha + B G_t^\alpha + \Bar{A} Z_t^\alpha, \quad t \in [0,T]; \quad \mu_0^\alpha = \mathbb{E}[\xi^\alpha].
    \end{aligned}
    \right.
\end{equation}
\begin{proposition}
    \label{proposition:optimG_fixKZ}
    Let $\alpha \in I$ and $Z^\alpha \in \mathcal{C}([0,T], \mathbb{R}^d)$. Let $P^* \in \mathcal{C}([0,T], \mathbb{S}^d)$ satisfy \eqref{eq:PI_riccati}. Let $\Psi^{Z, \alpha} \in \mathcal{C}([0,T], \mathbb{R}^d)$ satisfy for all $t \in [0,T]$,
    \begin{equation}
        \label{eq:localoptimal_PSI_MU_J2}
        \begin{split}
            &\frac{\partial \Psi_t^{Z, \alpha}}{\partial t} + \left( A^\top - P_t^* B R^{-1} B^\top \right) \Psi_t^{Z, \alpha} + 2 (P_t^* \Bar{A} - Q H ) Z_t^\alpha = 0, \quad \Psi_T^{Z, \alpha} = 2 \Bar{Q} \Bar{H} Z_T^\alpha.
        \end{split}
    \end{equation}
    Then the optimal intercept parameter $G^{*, K, Z, \alpha}$ for the optimization problem \eqref{eq:J2_controlproblem} with fixed slope parameter $K$ and graphon aggregate $Z^\alpha$ is given by
    \begin{equation}
    \label{eq:OPTIMALG_P2}
        G_{t}^{*, K, Z, \alpha} = (K_t^* - K_t) \mu_t^{*, \alpha} - \frac{1}{2} R^{-1} B^\top \Psi_t^{Z, \alpha},\quad \forall t\in [0,T],
        \end{equation}
    where 
     $K^*$ is  defined in \eqref{eq:systemcontrol} and 
    $\mu^{*, \alpha}\in \mathcal C([0,T],\mathbb R^d)$ satisfies
    \begin{equation}
        \label{eq:optim_state_intercept}
        \begin{split}
            &\frac{\partial \mu_t^{*, \alpha}}{\partial t} = \left( A - B R^{-1} B^\top P_t^* \right) \mu_t^{*, \alpha} - \frac{1}{2} B R^{-1} B^\top \Psi_t^{Z, \alpha} + \Bar{A} Z_t^\alpha, \quad \mu_0^{*, \alpha} = \mathbb{E}[\xi^\alpha].
        \end{split}
    \end{equation}
     Moreover, the optimal cost $\inf_{G^\alpha \in L^2([0,T], \mathbb{R}^{k})}{J_2^\alpha(K,G^\alpha, Z^\alpha)}$ is independent of $K$. 
\end{proposition}

\begin{proof}
    Similarly to \cite[Chapter 6, Section 2]{yong_stochastic_1999},
    the optimal feedback control 
    of   the deterministic LQ control problem \eqref{eq:J2_controlproblem} is given by 
    $$
    G_t^{*, K, Z, \alpha}(\mu) = - (R^{-1}/2) (B^\top \nabla_\mu v^{Z, \alpha}(t, \mu) + 2 R K_t \mu),
    \quad (t,\mu) \in [0,T]\times \mathbb R^d, 
    $$
    where 
     $v^{Z, \alpha}: [0,T] \times \mathbb{R}^d \to \mathbb{R}$ is
       the optimal value function
     given by $v^{Z, \alpha}(t, \mu) = \mu^\top P_t^* \mu + \mu^\top \Psi_t^{Z, \alpha} + \varphi_t^{Z, \alpha}$ for $t \in [0,T]$ and $\mu \in \mathbb{R}^d$, with  $P^* \in \mathcal{C}([0,T], \mathbb{S}^d)$ satisfying \eqref{eq:PI_riccati}, $\Psi^{Z, \alpha} \in \mathcal{C}([0,T], \mathbb{R}^d)$ satisfying \eqref{eq:localoptimal_PSI_MU_J2} and $\varphi^{Z, \alpha} \in \mathcal{C}([0,T], \mathbb{R}^d)$ satisfying for all $t \in [0,T]$,
     \begin{equation*}
         \frac{\partial \varphi_t^{Z, \alpha}}{\partial t} - \frac{1}{4} (\Psi_t^{Z, \alpha})^\top B R^{-1} B^\top \Psi_t^{Z, \alpha} + (\Bar{A} Z_t^\alpha)^\top \Psi_t^{Z, \alpha} + (H Z_t^\alpha)^\top Q (H Z_t^\alpha) = 0, \quad \varphi_T^{Z, \alpha} = \lvert \Bar{H} Z_T^\alpha \lvert^2.
     \end{equation*}
     Notice that $P^*, \Psi^{Z, \alpha}$ and $\varphi^{Z, \alpha}$ are independent of $K$. Therefore, the optimal value function given arbitrary slope $K$ and graphon aggregate $Z^\alpha$ is independent of $K$. By the relationship $v^{Z, \alpha}(0, \mu_0^\alpha) = \inf_{G^\alpha \in L^2([0,T], \mathbb{R}^{k})}{J_2^\alpha(K,G^\alpha, Z^\alpha)}$, we conclude that the optimal cost  for the minimization problem \eqref{eq:J2_controlproblem} $\inf_{G^\alpha \in L^2([0,T], \mathbb{R}^{k})}{J_2^\alpha(K,G^\alpha, Z^\alpha)}$ is also independent of $K$. Then, noting that $\nabla_\mu v^{Z, \alpha}(t,\mu) = 2 P_t^* \mu + \Psi_t^{Z, \alpha}$ and using the optimal slope $K^*$ \eqref{eq:systemcontrol} yields the optimal intercept \eqref{eq:OPTIMALG_P2}. Finally, inserting the optimal intercept $G^{*, K, Z, \alpha}$ in the state mean \eqref{eq:mean} gives \eqref{eq:optim_state_intercept}.
\end{proof}

\begin{proof}[Proof of Theorem \ref{theorem:cost_decomposition}]
    Let $\alpha \in I$.
    To decompose the cost functional $J^\alpha$ into $J_1^\alpha(K) + J_2^\alpha(K, G^\alpha, Z^\alpha)$ for any $Z^\alpha \in\mathcal{C}([0,T], \mathbb{R}^d)$ and $\theta \coloneqq (K,G^\alpha) \in L^2([0,T], \mathbb{R}^{k \times d}) \times L^2([0,T], \mathbb{R}^k)$, define the state second moment matrix $\Sigma_t^{\theta, \alpha} \coloneqq \mathbb{E}[X_t^{\theta, \alpha} (X_t^{\theta,\alpha})^\top]$ where $X^{\theta,\alpha}$ satisfies \eqref{eq:state_player_alpha} for policy parameter $\theta$ and graphon aggregate $Z^\alpha$. In the following, the $\theta$ in the superscript of the state process is dropped to simplify notation. Then the cost functional can be written as
    \begin{align*}
        J^\alpha(K, G^\alpha, Z^\alpha) &= \mathbb{E}\Bigg[\int_0^T \left[ (X_t^\alpha - H Z_t^\alpha)^\top Q (X_t^\alpha - H Z_t^\alpha) + (K_t X_t^\alpha + G_t^\alpha)^\top R (K_t X_t^\alpha + G_t^\alpha) \right] \, dt \\
        &\quad+ (X_T^\alpha - \Bar{H} Z_T^\alpha)^\top \Bar{Q} (X_T^\alpha - \Bar{H} Z_T^\alpha) \Bigg] \\
        &= \int_0^T \left[ \mathrm{tr}(Q \Sigma_t^\alpha) + \mathrm{tr}(K_t^\top R K_t \Sigma_t^\alpha) \right] \, dt + \mathrm{tr}(\Bar{Q} \Sigma_T^\alpha) \\
        &\quad+ \int_0^T \left[ - 2 (H Z_t^\alpha) Q \mu_t^\alpha + (H Z_t^\alpha)^\top Q (H Z_t^\alpha) - 2 (K_t \mu_t^\alpha)^\top R G_t^\alpha + (G_t^\alpha)^\top R G_t^\alpha \right] \, dt \\
        &\quad- 2 (\Bar{H} Z_T^\alpha)^\top \Bar{Q} \mu_T^\alpha + (\Bar{H} Z_T^\alpha)^\top \Bar{Q} (\Bar{H} Z_T^\alpha) \\
        &= \int_0^T \mathrm{tr}\left((Q + K_t^\top R K_t) \vartheta_t^\alpha \right) \, dt + \mathrm{tr}(\Bar{Q} \vartheta_T^\alpha) \\
        &\quad+ \int_0^T \left[ (\mu_t^\alpha - H Z_t^\alpha)^\top Q (\mu_t^\alpha - H Z_t^\alpha) + (K_t \mu_t^\alpha + G_t^\alpha)^\top R (K_t \mu_t^\alpha + G_t^\alpha) \right] \, dt \\
        &\quad+ (\mu_T^\alpha - \Bar{H} Z_T^\alpha)^\top \Bar{Q} (\mu_T^\alpha - \Bar{H} Z_T^\alpha) \\
        &= J_1^\alpha(K) + J_2^\alpha(K,G^\alpha,Z^\alpha),
    \end{align*}
    where we used $\Sigma_t^\alpha = \mu_t^\alpha (\mu_t^\alpha)^\top + \vartheta_t^\alpha$ with $\mu^\alpha$ satisfying \eqref{eq:mean} and $\vartheta^\alpha$ satisfying \eqref{eq:variance} for policy parameter $\theta$ and graphon aggregate $Z^\alpha$.
    
    By Proposition \ref{proposition:optimG_fixKZ}, for arbitrary $K \in L^2([0,T], \mathbb{R}^{k \times d})$ and $Z^\alpha \in \mathcal{C}([0,T], \mathbb{R}^d)$, the optimal cost $\inf_{G^\alpha \in L^2([0,T], \mathbb{R}^{k})} J_2^\alpha(K,G^\alpha,Z^\alpha)$ is independent of $K$.
    Let
    $ K^*\in \argmin_{K\in L^2([0,T], \mathbb{R}^{k \times d})} J_1^\alpha(K)$
    and $  G^{*, Z, \alpha} \in \argmin_{G^\alpha \in L^2([0,T], \mathbb{R}^{k})}J_2^\alpha(K^*,G^\alpha,Z^\alpha)$. 
     Then 
    \begin{align*}
        J(K, G^\alpha, Z^\alpha) &= J_1^\alpha(K) + J_2^\alpha(K,G^\alpha,Z^\alpha) \ge J_1^\alpha(K^*) + J_2^\alpha(K,G^\alpha,Z^\alpha)
        \\
        &\ge 
        J_1^\alpha(K^*) +  \inf_{G^\alpha \in L^2([0,T], \mathbb{R}^{k})}{J_2^\alpha(K,G^\alpha, Z^\alpha)} 
        =
        J_1^\alpha(K^*) +
        \inf_{G^\alpha \in L^2([0,T], \mathbb{R}^{k})}{J_2^\alpha(K^*,G^\alpha, Z^\alpha)}
        \\
        &=  J_1^\alpha(K^*)
        +J_2^\alpha(K^*,G^{*, Z, \alpha}, Z^\alpha) = J(K^*, G^{*, Z, \alpha}, Z^\alpha).
    \end{align*} 
    This finishes the proof.  
\end{proof}

\subsection{Convergence analysis for the slope}
\label{subsection:convergence_analysis_K}

To prove the convergence of the   scheme \eqref{eq:P1_graddescent}, we start by analyzing the cost functional $J_1^\alpha$.

The next lemma establishes a cost difference result which is of similar nature as \cite[Lemma 3.3]{giegrich_convergence_2022}. We approach the problem from a slightly different angle by controlling the state covariance matrix and optimizing a player-dependent cost functional. Nevertheless, the analysis works similarly as in \cite[Lemmas 3.2 and 3.3]{giegrich_convergence_2022} by considering the cost functional for every $\alpha \in I$.
\begin{lemma}
    \label{lemma:P1_COSTDIFFERENCE}
    Let $K, K' \in L^2([0,T], \mathbb{R}^{k \times d})$. Let $P^K$ satisfy \eqref{eq:ODE_P} for parameter $K$ and let $\vartheta'^{,\alpha}$ satisfy \eqref{eq:variance} for parameter $K'$. Then the cost difference for all $\alpha \in I$ is given by
    \begin{equation*}
        \begin{split}
            J_1^\alpha(K') &- J_1^\alpha(K) = \int_0^T \left[ \langle K_t' - K_t, 2 (B^\top P_t^K + R K_t) \vartheta_t'^{,\alpha} \rangle_F + \langle K_t' - K_t, R (K_t' - K_t) \vartheta_t'^{,\alpha} \rangle_F \right] \, dt. 
        \end{split}
    \end{equation*}
\end{lemma}

The next lemma gives the Gâteaux derivative of the cost functional $J_1^\alpha$ with respect to $K$.
\begin{lemma}
    \label{lemma:K_cost_derivative}
     Let $\alpha \in I$. Let $K \in L^2([0,T], \mathbb{R}^{k \times d})$. Let $P^K$ satisfy \eqref{eq:ODE_P} and let $\vartheta^{K,\alpha}$ satisfy \eqref{eq:variance} for parameter $K$. Then for all $K' \in L^2([0,T], \mathbb{R}^{k \times d})$, 
     \begin{equation*}
         \frac{\partial}{\partial \epsilon} J_1^\alpha(K + \epsilon K') \Big|_{\epsilon = 0} = \int_0^T \langle K_t', (\nabla_K J_1^\alpha(K))_t \rangle_F \, dt,
     \end{equation*}
     where $(\nabla_K J_1^\alpha(K))_t \coloneqq 2 (B^\top P_t^K + R K_t) \vartheta_t^{K,\alpha}$ for all $t \in [0,T]$.
\end{lemma}

\begin{proof}
    Define the perturbed parameter $K^\epsilon \coloneqq K + \epsilon K' \in L^2([0,T], \mathbb{R}^{k \times d})$ with $\epsilon > 0$ in an arbitrary direction $K' \in L^2([0,T], \mathbb{R}^{k \times d})$. Denote by $\vartheta^{\epsilon, \alpha}$ the function satisfying \eqref{eq:variance} with the perturbed parameter $K^\epsilon$ for player $\alpha$. An application of Gronwall's inequality shows that $\vartheta^{\epsilon, \alpha} \to \vartheta^{K,\alpha}$ uniformly as $\epsilon \searrow 0$. By Lemma \ref{lemma:P1_COSTDIFFERENCE}, the cost difference for $K$ and the perturbed parameter $K^\epsilon$ is
    \begin{equation*}
        J_1^\alpha(K^\epsilon) - J_1^\alpha(K) = \int_0^T \left[ \epsilon \langle K_t', 2 (B^\top P_t^K + R K_t) \vartheta_t^{\epsilon, \alpha} \rangle_F + \epsilon^2 \langle K_t', R K_t' \vartheta_t^{\epsilon, \alpha} \rangle_F \right] \, dt.
    \end{equation*}
    Hence, the Gâteaux derivative of the cost functional $J_1$ at $K$ in the direction of an arbitrary perturbation of $K$ is given by 
    \begin{equation*}
        \begin{split}
            \frac{\partial}{\partial \epsilon} J_1^\alpha(K^\epsilon) \Big|_{\epsilon = 0} &= \lim\limits_{\epsilon \searrow 0} \frac{1}{\epsilon} \left(J_1^\alpha(K^\epsilon) - J_1^\alpha(K)\right) = \int_0^T \langle K_t', 2 (B^\top P_t^K + R K_t) \vartheta_t^{K,\alpha} \rangle_F \, dt,
        \end{split}
    \end{equation*}
    where interchanging limit and integral is justified by the uniform convergence of $\vartheta^{\epsilon, \alpha}$.
\end{proof}
For $K^{(\ell)}$ defined via \eqref{eq:P1_graddescent} and $P^{(\ell)}$ satisfying \eqref{eq:ODE_P} for parameter $K^{(\ell)}$, let for all $t \in [0,T]$,
\begin{equation}
\label{eq:DK_mainpart}
    \mathcal{D}_{K,t}^{(\ell)} \coloneqq 2 (B^\top P_t^{(\ell)} + R K_t^{(\ell)}).
\end{equation}
Hence, $\nabla_K J_1^\alpha(K^{(\ell)}) = \mathcal{D}_{K,\cdot}^{(\ell)} \vartheta^{(\ell), \alpha}$, where $\vartheta^{(\ell), \alpha}$ is the state covariance satisfying \eqref{eq:variance} for parameter $K^{(\ell)}$ for all $\alpha \in I$.

The next lemma provides a result for the cost difference of two consecutive slope parameters defined through the gradient descent scheme introduced in \eqref{eq:P1_graddescent}.

\begin{lemma}
\label{lemma:P1_ITERATION_CD}
Let $\alpha \in I$. For $\ell \in \mathbb{N}_0$ let $K^{(\ell)}$ and $K^{(\ell+1)}$ be defined via \eqref{eq:P1_graddescent}. Let $\vartheta^{(\ell+1),\alpha}$ satisfy \eqref{eq:variance} for parameter $K^{(\ell+1)}$. Then
    \begin{equation*}
        J_1^\alpha(K^{(\ell+1)}) - J_1^\alpha(K^{(\ell)}) \le -\eta_K \int_0^T \left[ \left(\lambda_{\mathrm{min}}(\vartheta_t^{(\ell+1), \alpha}) - \eta_K \lambda_{\mathrm{max}}(\vartheta_t^{(\ell+1), \alpha}) \overline{\lambda}^R \right) \lvert \mathcal{D}_{K,t}^{(\ell)} \rvert^2 \right] \, dt.
    \end{equation*}
\end{lemma}

\begin{proof}
    By Lemma \ref{lemma:P1_COSTDIFFERENCE},
    \begin{equation*}
        \begin{split}
            J_1^\alpha(K^{(\ell+1)}) &- J_1^\alpha(K^{(\ell)}) \\
            &= \int_0^T \left[ -\eta_K \langle \mathcal{D}_{K,t}^{(\ell)}, 2 (B^\top P_t^{(\ell)} + R K_t^{(\ell)}) \vartheta_t^{(\ell+1), \alpha} \rangle_F + \eta_K^2 \langle \mathcal{D}_{K,t}^{(\ell)}, R \mathcal{D}_{K,t}^{(\ell)} \vartheta_t^{(\ell+1), \alpha} \rangle_F \right] \, dt \\
            &= -\eta_K \int_0^T \left[ \langle \mathcal{D}_{K,t}^{(\ell)}, \mathcal{D}_{K,t}^{(\ell)} \vartheta_t^{(\ell+1), \alpha} \rangle_F - \eta_K \langle \mathcal{D}_{K,t}^{(\ell)}, R \mathcal{D}_{K,t}^{(\ell)} \vartheta_t^{(\ell+1), \alpha} \rangle_F \right] \, dt \\
            &\le -\eta_K \int_0^T \left[ \left(\lambda_{\mathrm{min}}(\vartheta_t^{(\ell+1), \alpha}) - \eta_K \lambda_{\mathrm{max}}(\vartheta_t^{(\ell+1), \alpha})  \overline{\lambda}^R \right) \lvert \mathcal{D}_{K,t}^{(\ell)} \rvert^2 \right] \, dt,
        \end{split}
    \end{equation*}
    where $\mathcal{D}_{K}^{(\ell)}$ defined in \eqref{eq:DK_mainpart} was used. 
\end{proof}

The next lemma proves a gradient dominance result (also known as Polyak-{\L}ojasiewicz inequality), which is a widely used tool to prove convergence in the policy gradient literature.

\begin{lemma}
\label{lemma:P1_GRADDOMINANCE_VARBOUND}
Suppose Assumption \ref{assumption:RandCOV} holds. Let $M^{\vartheta,*}$ satisfy $M^{\vartheta,*} I_d \succcurlyeq \vartheta_t^{*, \alpha}$ for all $t \in [0,T]$ and $\alpha \in I$. Let $K^*$ be the optimal slope defined in \eqref{eq:systemcontrol}. For $\ell \in \mathbb{N}_0$ let $K^{(\ell)}$ be defined via \eqref{eq:P1_graddescent}. Then for all $\alpha \in I$,
    \begin{equation*}
        0 \ge J_1^\alpha(K^*) - J_1^\alpha(K^{(\ell)}) \ge - \frac{M^{\vartheta,*}}{2 \underline{\lambda}^R} \int_0^T \lvert \mathcal{D}_{K,t}^{(\ell)} \rvert^2 \, dt.
    \end{equation*}
\end{lemma}

\begin{proof}
    By Lemma \ref{lemma:P1_COSTDIFFERENCE},
    \begin{equation*}
        \begin{split}
            J_1^\alpha(K^*) - J_1^\alpha(K^{(\ell)}) &= \int_0^T \left[ \langle K_t^* - K_t^{(\ell)}, \mathcal{D}_{K,t}^{(\ell)} \vartheta_t^{*, \alpha} \rangle_F + \langle K_t^* - K_t^{(\ell)}, R (K_t^* - K_t^{(\ell)}) \vartheta_t^{*, \alpha} \rangle_F \right] \, dt \\
            &\ge \frac{1}{2} \int_0^T \Big[ \langle (K_t^* - K_t^{(\ell)}), R (K_t^* - K_t^{(\ell)}) \vartheta_t^{*, \alpha} \rangle_F - \langle R^{-1} \mathcal{D}_{K,t}^{(\ell)}, \mathcal{D}_{K,t}^{(\ell)} \vartheta_t^{*, \alpha} \rangle_F \Big] \, dt \\
            &\ge - \frac{1}{2} \int_0^T \left[ \langle R^{-1} \mathcal{D}_{K,t}^{(\ell)}, \mathcal{D}_{K,t}^{(\ell)} \vartheta_t^{*, \alpha} \rangle_F \right] \, dt \\
            &\ge - \frac{1}{2} \int_0^T \left[ \frac{\lambda_{\max}(\vartheta_t^{*, \alpha})}{\underline{\lambda}^R} \lvert \mathcal{D}_{K,t}^{(\ell)} \rvert^2 \right] \, dt ,
        \end{split}
    \end{equation*}
    where the first inequality follows from
    \begin{equation*}
        \langle \Delta K, D \vartheta \rangle_F + \langle \Delta K, R \Delta K \vartheta \rangle_F \ge \frac{1}{2} \left( \langle \Delta K, R \Delta K \vartheta \rangle_F - \langle R^{-1} D, D \vartheta \rangle_F \right) 
    \end{equation*}
    with $\Delta K \coloneqq K_t^* - K_t^{(\ell)}$ and $D \coloneqq \mathcal{D}_{K,t}^{(\ell)}$. The second inequality holds due to the positive (semi)-definiteness of $R$ and the covariance matrix. Then, by the definition of $M^{\vartheta,*}$,
    \begin{equation*}
        J_1^\alpha(K^*) - J_1^\alpha(K^{(\ell)}) \ge - \frac{M^{\vartheta,*}}{2 \underline{\lambda}^R} \int_0^T \lvert \mathcal{D}_{K,t}^{(\ell)} \rvert^2 \, dt.
    \end{equation*}
    Finally, as $K^*$ is the optimal slope parameter, $J_1^\alpha(K^*) \le J_1^\alpha(K^{(\ell)})$.
\end{proof}
Note that Lemma \ref{lemma:P1_ITERATION_CD} depends explicitly on the covariance dynamics $\vartheta^{(\ell+1), \alpha}$ and thus also on the player $\alpha \in I$, which distinguishes it from \cite[Proposition 3.9]{giegrich_convergence_2022}.

The next proposition establishes a uniform upper bound for the iterates $(K^{(\ell)})_{\ell \in \mathbb{N}_0}$ defined via \eqref{eq:P1_graddescent}. This is not only a crucial component to find uniform bounds for the eigenvalues of the covariance matrix, but also important for the subsequent convergence analysis of the intercept parameter. The proof adapts \cite[Proposition 3.5 (2)]{giegrich_convergence_2022} to the present setting.

\begin{proposition}
    
\label{prop:BOUND_K}
    Let $K^{(0)} \in L^\infty([0,T], \mathbb{R}^{k \times d})$ be the initial slope parameter and let $P^{(0)}$ satisfy \eqref{eq:ODE_P} for parameter $K^{(0)}$. Let $K^*$ be the optimal slope parameter defined in \eqref{eq:systemcontrol}. For $\ell \in \mathbb{N}$ let $K^{(\ell)}$ be defined via \eqref{eq:P1_graddescent}. Then there exists a $C_0^K > 0$ such that for all $\ell \in \mathbb{N}_0$ and $\eta_K \in (0, 1/(2 \overline{\lambda}^R))$, $\lVert K^{(\ell)} \rVert_{L^\infty} \le C_0^K$ and $\lVert K^* \rVert_{L^\infty} \le C_0^K$ with
    \begin{equation*}
        C_0^K \coloneqq \lVert K^{(0)} \rVert_{L^\infty} + \frac{\lVert B \rVert_{L^\infty}}{\underline{\lambda}^R} \lVert P^{(0)} \rVert_{L^\infty}.
    \end{equation*}
\end{proposition}

The following lemma shows the existence of player-independent uniform lower and upper bounds for the eigenvalues of the covariance along the gradient descent updates.

\begin{lemma}
\label{lemma:BOUND_VAR}
    Suppose Assumption \ref{assumption:RandCOV} holds. For $\ell \in \mathbb{N}_0$ let $K^{(\ell)}$ be be defined via \eqref{eq:P1_graddescent}.  For all $\alpha \in I$ let $\vartheta^{(\ell), \alpha}$ satisfy \eqref{eq:variance} for parameter $K^{(\ell)}$. For $\eta_K \in (0, 1/(2 \overline{\lambda}^R))$ there exist $M^\vartheta, M_\vartheta > 0$ such that $M^\vartheta I_d \succcurlyeq \vartheta_t^{(\ell), \alpha} \succcurlyeq M_\vartheta I_d$ for all $t \in [0,T]$ and $\alpha \in [0,1]$.
\end{lemma}

\begin{proof}
    To find the upper bound, we need to show that $\lambda_{\mathrm{min}}(M^\vartheta I_d - \vartheta_t^{(\ell), \alpha}) \ge 0$ for all $t \in [0,T]$, $\ell \in \mathbb{N}_0$ and $\alpha \in I$. To that end, note $\lambda_{\mathrm{min}}(M^\vartheta I_d - \vartheta_t^{(\ell), \alpha}) = M^\vartheta - \lambda_{\mathrm{max}}(\vartheta_t^{(\ell), \alpha})$. Moreover, $\lambda_{\mathrm{max}}(\vartheta_t^{(\ell), \alpha}) = \lVert \vartheta_t^{(\ell), \alpha} \rVert_2 \le \lvert \vartheta_t^{(\ell), \alpha} \rvert \le \lVert \vartheta^{(\ell), \alpha} \rVert_{L^\infty}$. For the lower bound, we need to show $\lambda_{\mathrm{min}}(\vartheta_t^{(\ell), \alpha}) \ge M_\vartheta$. Using similar arguments as \cite[Lemma 3.7]{giegrich_convergence_2022} yields
    \begin{equation*}
        \lVert \vartheta^{(\ell), \alpha} \rVert_{L^\infty} \le T(\lvert \vartheta_0^\alpha \rvert + \lVert D \rVert_{L^\infty}^2) \exp{(2T ( \lVert A \rVert_{L^\infty} + \lVert B \rVert_{L^\infty} \lVert K^{(\ell)} \rVert_{L^\infty}))},
    \end{equation*}
    and
    \begin{equation*}
        \lambda_{\mathrm{min}}(\vartheta_t^{(\ell), \alpha}) \ge \lambda_{\mathrm{min}}(\vartheta_0^\alpha) \exp{(- 2T(\lVert A \rVert_{L^\infty} + \lVert B \rVert_{L^\infty} \lVert K^{(\ell)} \rVert_{L^\infty}))}.
    \end{equation*}
    Combining these two inequalities under Assumption \ref{assumption:RandCOV} with the uniform bound $C_0^K$ from Proposition \ref{prop:BOUND_K} gives player-independent uniform bounds $M_\vartheta$ and $M^\vartheta$ such that for all $\alpha \in I$,
    \begin{equation*}
        \lVert \vartheta^{(\ell), \alpha} \rVert_{L^\infty} \le T( \overline{C}_0^\vartheta  + \lVert D \rVert_{L^\infty}^2) \exp{(2T ( \lVert A \rVert_{L^\infty} + \lVert B \rVert_{L^\infty} C_0^K)} \eqqcolon M^\vartheta,
    \end{equation*}
    and
    \begin{equation*}
        \lambda_{\mathrm{min}}(\vartheta_t^{(\ell), \alpha}) \ge \underline{C}_0^\vartheta \exp{(- 2T(\lVert A \rVert_{L^\infty} + \lVert B \rVert_{L^\infty} C_0^K))} \eqqcolon M_\vartheta.
    \end{equation*}
\end{proof}

The next corollary combines the cost difference in Lemma \ref{lemma:P1_ITERATION_CD} with the uniform bounds for the eigenvalues of the covariance from Lemma \ref{lemma:BOUND_VAR}.

\begin{corollary}
    \label{corollary:P1_ITERATION_CD_VARBOUND}
    Suppose Assumption \ref{assumption:RandCOV} holds. Let $\alpha \in I$. For $\ell \in \mathbb{N}$ let $K^{(\ell)}$ and $K^{(\ell+1)}$ be defined via \eqref{eq:P1_graddescent}. Denote by $M_\vartheta$ and $M^\vartheta$ the uniform bounds derived in Lemma \ref{lemma:BOUND_VAR}. Then for $\eta_K \in (0, 1/(2 \overline{\lambda}^R))$,
    \begin{equation*}
        J_1^\alpha(K^{(\ell+1)}) - J_1^\alpha(K^{(\ell)}) \le -\eta_K \left( M_\vartheta - \eta_K M^\vartheta \overline{\lambda}^R \right) \int_0^T \lvert \mathcal{D}_{K,t}^{(\ell)} \rvert^2 \, dt.
    \end{equation*}
\end{corollary}

Observe that the bound in Corollary \ref{corollary:P1_ITERATION_CD_VARBOUND} is now player-independent.

Now we state some  necessary constants used in the proof of Theorem \ref{theorem:K_CONVERGENCE}. Let $M^\vartheta$ and $M_\vartheta$ be the upper  and lower bounds derived in Lemma \ref{lemma:BOUND_VAR}, respectively. Let $M^{\vartheta,*} > 0$ be such that $M^{\vartheta,*} I_d \succcurlyeq \vartheta_t^{*, \alpha}$ for all $t \in [0,T]$ and $\alpha \in I$ and let $\overline{\lambda}^R, \underline{\lambda}^R > 0$ be such that $\overline{\lambda}^R I_d \succcurlyeq R(t) \succcurlyeq \underline{\lambda}^R I_d$ for all $t \in [0,T]$. Define 
\begin{equation}
\label{eq:constant_convergence_K}
    C_1^K \coloneqq \frac{M_\vartheta}{2 M^\vartheta \overline{\lambda}^R}, \quad
    C_2^K \coloneqq \frac{\underline{\lambda}^R M_\vartheta}{M^{\vartheta,*}},   \quad C_3^K \coloneqq \frac{1}{\underline{\lambda}^R M_\vartheta}.
\end{equation}
Note that 
$C^K_1\le 1/(2\overline{\lambda}^R)$ as $M^\vartheta 
 \ge M_\vartheta$. Moreover, $C_1^K \le 1/C_2^K$. To see this, suppose that Assumption \ref{assumption:RandCOV} holds. Then, note that there exists $M_\vartheta^* > 0$ such that $\mathrm{\lambda}_{\mathrm{min}}(\vartheta_t^{*,\alpha}) \ge M_\vartheta^* \ge M_\vartheta$ for all $t \in [0,T]$ and $\alpha \in [0,1]$ because $\lVert K^* \rVert_{L^\infty} \le C_0^K$ by Proposition \ref{prop:BOUND_K}.

\begin{proof}[Proof of Theorem \ref{theorem:K_CONVERGENCE}]
    By Lemma \ref{lemma:P1_GRADDOMINANCE_VARBOUND} and Corollary \ref{corollary:P1_ITERATION_CD_VARBOUND},
    for all $\alpha \in I$,
    \begin{equation*}
        \begin{split}
            J_1^\alpha(K^{(\ell+1)}) - J_1^\alpha(K^*) &= (J_1^\alpha(K^{(\ell+1)}) - J_1^\alpha(K^{(\ell)})) + (J_1^\alpha(K^{(\ell)}) - J_1^\alpha(K^*)) \\
            &\le - \eta_K \left( M_\vartheta - \eta_K M^\vartheta \overline{\lambda}^R \right) \int_0^T \lvert \mathcal{D}_{K,t}^{(\ell)} \rvert^2 \, dt + (J_1^\alpha(K^{(\ell)}) - J_1^\alpha(K^*)) \\
            &\le \left(1 - 2 \eta_K \frac{\underline{\lambda}^R}{M^{\vartheta,*}} \left( M_\vartheta - \eta_K M^\vartheta \overline{\lambda}^R \right) \right) (J_1^\alpha(K^{(\ell)}) - J_1^\alpha(K^*)) \\
            &\le \left(1 - \eta_K \frac{\underline{\lambda}^R M_\vartheta}{M^{\vartheta,*}} \right) (J_1^\alpha(K^{(\ell)}) - J_1^\alpha(K^*)),
        \end{split}
    \end{equation*}
    where the last inequality holds for   $\eta_K \in (0, C^K_1)$.
Integrating the above inequality with respect to $\alpha$ yields
    \begin{equation*}
         \lVert J_1^\cdot(K^{(\ell+1)}) - J_1^\cdot(K^*) \rVert_{L^2(I)} \le \left(1 - \eta_K C^K_2\right) \lVert J_1^\cdot(K^{(\ell)}) - J_1^\cdot(K^*) \rVert_{L^2(I)},
    \end{equation*}
    where $1 - \eta_K C_2^K > 0$ for $\eta_K \in (0, 1/C_2^K)$.
    Since $\mathcal{D}_{K}^* = 2 (B^\top P^*+RK^*) \equiv 0$, by Lemma \ref{lemma:P1_COSTDIFFERENCE}, 
    \begin{equation*}
        \begin{split}
            \lVert J_1^\cdot(K^{(\ell+1)}) - J_1^\cdot(K^*) \rVert_{L^2(I)}^2 &= \int_0^1 \left\lvert \int_0^T \langle K_t^{(\ell+1)} - K_t^*, R (K_t^{(\ell+1)} - K_t^*) \vartheta_t^{(\ell+1), \alpha} \rangle_F \, dt \right\rvert^2 \, d\alpha   \\
            &\ge (\underline{\lambda}^R)^2 M_\vartheta^2 \lVert K^{(\ell+1)} - K^* \rVert_{L^2([0,T])}^4.
        \end{split}
    \end{equation*}
    This implies that 
    \begin{equation*}
        \begin{split}
            \lVert K^{(\ell+1)} - K^* \rVert_{L^2([0,T])}^2 &\le C^K_3 \lVert J_1^\cdot(K^{(\ell+1)}) - J_1^\cdot(K^*) \rVert_{L^2(I)} \\
            &\le 
            C^K_3
            \left(1 - \eta_K C^K_2 \right) \lVert J_1^\cdot(K^{(\ell)}) - J_1^\cdot(K^*) \rVert_{L^2(I)} \\
            &\le C^K_3 \left(1 - \eta_K C^K_2 \right) ^{\ell+1} \lVert J_1^\cdot(K^{(0)}) - J_1^\cdot(K^*) \rVert_{L^2(I)},
        \end{split}
    \end{equation*}
    which finishes the proof.
\end{proof}

\subsection{Convergence analysis for the intercept with fixed slope and graphon aggregate}
\label{subsection:convergence_analysis_G}

Recall that the player $\alpha \in I$, the slope parameter $K^{(\ell)}$ obtained through the gradient descent scheme \eqref{eq:P1_graddescent} after $\ell \in \mathbb{N}_0$ iterations and the graphon aggregate $Z^\alpha \in \mathcal{C}([0,T], \mathbb{R}^d)$ are fixed. In particular, by Proposition \ref{prop:BOUND_K}, $\lVert K^{(\ell)} \rVert_{L^\infty} \le C_0^K < \infty$. 

The following lemma gives the Gâteaux derivative of $J_2^{\ell, \alpha}$ defined by \eqref{eq:J_2_ell_alpha} with respect to $G$.

\begin{lemma}
\label{lemma:derivative_G}
    Let $\alpha \in I$. Let $G^\alpha \in L^2([0,T], \mathbb{R}^k)$. Let $\zeta^{(K^{(\ell)}, G),\alpha}$ satisfy \eqref{eq:ode_zeta_gradG} and let $\mu^{(K^{(\ell)}, G), \alpha}$ satisfy \eqref{eq:mean}. Then for all $G'^{,\alpha} \in L^2([0,T], \mathbb{R}^k)$,
    \begin{equation*}
        \frac{\partial}{\partial \epsilon} J_2^{\ell, \alpha}(G^\alpha + \epsilon G'^{,\alpha}) \Big|_{\epsilon = 0} = \int_0^T \langle (\nabla_G J_2^{\ell, \alpha}(G^\alpha))_t, G_t'^{,\alpha} \rangle \, dt,
    \end{equation*}
    where $(\nabla_G J_2^{\ell, \alpha}(G^\alpha))_t \coloneqq B^\top \zeta_t^{(K^{(\ell)}, G), \alpha} + 2R (K_t^{(\ell)} \mu_t^{(K^{(\ell)}, G), \alpha} + G_t^\alpha)$ for all $t \in [0,T]$.
\end{lemma}
\begin{proof}
    For any $\alpha \in I$, the Hamiltonian $\mathcal{H}^\alpha: [0,T] \times \mathbb{R}^d \times \mathbb{R}^k \times \mathbb{R}^d \to \mathbb{R}$ of the deterministic control problem \eqref{eq:J2_controlproblem} is given by
    \begin{equation*}
        \begin{split}  
            \mathcal{H}^\alpha(t, \mu, G, \zeta) &= \left[(A + B K_t^{(\ell)}) \mu + B G + \Bar{A} Z_t^\alpha \right]^\top \zeta \\
            &\quad + (\mu - H Z_t^\alpha)^\top Q (\mu - H Z_t^\alpha) + (K_t^{(\ell)} \mu + G)^\top R (K_t^{(\ell)} \mu + G). 
        \end{split}
    \end{equation*}
    For any $G^\alpha \in L^2([0,T], \mathbb{R}^k)$, by Pontryagin's maximum principle
    \cite[Chapter 2, Theorem 3.1]{yong_stochastic_1999}, the adjoint $\zeta^{(K^{(\ell)}, G), \alpha} \in \mathcal{C}([0,T], \mathbb{R}^d)$ satisfies \eqref{eq:ode_zeta_gradG} with policy parameters $(K^{(\ell)}, G^\alpha)$ and graphon aggregate $Z^\alpha$. Hence, using \cite[Corollary 4.11]{carmona_lectures_2016} the Gâteaux derivative of the cost functional $J_2^\alpha$ at $G^\alpha$ in the direction of an arbitrary perturbation $G'^{,\alpha} \in L^2([0,T], \mathbb{R}^k)$ is given by
    \begin{equation*}
        \begin{split}
            \frac{\partial}{\partial \epsilon} J_2^{\ell, \alpha}(G^\alpha + \epsilon G'^{,\alpha}) \Big|_{\epsilon = 0} &= \int_0^T \left\langle \frac{\partial \mathcal{H^\alpha}}{\partial G} (t, \mu_t^{(K^{(\ell)}, G), \alpha}, G_t^\alpha, \zeta_t^{(K^{(\ell)}, G),\alpha}), G_t'^{,\alpha} \right\rangle \, dt \\
            &= \int_0^T \langle B^\top \zeta_t^{(K^{(\ell)}, G), \alpha} + 2R (K_t^{(\ell)} \mu_t^{(K^{(\ell)}, G),\alpha} + G_t^\alpha), G_t'^{,\alpha} \rangle \, dt 
        \end{split}
    \end{equation*}
\end{proof}
To prove the convergence of the gradient descent scheme for the intercept parameter in \eqref{eq:P2_graddescent}, we rewrite the cost functional $J_2^{\ell,\alpha}$ and the gradient expression using integral operators. Firstly, split $J_2^{\ell, \alpha}$ into two parts and define for $G \in L^2([0,T], \mathbb{R}^k)$,

\begin{equation*}
    \begin{split}
        &J_{2,1}^{\ell, \alpha}(G) \coloneqq \left\lVert Q^{1/2} \left( \mu^{(K^{(\ell)}, G),\alpha} - H Z^\alpha \right) \right\rVert_{L^2([0,T])}^2 + \left\lvert \Bar{Q}^{1/2} \left(\mu_T^{(K^{(\ell)}, G),\alpha} - \Bar{H} Z_T^\alpha \right) \right\rvert^2, \\
        &J_{2,2}^{\ell, \alpha}(G) \coloneqq \left\lVert R^{1/2} \left(K^{(\ell)} \mu^{(K^{(\ell)}, G),\alpha} + G \right) \right\rVert_{L^2([0,T])}^2,
    \end{split}
\end{equation*}
where $\mu^{(K^{(\ell)}, G),\alpha}$ satisfies \eqref{eq:mean} for policy parameters $(K^{(\ell)}, G)$ and fixed graphon aggregate $Z^\alpha$.
Hence, $J_2^{\ell, \alpha}(G) = J_{2,1}^{\ell, \alpha}(G) + J_{2,2}^{\ell, \alpha}(G)$. Next, define the linear Volterra integral operators $\Gamma: L^2([0,T], \mathbb{R}^k) \to L^2([0,T], \mathbb{R}^k)$ such that for any $f \in L^2([0,T], \mathbb{R}^k)$,
\begin{equation}
    \label{eq:GAMMA_operator}
    \Gamma[f](t) \coloneqq \int_0^t Y(t) Y^{-1}(s) f(s) \, ds
\end{equation}
and $\Tilde{\Gamma}: L^2([0,T], \mathbb{R}^k) \to L^2([0,T], \mathbb{R}^k)$ such that
\begin{equation}
    \label{eq:GAMMAtilde_operator}
    \Tilde{\Gamma}[f](t) \coloneqq \int_0^t K_t^{(\ell)} Y(t) Y^{-1}(s) B f(s) \, ds
\end{equation}
with $Y \in \mathcal{C}([0,T], \mathbb{R}^{d \times d})$ satisfying
\begin{equation}
    \label{eq:fundamentalmatrix}
    \frac{\partial Y(t)}{\partial t} = (A + B K_t^{(\ell)}) Y(t), \quad t \in [0,T]; \quad Y(0) = I_d.
\end{equation} 
Moreover, let
\begin{equation}
    \label{eq:b1ANDb2}
    b_t^{(1), \alpha} \coloneqq \Gamma[\Bar{A} Z^\alpha](t) - H Z_t^\alpha, \quad b_t^{(2), \ell, \alpha} \coloneqq K_t^{(\ell)} \Gamma[\Bar{A} Z^\alpha](t), \quad c^\alpha \coloneqq \Gamma[\Bar{A} Z^\alpha](T) - \Bar{H} Z_T^\alpha.
\end{equation}
By variation of constants for ODEs \cite[Theorem 3.12]{teschl_ordinary_2012} and linearity, $\mu_t^{(K^{(\ell)}, G), \alpha} = \Gamma[B G](t) + \Gamma[\Bar{A} Z^\alpha](t)$ for all $t \in [0,T]$. Therefore,
\begin{equation*}
    \begin{split}
        &J_{2,1}^{\ell, \alpha}(G) = \left\lVert Q^{1/2} \left(\Gamma[B G] + b^{(1), \alpha} \right) \right\rVert_{L^2([0,T])}^2 + \left\lvert \Bar{Q}^{1/2} \left(\Gamma[B G](T) + c^\alpha \right) \right\rvert^2, \\
        &J_{2,2}^{\ell, \alpha}(G) = \left\lVert R^{1/2} \left( (\Tilde{\Gamma} + \mathrm{id})[G] + b^{(2), \ell, \alpha} \right) \right\rVert_{L^2([0,T])}^2,
    \end{split}
\end{equation*}
where 
$\mathrm{id}$ is the identity operator on 
$L^2([0,T], \mathbb R^k)$.

To prove the strong convexity of $G \mapsto J_2^{\ell, \alpha}(G)$, we show the convexity of $G \mapsto J_{2,1}^{\ell, \alpha}(G)$ and the strong convexity of $G \mapsto J_{2,2}^{\ell, \alpha}(G)$. Approaching the convergence of the intercept parameter using strong convexity differs from the existing literature for policy optimization  methods for mean field game or control problems, e.g., \cite{carmona_linear-quadratic_2019,  wang2021global,frikha_full_2024}. 

\begin{lemma}
\label{lemma:J21_convex}
    The function $L^2([0,T], \mathbb{R}^k) \ni G \mapsto J_{2,1}^{\ell, \alpha}(G)\in \mathbb R$ is convex.
\end{lemma}

\begin{proof}
    First, note that the second derivative at $G \in L^2([0,T], \mathbb{R}^k)$ for $G_1, G_2 \in L^2([0,T], \mathbb{R}^k)$ is
    \begin{equation*}
    \nabla_G^2 J_{2,1}^{\ell, \alpha}(G)[G_1, G_2] = 2 \langle \Gamma[B G_1], Q \Gamma[B G_2] \rangle_{L^2} + 2 \langle \Gamma[B G_1](T), \Bar{Q} \Gamma[ B G_2](T) \rangle.
    \end{equation*}
    Then the convexity follows immediately because $Q(t) \succcurlyeq 0$ for all $t \in [0,T]$ and  $\Bar{Q} \succcurlyeq 0$.
\end{proof}

To prove the strong convexity of $G \mapsto J_{2,2}^{\ell, \alpha}(G)$, we utilize theory on Volterra integral equations; for an overview, see \cite{gripenberg_volterra_1990}.

\begin{definition}
\label{def:volterrakernel}
    A measurable function $\kappa: [0,T]^2 \to \mathbb{R}^{k \times k}$ with $\kappa(t,s) = \mathbf{0} \in \mathbb{R}^{k \times k}$ for $t < s$ is called a Volterra kernel.
\end{definition}

The following lemma provides an upper bound for the operator norm of the resolvent of the operator $\Tilde{\Gamma}$. This bound will be essential to prove the strong convexity of $G \mapsto J_{2,2}^{\ell, \alpha}(G)$.

\begin{lemma}
\label{lemma:resolvent}
    The operator resolvent $(\Tilde{\Gamma} + \mathrm{id})^{-1}: L^2([0,T], \mathbb{R}^k) \to L^2([0,T], \mathbb{R}^k)$ of the Volterra integral operator $\Tilde{\Gamma}$ defined in \eqref{eq:GAMMAtilde_operator} exists and
    \begin{equation*}
        \lVert (\Tilde{\Gamma} + \mathrm{id})^{-1} \rVert_{\mathrm{op}} \le \exp{\left( 2 T^2 \lVert B \rVert_{L^\infty}^2 (C_0^K)^2 \exp{(2T( \lVert A \rVert_{L^\infty} + \lVert B \rVert_{L^\infty} C_0^K))} + \frac{1}{2} \right)} \eqqcolon M_{\mathrm{res}}.
    \end{equation*}
\end{lemma}
\begin{proof}
    Let $\kappa(t,s) \coloneqq \mathds{1}_{[t > s]} K_t^{(\ell)} Y(t) Y^{-1}(s) B$, where $Y$ satisfies \eqref{eq:fundamentalmatrix}. This is a Volterra kernel by Definition \ref{def:volterrakernel}. Note that $\Tilde{\Gamma}$ can be expressed using the kernel $\kappa$ as $\Tilde{\Gamma}[f](t) = \int_0^T \kappa(t,s) f(s) \, ds$. To show that the Volterra kernel $\kappa$ has a kernel resolvent using \cite[Chapter 9, Corollary 3.16]{gripenberg_volterra_1990}, we need to prove that $\kappa$ satisfies the square integrability condition
    \begin{equation*}
        \int_0^T \int_0^T \lvert \kappa(t,s) \rvert^2 \, ds \, dt < \infty.
    \end{equation*}
    Observe that $Y^{-1} \in \mathcal{C}([0,T], \mathbb{R}^{d \times d})$ satisfies for all $t \in [0,T]$,
    \begin{equation}
    \label{eq:fundamentalmatrix_inverse}
        \frac{\partial Y^{-1}(t)}{\partial t} = - Y^{-1}(t) (A + B K_t^{(\ell)}), \quad Y^{-1}(0) = I_d.
    \end{equation}Then, applying Gronwall's inequality on \eqref{eq:fundamentalmatrix} and \eqref{eq:fundamentalmatrix_inverse} yields
    \begin{equation*}
        \begin{split}
             &\lVert Y \rVert_{L^\infty} \le \exp{(T(\lVert A \rVert_{L^\infty} + \lVert B \rVert_{L^\infty} \lVert K^{(\ell)} \rVert_{L^\infty}))}, \quad \text{and} \\
             &\lVert Y^{-1} \rVert_{L^\infty} \le \exp{(T(\lVert A \rVert_{L^\infty} + \lVert B \rVert_{L^\infty} \lVert K^{(\ell)} \rVert_{L^\infty}))}.
        \end{split}
    \end{equation*}
    Hence, $\kappa$ satisfies the square integrability condition as
    \begin{equation}
        \label{eq:square_integr_kernel}
        \int_0^T \int_0^T \lvert \kappa(t,s) \rvert^2 \, ds \, dt \le T^2 \lVert K^{(\ell)} \rVert_{L^\infty}^2 \lVert B \rVert_{L^\infty}^2 \exp{(2T(\lVert A \rVert_{L^\infty} + \lVert B \rVert_{L^\infty} \lVert K^{(\ell)} \rVert_{L^\infty}))} < \infty.
    \end{equation}
    Now, let $f \in L^2([0,T], \mathbb{R}^k)$ and $\gamma \neq 0$ be arbitrary. Then by \cite[Chapter 9, Theorem 3.6]{gripenberg_volterra_1990}, the Volterra integral equation $(\Tilde{\Gamma} + \gamma \mathrm{id})x = f$ has a unique solution $x \in L^2([0,T], \mathbb{R}^k)$. Hence, $(\Tilde{\Gamma} + \gamma \mathrm{id})^{-1}f \in L^2([0,T], \mathbb{R}^k)$ and the resolvent $( \Tilde{\Gamma} + \gamma \mathrm{id})^{-1}$ exists.
    To find the upper bound, we first observe that the spectrum of $\Tilde{\Gamma}$ is given by $\sigma(\Tilde{\Gamma}) = \{ 0 \}$.
    Then, \cite[Theorem 3.1]{dechevski_sharp_1994} implies the following bound for the resolvent of the operator $\Tilde{\Gamma}$,
    \begin{equation}
        \begin{split}
            \lVert (\Tilde{\Gamma} + \mathrm{id})^{-1} \rVert_{\mathrm{op}} &\le \exp{\left( 2 \lVert  \Tilde{\Gamma} \rVert_{\mathrm{HS}}^2 + \frac{1}{2} \right)},
        \end{split}
    \end{equation}
    where $\lVert \cdot \rVert_{\mathrm{HS}}$ denotes the Hilbert-Schmidt norm. The integral operator $\Tilde{\Gamma}$ is a Hilbert-Schmidt integral operator as the Volterra kernel satisfies the square integrability condition \eqref{eq:square_integr_kernel} -- see, for example \cite[Theorem 6.12]{brezis_functional_2011}. Hence, by \cite[Problem 2, p.\@499]{brezis_functional_2011}, $\lVert \Tilde{\Gamma} \rVert_{\mathrm{HS}}^2 \le \int_0^T \int_0^T \lvert \kappa(t,s) \rvert^2 \, ds \, dt.$ Using \eqref{eq:square_integr_kernel}, the operator norm can be bounded,
    \begin{equation*}
        \begin{split}
            \lVert (\Tilde{\Gamma} + \mathrm{id})^{-1} \rVert_{\mathrm{op}} &\le \exp{\left( 2 T^2 \lVert B \rVert_{L^\infty}^2 \lVert K^{(\ell)} \rVert_{L^\infty}^2 \exp{(2T( \lVert A \rVert_{L^\infty} + \lVert B \rVert_{L^\infty} \lVert K^{(\ell)} \rVert_{L^\infty}))} + \frac{1}{2} \right)} \\
            &\le \exp{\left( 2 T^2 \lVert B \rVert_{L^\infty}^2 (C_0^K)^2 \exp{(2T( \lVert A \rVert_{L^\infty} + \lVert B \rVert_{L^\infty} C_0^K))} + \frac{1}{2} \right)} < \infty,
        \end{split}
    \end{equation*}
    where the uniform bound from Proposition \ref{prop:BOUND_K} was used.
\end{proof}

The following lemma proves the strong convexity of $G \mapsto J_{2,2}^{\ell, \alpha}(G)$.

\begin{lemma}
\label{lemma:J22_stronglyconvex}
    Let $m \coloneqq (2 \underline{\lambda}^R)/M_{\mathrm{res}}^2 > 0$ with $M_{\mathrm{res}}$ introduced in Lemma \ref{lemma:resolvent}. Then the function $L^2([0,T], \mathbb{R}^k) \ni G \mapsto J_{2,2}^{\ell, \alpha}(G)$ is $m$-strongly convex w.r.t. $\lVert \cdot \rVert_{L^2([0,T])}$. In particular, $m$ does not depend on the fixed slope parameter $K^{(\ell)}$ and graphon aggregate $Z^\alpha$.
\end{lemma}

\begin{proof}
Observe  that the second derivative at $G \in L^2([0,T], \mathbb{R}^k)$ for $G_1, G_2 \in L^2([0,T], \mathbb{R}^k)$ is given by $\nabla_G^2 J_{2,2}^{\ell,\alpha}(G)[G_1, G_2] = 2 \langle (\Tilde{\Gamma} + \mathrm{id})[G_1], R (\Tilde{\Gamma} + \mathrm{id})[G_2] \rangle_{L^2}$. Proving the strong convexity of $G \mapsto J_{2,2}^{\ell,\alpha}(G)$
is reduced  to proving the existence of $m > 0$ such that for any $G \in L^2([0,T], \mathbb{R}^k)$, $\nabla_G^2 J_{2,2}^{\ell, \alpha}(G)[G, G] \ge m \lVert G \rVert_{L^2([0,T])}^2$.
    To that end, observe that $\underline{\lambda}^R > 0$ as $R \in \mathcal{C}([0,T], \mathbb{S}_{+}^k)$ and 
    \begin{equation*}
        \begin{split}
            \nabla_G^2 J_{2,2}^{\ell,\alpha}(G)[G, G] 
            &= 2 \langle (\Tilde{\Gamma} + \mathrm{id})[G], R (\Tilde{\Gamma} + \mathrm{id})[G] \rangle_{L^2}  \ge 2 \underline{\lambda}^R \lVert (\Tilde{\Gamma} + \mathrm{id})[G] \rVert_{L^2([0,T])}^2.
        \end{split}
    \end{equation*}
Note that 
    \begin{equation*}
        \begin{split}
            \lVert G \rVert_{L^2([0,T])}^2 &= \lVert (\Tilde{\Gamma} + \mathrm{id})^{-1}(\Tilde{\Gamma} + \mathrm{id})[G] \rVert_{L^2([0,T])}^2 \\
            &\le \lVert (\Tilde{\Gamma} + \mathrm{id})^{-1} \rVert_{\mathrm{op}}^2 \lVert (\Tilde{\Gamma} + \mathrm{id})[G] \rVert_{L^2([0,T])}^2 \\
            &\le M_{\mathrm{res}}^2 \lVert (\Tilde{\Gamma} + \mathrm{id})[G] \rVert_{L^2([0,T])}^2 = \frac{2 \underline{\lambda}^R}{m} \lVert (\Tilde{\Gamma} + \mathrm{id})[G] \rVert_{L^2([0,T])}^2,
        \end{split}
    \end{equation*}
    where the upper bound for the operator norm of the resolvent from Lemma \ref{lemma:resolvent} was used. Multiplying the inequality by $m$ concludes the proof.
\end{proof}

The previous results allow to show the strong convexity of $G \mapsto J_{2}^{\ell, \alpha}(G)$.

\begin{proof}[Proof of Proposition \ref{proposition:J2_strongconvexity}]
The strong convexity follows immediately from Lemmas \ref{lemma:J21_convex} and \ref{lemma:J22_stronglyconvex} as $J_2^{\ell, \alpha}$ is the sum of a convex and $m$-strongly convex function with $m$ defined in Lemma \ref{lemma:J22_stronglyconvex}.
\end{proof}

Next, we prove the Lipschitz smoothness property.

\begin{proof}[Proof of Proposition \ref{proposition:LSmooth}]
    By Taylor expansion and the quadratic structure of $J_2^{\ell, \alpha}$,
    \begin{equation}
        \label{eq:L_smooth_taylor}
        J_2^{\ell, \alpha}(G') - J_2^{\ell, \alpha}(G) = \langle G' - G, \nabla_G J_2^{\ell, \alpha}(G) \rangle_{L^2} + \nabla_G^2 J_2^{\ell, \alpha}(G)[G' - G, G' - G].
    \end{equation}
    Hence, it remains to bound the second order term. Recall the second derivatives for $J_{2,1}^{\ell,\alpha}$ and $J_{2,2}^{\ell,\alpha}$ from Lemmas \ref{lemma:J21_convex} and \ref{lemma:J22_stronglyconvex}, respectively. Then
    \begin{equation*}
        \begin{split}
            \nabla_G^2 &J_2^{\ell,\alpha}(G)[G' - G, G' - G] \\
            &= 2 \langle \Gamma[B(G' - G)], Q \Gamma[B(G' - G)] \rangle_{L^2} + 2 \langle \Gamma[B (G' - G)](T), \Bar{Q} \Gamma[B (G' - G)](T) \rangle \\
            &\quad\quad\quad+ 2 \langle (\Tilde{\Gamma} + \mathrm{id})[G' - G], R (\Tilde{\Gamma} + \mathrm{id})[G' - G] \rangle_{L^2} \\
            &\le 2 \overline{\lambda}^Q \lVert \Gamma[B (G' - G)] \rVert_{L^2([0,T])}^2 + 2 \overline{\lambda}^{\Bar{Q}} \lvert \Gamma [B (G' - G)](T) \rvert^2 + 2 \overline{\lambda}^R \lVert (\Tilde{\Gamma} + \mathrm{id})[G' - G] \rVert_{L^2([0,T])}^2 \\
            &\le 2 \left( \overline{\lambda}^Q \lVert \Gamma B \rVert_{\mathrm{op}}^2 + \overline{\lambda}^{\Bar{Q}} \lVert \Gamma B [\cdot](T) \rVert_{\mathrm{op}}^2 + \overline{\lambda}^R \lVert \Tilde{\Gamma} + \mathrm{id} \rVert_{\mathrm{op}}^2 \right) \lVert G' - G \rVert_{L^2([0,T])}^2,
        \end{split}
    \end{equation*}
     where $\overline{\lambda}^R, \overline{\lambda}^Q, \overline{\lambda}^{\Bar{Q}} > 0$ are such that $\overline{\lambda}^R I_d \succcurlyeq R(t)$, $\overline{\lambda}^Q I_d \succcurlyeq Q(t)$ for all $t \in [0,T]$ and $\overline{\lambda}^{\Bar{Q}} I_d \succcurlyeq \Bar{Q}$.
    Moreover, using the uniform bound from Proposition \ref{prop:BOUND_K}, there exists $L > 0$ such that
    \begin{equation}
        \label{eq:L_SMOOTHNESS_DEF}
        L \ge 4 \left( \overline{\lambda}^Q \lVert \Gamma B \rVert_{\mathrm{op}}^2 + \overline{\lambda}^{\Bar{Q}} \lVert \Gamma B [\cdot](T) \rVert_{\mathrm{op}}^2 + \overline{\lambda}^R \lVert \Tilde{\Gamma} + \mathrm{id} \rVert_{\mathrm{op}}^2 \right)
    \end{equation}
    for all  slope parameters $(K^{(\ell)})_{\ell \in \mathbb N}$ and graphon aggregates $(Z^\alpha)_{\alpha\in I}\in \mathcal C([0,T],\mathbb R^d)$. 
    Combining this upper bound $L$ with \eqref{eq:L_smooth_taylor} 
    leads to the desired   Lipschitz smoothness.
\end{proof}

The following convergence result is well known for a finite-dimensional optimization problem, however, for completeness, we give the main steps of the proof in $L^2([0,T], \mathbb{R}^k)$.

\begin{proof}[Proof of Theorem \ref{theorem:G_CONVERGENCE}]
    Let $\alpha \in I$. By Proposition \ref{proposition:LSmooth} and \eqref{eq:P2_graddescent}, we can show
    \begin{equation}
        \label{eq:gradbound_LSmooth}
        \lVert \nabla_G J_2^{\ell, \alpha}(G^{(i), \alpha}) \lVert_{L^2([0,T])}^2 \le \frac{1}{\frac{L}{2} \eta^2 - \eta} \left( J_2^{\ell, \alpha}(G^{(i), \alpha}) - J_2^{\ell, \alpha}(G^{*, \alpha}) \right),
    \end{equation}
    where $\eta < 2/L$.
    Moreover, note that the $m$-strong convexity of $G \mapsto J_2(G)$ is equivalent to
    \begin{equation}
        \label{eq:strong_convex_cond2}
        J_2^{\ell, \alpha}(G^{*, \ell, \alpha}) - J_2^{\ell, \alpha}(G^{(i), \alpha}) \ge \langle G^{*, \ell, \alpha} - G^{(i), \alpha}, \nabla_G J_2^{\ell, \alpha}(G^{(i), \alpha}) \rangle_{L^2} + \frac{m}{2} \lVert G^{*, \ell, \alpha} - G^{(i), \alpha} \rVert_{L^2([0,T])}.
    \end{equation}
    Then using \eqref{eq:P2_graddescent}, \eqref{eq:gradbound_LSmooth} and \eqref{eq:strong_convex_cond2} proves the first statement
    \begin{equation}
        \label{eq:decreasing_G}
        \begin{split}
            \lVert G^{(i+1), \alpha} - G^{*, \ell, \alpha} \rVert_{L^2([0,T])}^2 &\le (1 - \eta m) \lVert G^{(i), \alpha} - G^{*, \ell, \alpha} \rVert_{L^2([0,T])}^2 \\
            &\quad+ \left( \frac{\eta^2}{\frac{L}{2} \eta^2 - \eta} - 2 \eta \right) \left(J_2^{\ell, \alpha}(G^{(i), \alpha}) - J_2^{\ell, \alpha}(G^{*, \ell, \alpha}) \right) \\
            &\le (1 - \eta m) \lVert G^{(i), \alpha} - G^{*, \ell, \alpha} \rVert_{L^2([0,T])}^2,
        \end{split}
    \end{equation}
    where the last inequality holds for $\eta < 2/L$. 
    The last statement
    follows from 
      Fubini's theorem.
\end{proof}

\subsection{Convergence analysis to Nash equilibrium}
\label{subsection:convergence_analsis_to_NE}

We established linear convergence of the slope parameter $K$ to the optimal parameter $K^*$. The intercept $G$ converges linearly to the best-response $\argmin_{G^\alpha \in L^2([0,T], \mathbb{R}^k)} J_2(K, G^\alpha, Z^\alpha)$ for fixed $K$ and $Z^\alpha$ for every $\alpha \in I$. It remains to show the convergence of Algorithm \ref{algo} to the NE.

For the convergence analysis of the intercept, introduce for any $\mu \in \mathcal{C}([0,T], L^2(I, \mathbb{R}^d))$ the integral operator $\Lambda[\cdot, \mu]: L^2([0,T], L^2(I, \mathbb{R}^d)) \to L^2([0,T], L^2(I, \mathbb{R}^d))$ defined as
\begin{equation*}
    \begin{split}
        \Lambda^{\alpha}[\Psi, \mu](t) &\coloneqq 2 \Bar{Q} \Bar{H} \mathbb{W}[\mu_T](\alpha) - \int_t^T \Bigg[ \left(A^\top - P_s^* B R^{-1} B^\top \right) \Psi_s^\alpha + 2 (P_s^* \Bar{A} - Q H) \mathbb{W}[\mu_s](\alpha) \Bigg] \, ds
    \end{split}
\end{equation*}
for all $\alpha \in I$ and $t \in [0,T]$. Note that for $\mu = \mu^{(n)} \in \mathcal{C}([0,T], L^2(I, \mathbb{R}^d))$, i.e., the state mean fixed at the beginning of the $n$-th outer iteration, this integral operator corresponds to the integral equation of \eqref{eq:localoptimal_PSI_MU_J2} as $Z_t^{(n),\alpha} = \mathbb{W}[\mu_t^{(n)}](\alpha)$.

To analyze the convergence of the state mean, define for any $K \in L^2([0,T], \mathbb{R}^{k \times d})$ and $G \in L^2([0,T], L^2(I, \mathbb{R}^k))$ the integral operator $\Phi^{(n)}[\cdot, K, G]: \mathcal{C}([0,T], L^2(I, \mathbb{R}^d)) \to \mathcal{C}([0,T], L^2(I, \mathbb{R}^d))$,
\begin{equation*}
    \Phi^{(n), \alpha}[\mu, K, G](t) \coloneqq \mu_0^\alpha + \int_0^t \left[ (A + B K_s) \mu_s^\alpha + B G_s^\alpha + \Bar{A} Z_s^{(n), \alpha} \right] \, ds,
\end{equation*} 
and $\Phi[\cdot, K, G]: \mathcal{C}([0,T], L^2(I, \mathbb{R}^d)) \to \mathcal{C}([0,T], L^2(I, \mathbb{R}^d))$,
\begin{equation}
\label{eq:PHI_OPERATOR}
    \Phi^\alpha[\mu, K, G](t) \coloneqq \mu_0^\alpha + \int_0^t \left[ (A + B K_s) \mu_s^\alpha + B G_s^\alpha + \Bar{A} \mathbb{W}[\mu_s](\alpha) \right] \, ds
\end{equation}
for all $\alpha \in I$ and $t \in [0,T]$. These integral operators represent the integral equations of the state mean dynamics with fixed \eqref{eq:mean} and consistent \eqref{eq:mean_policy} graphon aggregate, respectively.

For the subsequent analysis, recall the following facts:
(i) $\lVert R^{-1} B^\top P^* \rVert_{L^\infty} = \lVert K^* \rVert_{L^\infty}$ by Proposition \ref{proposition:existenceANDuniqueness}; (ii) recall the uniform bound $C_0^K$ in Proposition \ref{prop:BOUND_K} for slope parameters obtained through the gradient descent scheme in \eqref{eq:P1_graddescent} and $\lVert K^* \rVert_{L^\infty} \le C_0^K$; (iii) under Assumption \ref{assumption:convergence}, $M_1 < 1$ and $T(\lVert A \rVert_{L^\infty} + \lVert B \rVert_{L^\infty} C_0^K) \le M_1$.

We start by providing a perturbation analysis for the operator $\Lambda$.

\subsubsection{Perturbation analysis for $\Lambda$}

The following lemma analyzes the perturbation of $\Lambda$ in $\Psi$.

\begin{lemma}
\label{lemma:PERTURB_PSI}
    Let $K^*$ be the optimal slope parameter defined in \eqref{eq:systemcontrol} and let $\mu \in \mathcal{C}([0,T], L^2(I, \mathbb{R}^d))$. Then for any $\Psi, \Psi' \in L^2([0,T], L^2(I, \mathbb{R}^d))$,
    \begin{equation*}
        \lVert \Lambda[\Psi, \mu] - \Lambda[\Psi', \mu] \rVert_{L_B^2} \le T(\lVert A \rVert_{L^\infty} + \lVert B \rVert_{L^\infty} \lVert K^* \rVert_{L^\infty}) \lVert \Psi - \Psi' \rVert_{L_B^2}.
    \end{equation*}
\end{lemma}

\begin{proof}
    By Jensen's inequality, for $t \in [0,T]$ and $\alpha \in I$,
    \begin{equation*}
        \begin{split}
            \lvert \Lambda^{\alpha}[\Psi, \mu](t) - \Lambda^{\alpha}[\Psi', \mu](t) \rvert^2 &= \left\lvert \int_t^T \left[ \left(A^\top - P_s^* B R^{-1} B^\top \right)(\Psi_s^\alpha - \Psi_s'^{,\alpha} ) \right] \, ds \right\rvert^2 \\
            &\le T (\lVert A \rVert_{L^\infty} + \lVert B \rVert_{L^\infty} \lVert K^* \rVert_{L^\infty})^2 \int_0^T \lvert \Psi_s^\alpha - \Psi_s'^{,\alpha} \rvert^2 \, ds.
        \end{split}
    \end{equation*}
    Hence, by Fubini's theorem,
    \begin{equation*}
        \begin{split}
            \lVert \Lambda[\Psi, \mu] - \Lambda[\Psi', \mu] \rVert_{L_B^2}^2 &\le T (\lVert A \rVert_{L^\infty} + \lVert B \rVert_{L^\infty} \lVert K^* \rVert_{L^\infty})^2 \int_0^T \int_0^1 \int_0^T \lvert \Psi_s^\alpha - \Psi_s'^{,\alpha} \rvert^2 \, ds \, d \alpha \, dt \\
            &= T^2 (\lVert A \rVert_{L^\infty} + \lVert B \rVert_{L^\infty} \lVert K^* \rVert_{L^\infty})^2 \lVert \Psi - \Psi' \rVert_{L_B^2}^2.
        \end{split}
    \end{equation*}
    Taking the square root completes the proof.
\end{proof}

The next result analyzes the error introduced through the fixed graphon aggregate at the beginning of the $n$-th outer iteration.

\begin{lemma}
\label{lemma:LAMBDA_nvsexact}
    Let $P^* \in \mathcal{C}([0,T], \mathbb{S}^d)$ satisfy \eqref{eq:PI_riccati},   $\Psi \in L^2([0,T], L^2(I, \mathbb{R}^d))$ and $\mu \in \mathcal{C}([0,T], L^2(I, \mathbb{R}^d))$. Then
    \begin{equation*}
        \begin{split}
            &\lVert \Lambda[\Psi, \mu^{(n)}] - \Lambda[\Psi, \mu] \rVert_{L_B^2}^2 \\
            &\quad\le 2 \sqrt{T} \left( \lvert \Bar{Q} \rvert \lvert \Bar{H} \rvert + T \left( \lVert P^* \rVert_{L^\infty} \lVert \Bar{A} \rVert_{L^\infty} + \lVert Q \rVert_{L^\infty} \lVert H \rVert_{L^\infty} \right) \right) \lVert W \rVert_{L^2(I^2)} \lVert \mu^{(n)} - \mu \rVert_{\mathcal{C}}.
        \end{split}
    \end{equation*}
\end{lemma}
\begin{proof}
    First, observe that
    \begin{equation}
    \label{eq:norm_relations_L2C}
        \lVert \mathbb{W} \rVert_{\mathrm{op}} \le \lVert W \rVert_{L^2(I^2)}, \quad \text{and} \quad \sup_{t \in [0,T]} \lVert \mathbb{W}[\mu_t^{(n)} - \mu_t] \rVert_{L^2(I)} \le \lVert W \rVert_{L^2(I^2)} \lVert \mu^{(n)} - \mu \rVert_{\mathcal{C}}.
    \end{equation}
    Moreover, by the Cauchy-Schwarz inequality,
    \begin{equation}
        \label{eq:cross_term_bound}
        \begin{split}
            \int_0^1 \bigg[ \lvert \mathbb{W}[\mu_T^{(n)} - \mu_T](\alpha) \rvert &\int_0^T \lvert \mathbb{W}[\mu_s^{(n)} - \mu_s](\alpha) \rvert \, ds \bigg] \, d \alpha \\
            &\le T \int_0^1 \sup_{t \in [0,T]} \left( \int_0^1 \lvert W(\alpha, \beta) \rvert \lvert \mu_t^{(n), \beta} - \mu_t^\beta \rvert \, d \beta \right)^2 \, d \alpha \\
            &\le T \lVert W \rVert_{L^2(I^2)}^2 \lVert \mu^{(n)} - \mu \rVert_{\mathcal{C}}^2.
        \end{split}
    \end{equation}
    Then for $t \in [0,T]$ using Jensen's inequality and Fubini's theorem,    
    \begin{equation*}
        \begin{split}
            &\int_0^1 \left\lvert \Lambda^{\alpha}[\Psi, \mu^{(n)}](t) - \Lambda^\alpha[\Psi, \mu](t) \right\rvert^2 \, d \alpha \\
            &= \int_0^1 \left\lvert 2 \Bar{Q} \Bar{H} \mathbb{W}[\mu_T^{(n)} - \mu_T](\alpha) - \int_t^T \left[ 2 (P_s^* \Bar{A} - Q H) \mathbb{W}[\mu_s^{(n)} - \mu_s](\alpha) \right] \, ds  \right\rvert^2 \, d \alpha \\
            &\le 4 \Bigg( \lvert \Bar{Q} \rvert^2 \lvert \Bar{H} \rvert^2 \sup_{t \in [0,T]} \lVert \mathbb{W}[\mu_t^{(n)} - \mu_t] \rVert_{L^2(I)}^2 \\
            &\quad+ 2 \lvert \Bar{Q} \rvert \lvert \Bar{H} \rvert \left( \lVert P^* \rVert_{L^\infty} \lVert \Bar{A} \rVert_{L^\infty} + \lVert Q \rVert_{L^\infty} \lVert H \rVert_{L^\infty} \right) \int_0^1 \left[ \lvert \mathbb{W}[\mu_T^{(n)} - \mu_T](\alpha) \rvert \int_0^T \lvert \mathbb{W}[\mu_s^{(n)} - \mu_s](\alpha) \rvert \, ds \right] \, d \alpha \\
            &\quad+ T^2 \left( \lVert P^* \rVert_{L^\infty} \lVert \Bar{A} \rVert_{L^\infty} + \lVert Q \rVert_{L^\infty} \lVert H \rVert_{L^\infty} \right)^2 \sup_{t \in [0,T]} \lVert \mathbb{W}[\mu_t^{(n)} - \mu_t] \rVert_{L^2(I)}^2 \, ds \Bigg) \\
            &\le 4 \left( \lvert \Bar{Q} \rvert \lvert \Bar{H} \rvert + T \left( \lVert P^* \rVert_{L^\infty} \lVert \Bar{A} \rVert_{L^\infty} + \lVert Q \rVert_{L^\infty} \lVert H \rVert_{L^\infty} \right) \right)^2 \lVert W \rVert_{L^2(I^2)}^2 \lVert \mu^{(n)} - \mu \rVert_{\mathcal{C}}^2,
        \end{split}
    \end{equation*}
    where the last inequality holds due to \eqref{eq:norm_relations_L2C} and \eqref{eq:cross_term_bound}.
    Hence, integrating with respect to $t$ yields the desired inequality,
    \begin{equation*}
        \begin{split}
           & \lVert \Lambda[\Psi, \mu^{(n)}] - \Lambda[\Psi, \mu] \rVert_{L_B^2}^2 \\
            &\quad\le 4 T \left( \lvert \Bar{Q} \rvert \lvert \Bar{H} \rvert + T \left( \lVert P^* \rVert_{L^\infty} \lVert \Bar{A} \rVert_{L^\infty} + \lVert Q \rVert_{L^\infty} \lVert H \rVert_{L^\infty} \right) \right)^2 \lVert W \rVert_{L^2(I^2)}^2 \lVert \mu^{(n)} - \mu \rVert_{\mathcal{C}}^2.
        \end{split}
    \end{equation*}
    Taking the square root completes the proof.
\end{proof}

\subsubsection{Error bounds for $\Psi$}

Let $\Psi^{(n+1),*, \alpha}$ satisfy \eqref{eq:localoptimal_PSI_MU_J2} for the graphon aggregate $Z^{(n), \alpha}$ fixed at the beginning of the $n$-th outer iteration for all $\alpha \in I$. In other words, $\Psi^{(n+1),*} = \Lambda[\Psi^{(n+1),*}, \mu^{(n)}]$. Additionally, define for all $\alpha \in I$ and $t \in [0,T]$,
\begin{equation}
\label{eq:ALLGs}
    \begin{split}
        &G_t^{(n+1),*, \alpha} = (K_t^* - K_t^{(n+1)}) \Hat{\mu}_t^{(n+1), *, \alpha} -\frac{1}{2} R^{-1} B^\top \Psi_t^{(n+1),*,\alpha} \text{,} \\
        &\Bar{G}_t^{(n+1),*, \alpha} = -\frac{1}{2} R^{-1} B^\top \Psi_t^{(n+1),*,\alpha},\quad \text{and} \quad G_t^{*, \alpha} = -\frac{1}{2} R^{-1} B^\top \Psi_t^{*,\alpha},
    \end{split}
\end{equation}
where $\Hat{\mu}^{(n+1), *} = \Phi^{(n)}[\Hat{\mu}^{(n+1), *}, K^{(n+1)}, G^{(n+1),*}]$. Hence, $\Hat{\mu}^{(n+1),*, \alpha}$ satisfies \eqref{eq:mean} with policy parameters $(K^{(n+1)}, G^{(n+1),*, \alpha})$ and fixed graphon aggregate $Z^{(n), \alpha}$ for all $\alpha \in I$. Moreover, $\Psi^{*} = \Lambda[\Psi^{*}, \mu^{*}]$ and $\mu^* = \Phi[\mu^*, K^*, G^*]$. 

We highlight the following observations: (i) $G^{(n+1),*}$ is the optimal intercept parameter for the optimization problem solved using the gradient descent scheme \eqref{eq:P2_graddescent} for fixed slope parameter $K^{(n+1)}$ and graphon aggregate $Z^{(n)}$; (ii) $\Bar{G}^{(n+1),*}$ is the optimal intercept parameter for the optimization problem, where the slope parameter is already optimal, but the graphon aggregate is not the NE graphon aggregate; (iii) $G^*$ is the NE intercept parameter.

\begin{lemma}
\label{lemma:BOUND_PSILOCALTOOPTIMAL}
Suppose Assumption \ref{assumption:convergence} holds. Let $n \in \mathbb{N}_0$. Let $(K^*,G^*)$ be the optimal policy parameters defined in \eqref{eq:systemcontrol}, let $P^* \in \mathcal{C}([0,T], \mathbb{S}^d)$ satisfy \eqref{eq:PI_riccati} and  $\mu^* = \Phi[\mu^*, K^*, G^*]$. Then for $\Psi^{(n+1),*} = \Lambda[\Psi^{(n+1),*}, \mu^{(n)}]$ and $\Psi^{*} = \Lambda[\Psi^{*}, \mu^{*}]$, 
    \begin{equation*}
        \begin{split}
            \lVert \Psi^{(n+1),*} &- \Psi^* \rVert_{L_B^2} \le \frac{2 \sqrt{T} \left( \lvert \Bar{Q} \rvert \lvert \Bar{H} \rvert + T \left( \lVert P^* \rVert_{L^\infty} \lVert \Bar{A} \rVert_{L^\infty} + \lVert Q \rVert_{L^\infty} \lVert H \rVert_{L^\infty} \right) \right) \lVert W \rVert_{L^2(I^2)}}{1 - T(\lVert A \rVert_{L^\infty} + \lVert B \rVert_{L^\infty} \lVert K^* \rVert_{L^\infty})} \lVert \mu^{(n)} - \mu^* \rVert_{\mathcal{C}}.
        \end{split}
    \end{equation*}
\end{lemma}

\begin{proof}
    Using the triangle inequality,
    \begin{equation*}
        \begin{split}
           & \lVert \Psi^{(n+1),*} - \Psi^* \rVert_{L_B^2} \\
            & \quad = \lVert \Lambda[\Psi^{(n+1),*}, \mu^{(n)}] - \Lambda[\Psi^*,\mu^*] \rVert_{L_B^2} \\
            &\quad \le \lVert \Lambda[\Psi^{(n+1),*}, \mu^{(n)}] - \Lambda[\Psi^*, \mu^{(n)}] \rVert_{L_B^2} + \lVert \Lambda[\Psi^*, \mu^{(n)}] - \Lambda[\Psi^*, \mu^*] \rVert_{L_B^2} \\
            &\quad \le T(\lVert A \rVert_{L^\infty} + \lVert B \rVert_{L^\infty} \lVert K^* \rVert_{L^\infty}) \lVert \Psi^{(n+1),*} - \Psi^* \rVert_{L_B^2} \\
            &\qquad+ 2 \sqrt{T} \left( \lvert \Bar{Q} \rvert \lvert \Bar{H} \rvert + T \left( \lVert P^* \rVert_{L^\infty} \lVert \Bar{A} \rVert_{L^\infty} + \lVert Q \rVert_{L^\infty} \lVert H \rVert_{L^\infty} \right) \right) \lVert W \rVert_{L^2(I^2)} \lVert \mu^{(n)} - \mu^* \rVert_{\mathcal{C}},
        \end{split}
    \end{equation*}
    where the last inequality holds because of Lemmas \ref{lemma:PERTURB_PSI} and \ref{lemma:LAMBDA_nvsexact}. Re-arranging the inequality under Assumption \ref{assumption:convergence} completes the proof. 
\end{proof}

\subsubsection{Error bounds for $G$}

The next two results analyze the error bounds which are necessary to establish the convergence for the intercept parameter utilizing the result for $\Psi$ above. 

\begin{lemma}
\label{lemma:BOUND_GUPDATETOLOCAL}
Let $K^*$ be the optimal slope parameter defined in \eqref{eq:systemcontrol}. For $n \in \mathbb{N}_0$ let $K^{(n+1)}$ be obtained through the gradient descent scheme in \eqref{eq:P1_graddescent}. Let $\Hat{\mu}^{(n+1), *} = \Phi^{(n)}[\Hat{\mu}^{(n+1), *}, K^{(n+1)}, G^{(n+1),*}]$. Then for $G^{(n+1),*}$ and $\Bar{G}^{(n+1),*}$ defined in \eqref{eq:ALLGs},
    \begin{equation*}
        \lVert G^{(n+1),*} - \Bar{G}^{(n+1),*} \rVert_{L_B^2} \le \lVert \Hat{\mu}^{(n+1),*} \rVert_{\mathcal{C}} \lVert K^{(n+1)} - K^* \rVert_{L^2([0,T])}.
    \end{equation*}
\end{lemma}

\begin{proof}
    The definition of $G^{(n+1),*}$ and $\Bar{G}^{(n+1),*}$ in \eqref{eq:ALLGs} immediately implies
    \begin{equation*}
        \begin{split}
            \lVert G^{(n+1),*} - \Bar{G}^{(n+1),*} \rVert_{L_B^2}^2 &= \int_0^T \int_0^1 \lvert K_t^* - K_t^{(n+1)} \rvert^2 \lvert \Hat{\mu}_t^{(n+1),*} \rvert^2 \, d\alpha \, dt \\
            &\le \lVert K^{(n+1)} - K^* \rVert_{L^2([0,T])}^2 \sup\limits_{t \in [0,T]}{\lVert \Hat{\mu}_t^{(n+1),*} \rVert_{L^2(I)}^2 }.
        \end{split}
    \end{equation*}
    Taking the square root completes the proof. 
\end{proof}

The following lemma immediately follows from \eqref{eq:ALLGs} and Lemma \ref{lemma:BOUND_PSILOCALTOOPTIMAL}.

\begin{lemma}
\label{lemma:BOUND_GLOCALTOPOPTIMAL}
    Suppose Assumption \ref{assumption:convergence} holds. Let $(K^*, G^*)$ be the optimal policy parameters defined in \eqref{eq:systemcontrol}, let $P^* \in \mathcal{C}([0,T], \mathbb{S}^d)$ satisfy \eqref{eq:PI_riccati} and let $\mu^* = \Phi[\mu^*, K^*, G^*]$. Then for $\Bar{G}^{(n+1),*}$ defined in \eqref{eq:ALLGs} the following holds
    \begin{equation*}
        \begin{split}
           & \lVert \Bar{G}^{(n+1),*} - G^* \rVert_{L_B^2} \\
            &\quad \le \frac{\lVert B \rVert_{L^\infty} \sqrt{T} \left( \lvert \Bar{Q} \rvert \lvert \Bar{H} \rvert + T \left( \lVert P^* \rVert_{L^\infty} \lVert \Bar{A} \rVert_{L^\infty} + \lVert Q \rVert_{L^\infty} \lVert H \rVert_{L^\infty} \right) \right) \lVert W \rVert_{L^2(I^2)}}{\underline{\lambda}^R (1 - T(\lVert A \rVert_{L^\infty} + \lVert B \rVert_{L^\infty} \lVert K^* \rVert_{L^\infty}))} \lVert \mu^{(n)} - \mu^* \rVert_{\mathcal{C}}.
        \end{split}
    \end{equation*}
\end{lemma}

\subsubsection{Perturbation analysis for $\Phi$}

The following lemma considers the perturbation of $\Phi$ in $\mu$.

\begin{lemma}
\label{lemma:FPI_AGGREGATE}
    Let $K \in L^\infty([0,T], \mathbb{R}^{k \times d})$ and $G \in L^2([0,T], L^2([0,T], \mathbb{R}^k))$. Then for $\mu, \mu' \in  \mathcal{C}([0,T], L^2(I, \mathbb{R}^d))$ with $\mu_0^\alpha = \mu_0'^{,\alpha}$ for all $\alpha \in I$,
    \begin{equation*}
        \lVert \Phi[\mu', K, G] - \Phi[\mu, K, G] \rVert_{\mathcal{C}} \le T(\lVert A \rVert_{L^\infty} + \lVert B \rVert_{L^\infty} \lVert K \rVert_{L^\infty} + \lVert \Bar{A} \rVert_{L^\infty} \lVert W \rVert_{L^2(I^2)}) \lVert \mu' - \mu \rVert_{\mathcal{C}}.
    \end{equation*}
\end{lemma}

\begin{proof}
    First, note that by Fubini's theorem, the Cauchy-Schwarz inequality and \eqref{eq:norm_relations_L2C},
    \begin{equation*}
        \begin{split}
          &  \int_0^1 \int_0^T \lvert \mu_s'^{, \alpha} - \mu_s^\alpha \rvert  \lvert \mathbb{W} [\mu_s' - \mu_s](\alpha) \rvert \, ds \, d \alpha \le \int_0^T \left[ \lVert \mu_s' - \mu_s \rVert_{L^2(I)}^2 \lVert \mathbb{W}[\mu_s' - \mu_s] \rVert_{L^2(I)}^2 \right]^{1/2} \, ds \\
            &\quad \le \lVert W \rVert_{L^2(I^2)} \int_0^T \lVert \mu_s' - \mu_s \rVert_{L^2(I)}^2 \, ds  \le T \lVert W \rVert_{L^2(I^2)} \lVert \mu' - \mu \rVert_{\mathcal{C}}^2.
        \end{split}
    \end{equation*}
    Then for any $t \in [0,T]$,  using Jensen's inequality, Fubini's theorem and \eqref{eq:norm_relations_L2C} shows that 
    \begin{equation*}
        \begin{split}
         &    \int_0^1 \big\lvert \Phi^\alpha[\mu', K, G](t) - \Phi^\alpha[\mu, K, G](t) \big\rvert^2 \, d \alpha \\
            & \quad = \int_0^1 \left\lvert \int_0^t \left[ (A + B K_s) ( \mu_s'^{, \alpha} - \mu_s^\alpha ) + \Bar{A} \mathbb{W}[\mu_s' - \mu_s](\alpha) \, \right] \, ds \right\rvert^2 \, d \alpha \\
            &\quad \le T \int_0^1 \int_0^T \left\vert ( A  +  B  K_s) ( \mu_s'^{, \alpha} - \mu_s^\alpha ) + \Bar{A} \mathbb{W}[\mu_s' - \mu_s](\alpha) \right\rvert^2 \, ds \, d \alpha \\
            &\quad \le T \Bigg( (\lVert A \rVert_{L^\infty} + \lVert B \rVert_{L^\infty} \lVert K \rVert_{L^\infty} )^2 T \lVert \mu' - \mu \rVert_{\mathcal{C}}^2 \\
            &\qquad+ 2 (\lVert A \rVert_{L^\infty} + \lVert B \rVert_{L^\infty} \lVert K \rVert_{L^\infty}) \lVert \Bar{A} \rVert_{L^\infty} \int_0^1 \int_0^T \lvert \mu_s'^{, \alpha} - \mu_s^\alpha \rvert  \lvert \mathbb{W} [\mu_s' - \mu_s](\alpha) \rvert \, ds \, d \alpha \\
            &\qquad+ \lVert \Bar{A} \rVert_{L^\infty}^2 T \sup_{t \in [0,T]} \lVert \mathbb{W}[\mu_t' - \mu_t] \rVert_{L^2(I)}^2 \Bigg) \\
            &\quad \le T^2 \left( \lVert A \rVert_{L^\infty} + \lVert B \rVert_{L^\infty} \lVert K \rVert_{L^\infty} + \lVert \Bar{A} \rVert_{L^\infty} \lVert W \rVert_{L^2(I^2)} \right)^2 \lVert \mu' - \mu \rVert_{\mathcal{C}}^2.
        \end{split}
    \end{equation*}
    Taking the supremum with respect to $t$ yields 
    \begin{equation*}
        \begin{split}
            \lVert \Phi[\mu', K, G] - \Phi[\mu, K, G] \rVert_{\mathcal{C}}^2 &\le T^2 (\lVert A \rVert_{L^\infty} + \lVert B \rVert_{L^\infty} \lVert K \rVert_{L^\infty} + \lVert \Bar{A} \rVert_{L^\infty} \lVert W \rVert_{L^2(I^2)})^2 \lVert \mu' - \mu \rVert_{\mathcal{C}}^2.
        \end{split}
    \end{equation*}
    Taking the square root completes the proof.
\end{proof}

The following lemma considers the perturbation of $\Phi$ in $K$.

\begin{lemma}
\label{lemma:MU_PERTURB_K}
    Let $\mu \in  \mathcal{C}([0,T], L^2(I, \mathbb{R}^d))$ and $G \in  L^2([0,T], L^2(I, \mathbb{R}^k))$. Then for any $K, K' \in L^2([0,T], \mathbb{R}^{k \times d})$,
    \begin{equation*}
        \lVert \Phi[\mu, K, G] - \Phi[\mu, K', G] \rVert_{\mathcal{C}} \le \sqrt{T} \lVert B \rVert_{L^\infty} \lVert \mu \rVert_{\mathcal{C}} \lVert K - K' \rVert_{L^2([0,T])}.
    \end{equation*}
\end{lemma}

\begin{proof}
The Cauchy-Schwarz inequality and Fubini's theorem yield for any $K, K' \in L^2([0,T], \mathbb{R}^{k \times d})$,
    \begin{equation*}
        \begin{split}
            \int_0^1 \big\lvert \Phi^\alpha[\mu, K, G] - \Phi^\alpha[\mu, K', G] \big\rvert^2 \, d\alpha &= \int_0^1 \left\lvert \int_0^t B (K_s - K_s') \mu_s^{\alpha} \, ds \right\rvert^2 \, ds \, d\alpha \\
            &\le \lVert B \rVert_{L^\infty}^2 \lVert K - K' \rVert_{L^2([0,T])}^2 \int_0^T \int_0^1 \lvert \mu_s^{\alpha} \rvert^2 \, d\alpha \, ds \\
            &\le T \lVert B \rVert_{L^\infty}^2 \lVert K - K' \rVert_{L^2([0,T])}^2 \lVert \mu \rVert_{\mathcal{C}}^2,
        \end{split}
    \end{equation*}    
    where the supremum with respect to $t$ is taken in the last step.
    Taking the square root completes the proof.
\end{proof}

The next lemma analyzes the perturbation of $\Phi$ in $G$.

\begin{lemma}
\label{lemma:MU_PERTURB_G}
    Let $\mu \in \mathcal{C}([0,T], L^2(I, \mathbb{R}^d))$ and $K \in L^2([0,T], \mathbb{R}^{k \times d})$. Then for any $G, G' \in L^2([0,T], L^2(I, \mathbb{R}^k))$,
    \begin{equation*}
        \lVert \Phi[\mu, K, G] - \Phi[\mu, K, G'] \rVert_{\mathcal{C}} \le \sqrt{T} \lVert B \rVert_{L^\infty} \lVert G - G' \rVert_{L_B^2}.
    \end{equation*}
\end{lemma}

\begin{proof}
    For any $G, G' \in L^2([0,T], L^2(I, \mathbb{R}^k))$, Jensen's inequality and Fubini's theorem yield
    \begin{equation*}
        \begin{split}
            \int_0^1 \big\lvert \Phi^\alpha[\mu, K, G] - \Phi^\alpha[\mu, K, G'] \big\rvert^2 \, d\alpha &= \int_0^1 \left\lvert \int_0^t B (G_s^\alpha - G_s'^{,\alpha} ) \, ds \right\rvert^2 \, d\alpha  \le T \lVert B \rVert_{L^\infty}^2 \lVert G - G' \rVert_{L_B^2}^2.
        \end{split}
    \end{equation*}
    Hence, taking the supremum with respect to $t$ gives
    \begin{equation*}
        \lVert \Phi[\mu, K, G] - \Phi[\mu, K, G'] \rVert_{\mathcal{C}}^2 \le T \lVert B \rVert_{L^\infty}^2 \lVert G - G' \rVert_{L_B^2}^2.
    \end{equation*}
    Taking the square root concludes the proof.
\end{proof}

\subsubsection{Error bounds for $\mu$}
We first introduce $\Hat{\mu}_E^{(n+1),*} = \Phi[\Hat{\mu}_E^{(n+1),*}, K^{(n+1)}, G^{(n+1),*}]$ and $\Bar{\mu}_E^{(n+1),*} = \Phi[\Bar{\mu}_E^{(n+1),*}, K^{*}, \Bar{G}^{(n+1),*}]$. The following lemma analyzes the error bound for the state mean dynamics defined with the intercept parameter obtained through the gradient descent scheme after a certain number of iterations (i.e., $\mu^{(n+1),*}$) and state mean dynamics defined with the optimal intercept parameter for fixed slope and graphon aggregate (i.e., $\Hat{\mu}_E^{(n+1),*}$).

\begin{lemma}
\label{lemma:BOUND_MUFPITOLOCAL}
    Suppose Assumption \ref{assumption:convergence} holds. For $n \in \mathbb{N}_0$ let $K^{(n+1)}$ be obtained through the gradient descent scheme in \eqref{eq:P1_graddescent} and $G^{(n+1)}$ through \eqref{eq:P2_graddescent} with fixed slope $K^{(n+1)}$. Moreover, let $G^{(n+1),*}$ be defined in \eqref{eq:ALLGs}. Then for $\mu^{(n+1),*} = \Phi[\mu^{(n+1),*}, K^{(n+1)}, G^{(n+1)}]$ and $\Hat{\mu}_E^{(n+1),*} = \Phi[\Hat{\mu}_E^{(n+1),*}, K^{(n+1)}, G^{(n+1),*}]$ the following holds
    \begin{equation*}
        \lVert \mu^{(n+1), *} - \Hat{\mu}_E^{(n+1), *} \rVert_{\mathcal{C}} \le \frac{\sqrt{T} \lVert B \rVert_{L^\infty}}{1 - M_1}\lVert G^{(n+1)} - G^{(n+1),*} \rVert_{L_B^2}.
    \end{equation*}
\end{lemma}

\begin{proof}
    By the triangle inequality,
    \begin{equation*}
        \begin{split}
            \lVert \mu^{(n+1), *} - \Hat{\mu}_E^{(n+1), *} \rVert_{\mathcal{C}} &= \lVert \Phi[\mu^{(n+1),*}, K^{(n+1)}, G^{(n+1)}] - \Phi[\Hat{\mu}_E^{(n+1),*}, K^{(n+1)}, G^{(n+1),*}] \rVert_{\mathcal{C}} \\
            &\le \lVert \Phi[\mu^{(n+1),*}, K^{(n+1)}, G^{(n+1)}] - \Phi[\Hat{\mu}_E^{(n+1),*}, K^{(n+1)}, G^{(n+1)}] \rVert_{\mathcal{C}} \\
            &\quad+ \lVert \Phi[\Hat{\mu}_E^{(n+1),*}, K^{(n+1)}, G^{(n+1)}] - \Phi[\Hat{\mu}_E^{(n+1),*}, K^{(n+1)}, G^{(n+1),*}] \rVert_{\mathcal{C}} \\
            &\le M_1 \lVert \mu^{(n+1), *} - \Hat{\mu}_E^{(n+1), *} \rVert_{\mathcal{C}} + \sqrt{T} \lVert B \rVert_{L^\infty} \lVert G^{(n+1)} - G^{(n+1),*} \rVert_{L_B^2},
        \end{split}
    \end{equation*}
    where the last inequality holds due to Lemma \ref{lemma:FPI_AGGREGATE} combined with the uniform bound $C_0^K$ and Lemma \ref{lemma:MU_PERTURB_G}.
    Re-arranging this inequality under Assumption \ref{assumption:convergence} completes the proof.
\end{proof}

The next lemma provides an error bound for the state mean dynamics defined with the optimal intercept parameter for fixed slope and graphon aggregate (i.e., $\Hat{\mu}_E^{(n+1),*}$) and the state mean dynamics defined with the optimal slope parameter, but where the intercept parameter is based on some fixed graphon aggregate (i.e., $\Bar{\mu}_E^{(n+1),*}$).

\begin{lemma}
\label{lemma:BOUND_HATTOBAR}
    Suppose Assumption \ref{assumption:convergence} holds. Let $K^*$ be the optimal slope parameter defined in \eqref{eq:systemcontrol}. For $n \in \mathbb{N}_0$ let $K^{(n+1)}$ be obtained through the gradient descent scheme in \eqref{eq:P1_graddescent}. Moreover, let $G^{(n+1),*}$ and $\Bar{G}^{(n+1),*}$ be defined in \eqref{eq:ALLGs}. Then for $\Hat{\mu}_E^{(n+1),*} = \Phi[\Hat{\mu}_E^{(n+1),*}, K^{(n+1)}, G^{(n+1),*}]$ and $\Bar{\mu}_E^{(n+1),*} = \Phi[\Bar{\mu}_E^{(n+1),*}, K^{*}, \Bar{G}^{(n+1),*}]$ the following holds
    \begin{equation*}
        \lVert \Hat{\mu}_E^{(n+1),*} - \Bar{\mu}_E^{(n+1),*} \rVert_{\mathcal{C}} \le \frac{\lVert B \rVert_{L^\infty} \sqrt{T}}{1 - M_1} \left(  \lVert \Bar{\mu}_E^{(n+1),*} \rVert_{\mathcal{C}} \lVert K^{(n+1)} - K^* \rVert_{L^2([0,T])} + \lVert G^{(n+1),*} - \Bar{G}^{(n+1),*} \rVert_{L_B^2} \right).
    \end{equation*}
\end{lemma}

\begin{proof}
    By the triangle inequality,
    \begin{equation}
    \label{eq:MEAN_TRIANGLE_firstineq}
        \begin{split}
            \lVert \Hat{\mu}_E^{(n+1), *} - \Bar{\mu}_E^{(n+1),*} \rVert_{\mathcal{C}} &= \lVert \Phi[\Hat{\mu}_E^{(n+1), *}, K^{(n+1)}, G^{(n+1), *}] - \Phi[\Bar{\mu}_E^{(n+1), *}, K^{*}, \Bar{G}^{(n+1), *}] \rVert_{\mathcal{C}} \\
            &\le \lVert \Phi[\Hat{\mu}_E^{(n+1), *}, K^{(n+1)}, G^{(n+1), *}] - \Phi[\Bar{\mu}_E^{(n+1), *}, K^{(n+1)}, G^{(n+1), *}] \rVert_{\mathcal{C}} \\
            &\quad+ \lVert \Phi[\Bar{\mu}_E^{(n+1), *}, K^{(n+1)}, G^{(n+1), *}] - \Phi[\Bar{\mu}_E^{(n+1), *}, K^{*}, G^{(n+1), *}] \rVert_{\mathcal{C}} \\
            &\quad+ \lVert \Phi[\Bar{\mu}_E^{(n+1), *}, K^{*}, G^{(n+1), *}] - \Phi[\Bar{\mu}_E^{(n+1), *}, K^{*}, \Bar{G}^{(n+1), *}] \rVert_{\mathcal{C}}.
        \end{split}
    \end{equation}
    By Lemma \ref{lemma:FPI_AGGREGATE} combined with the uniform bound $C_0^K$, we upper bound  the first term by 
    \begin{equation*}
        \begin{split}
            \lVert \Phi[\Hat{\mu}_E^{(n+1), *}, K^{(n+1)}, G^{(n+1), *}] &- \Phi[\Bar{\mu}_E^{(n+1), *}, K^{(n+1)}, G^{(n+1), *}] \rVert_{\mathcal{C}} \\
            &\le M_1 \lVert \Hat{\mu}_E^{(n+1), *} - \Bar{\mu}_E^{(n+1), *} \rVert_{\mathcal{C}}.
        \end{split}
    \end{equation*}
  Using  Lemma \ref{lemma:MU_PERTURB_K}, the second term is upper bounded by  
    \begin{equation*}
        \begin{split}
            & \lVert \Phi[\Bar{\mu}_E^{(n+1), *}, K^{(n+1)}, G^{(n+1), *}] - \Phi[\Bar{\mu}_E^{(n+1), *}, K^{*}, G^{(n+1), *}] \rVert_{\mathcal{C}} \\
             &\quad \le \sqrt{T} \lVert B \rVert_{L^\infty} \lVert \Bar{\mu}_E^{(n+1),*} \rVert_{\mathcal{C}}  \lVert K^{(n+1)} - K^* \rVert_{L^2([0,T])}.
        \end{split}
    \end{equation*}
    Lemma \ref{lemma:MU_PERTURB_G} implies the following bound for the third term,
    \begin{equation*}
        \lVert \Phi[\Bar{\mu}_E^{(n+1), *}, K^{*}, G^{(n+1), *}] - \Phi[\Bar{\mu}_E^{(n+1), *}, K^{*}, \Bar{G}^{(n+1), *}] \rVert_{\mathcal{C}} \le \sqrt{T} \lVert B \rVert_{L^\infty} \lVert G^{(n+1),*} - \Bar{G}^{(n+1),*} \rVert_{L_B^2}.
    \end{equation*}
    Using those three inequalities in \eqref{eq:MEAN_TRIANGLE_firstineq} and re-arranging the resulting inequality under Assumption \ref{assumption:convergence} completes the proof.
\end{proof}

Finally, the next lemma provides an error bound to the NE population state mean $\mu^*$.

\begin{lemma}
\label{lemma:BOUND_MULOCALTOOPTIMAL}
    Suppose Assumption \ref{assumption:convergence} holds. For $n \in \mathbb{N}_0$ let $\Bar{G}^{(n+1),*}$ be defined in \eqref{eq:ALLGs}. Let $(K^*,G^*)$ be the optimal policy parameters defined in \eqref{eq:systemcontrol}. Then for $\Bar{\mu}_E^{(n+1),*} = \Phi[\Bar{\mu}_E^{(n+1),*}, K^{*}, \Bar{G}^{(n+1),*}]$ and $\mu^* = \Phi[\mu^*, K^*, G^*]$ the following holds
    \begin{equation*}
        \lVert \Bar{\mu}_E^{(n+1),*} - \mu^* \rVert_{\mathcal{C}} \le \frac{\sqrt{T} \lVert B \rVert_{L^\infty}}{1 -T(\lVert A \rVert_{L^\infty} + \lVert B \rVert_{L^\infty} \lVert K^* \rVert_{L^\infty} + \lVert \Bar{A} \rVert_{L^\infty} \lVert W \rVert_{L^2(I^2)})} \lVert \Bar{G}^{(n+1),*} - G^* \rVert_{L_B^2}.
    \end{equation*}
\end{lemma}

\begin{proof}
    By the triangle inequality,
    \begin{equation}
        \begin{split}
        \label{eq:MEAN_TRIANGLE_secondineq}
            \lVert \Bar{\mu}_E^{(n+1),*} - \mu^* \rVert_{\mathcal{C}} &= \lVert \Phi[\Bar{\mu}_E^{(n+1),*}, K^*, \Bar{G}^{(n+1),*}] - \Phi[\mu^*, K^*, G^*] \rVert_{\mathcal{C}} \\
            &\le \lVert \Phi[\Bar{\mu}_E^{(n+1),*}, K^*, \Bar{G}^{(n+1),*}] - \Phi[\mu^*, K^{*}, \Bar{G}^{(n+1),*}] \rVert_{\mathcal{C}} \\
            &\quad+\lVert \Phi[\mu^{*}, K^{*}, \Bar{G}^{(n+1),*}] - \Phi[\mu^*, K^*, G^*] \rVert_{\mathcal{C}}.
        \end{split}
    \end{equation}
    Then, we have by Lemma \ref{lemma:FPI_AGGREGATE} for the first term,
    \begin{equation*}
        \begin{split}
            \lVert \Phi[\Bar{\mu}_E^{(n+1),*}, K^*, \Bar{G}^{(n+1),*}] &- \Phi[\mu^*, K^*, \Bar{G}^{(n+1),*}] \rVert_{\mathcal{C}} \\
            &\le T(\lVert A \rVert_{L^\infty} + \lVert B \rVert_{L^\infty} \lVert K^* \rVert_{L^\infty} + \lVert \Bar{A} \rVert_{L^\infty} \lVert W \rVert_{L^2(I^2)}) \lVert \Bar{\mu}_E^{(n+1),*} - \mu^* \rVert_{\mathcal{C}},
        \end{split}
    \end{equation*}
    and by
    Lemma \ref{lemma:MU_PERTURB_G} for the second term,
    \begin{equation*}
        \lVert \Phi[\mu^{*}, K^{*}, \Bar{G}^{(n+1),*}] - \Phi[\mu^*, K^*, G^*] \rVert_{\mathcal{C}} \le \sqrt{T} \lVert B \rVert_{L^\infty} \lVert \Bar{G}^{(n+1),*} - G^* \rVert_{L_B^2}.
    \end{equation*}
    Finally, applying those two inequalities to \eqref{eq:MEAN_TRIANGLE_secondineq} and re-arranging the resulting inequality under Assumption \ref{assumption:convergence} finishes the proof.
\end{proof}

\subsubsection{Proof of Theorem \ref{theorem:CONVERGENCE_FINAL}}

Before proving Theorem \ref{theorem:CONVERGENCE_FINAL}, we specify all necessary constants as follows. For $n \in \{0, 1, \dots, N - 1 \}$, where $N \in \mathbb{N}$ is the number of outer iterations, let the numbers of gradient descent iterations  $L_n^K , L_n^G \in \mathbb{N}_0$ satisfy 
\begin{equation*}
    \begin{split}
        &L_n^K > - \frac{2}{\log{\left(1 - \eta_K C_2^K \right)}} \log{\left( \frac{ 3 \lVert J_1^\cdot(K^{(n)}) - J_1^\cdot(K^*) \rVert_{L^2(I)}^{1/2} \sqrt{C_3^K} C_{\Hat{\mu}, \Bar{\mu}_E}^{(n+1),*} }{\omega_n} \right)}, \\
        &L_n^G > -\frac{2}{\log{(1 - \eta_G m)}} \log{\left( \frac{3 T \lVert B \rVert_{L^\infty} \lVert G^{(n)} - G^{(n+1),*} \rVert_{L_B^2}}{(1 - M_1) \omega_n} \right)},
    \end{split}
\end{equation*}
where
    \begin{equation*}
        \begin{split}
            &C_{\Hat{\mu}}^{(n+1),*} \coloneqq \lVert \Hat{\mu}^{(n+1),*} \rVert_{\mathcal{C}}, \quad C_{\Hat{\mu}, \Bar{\mu}_E}^{(n+1),*} \coloneqq \frac{\sqrt{T} \lVert B \rVert_{L^\infty}}{1 - M_1} \left( \lVert \Bar{\mu}_E^{(n+1),*} \rVert_{\mathcal{C}} + C_{\Hat{\mu}}^{(n+1),*} \right).
        \end{split}    
    \end{equation*}
    Moreover, for the last iteration, additionally impose
    \begin{equation*}
        \begin{split}
            &L_{N-1}^K > - \frac{2}{\log{{\left(1 - \eta_K C_2^K \right)}}}  \max\Bigg\{ \log{\left( \frac{ \lVert J_1^\cdot(K^{(N-1)}) - J_1^\cdot(K^*) \rVert_{L^2(I)}^{1/2} \sqrt{C_3^K}}{\omega_{N-1}} \right)}, \\
            &\quad\quad\quad\quad\quad\log{\left( \frac{\lVert J_1^\cdot(K^{(N-1)}) - J_1^\cdot(K^*)\rVert_{L^2(I)}^{1/2} \sqrt{C_3^K} C_{\Hat{\mu}}^{(N),*}}{\omega_{N-1}} \right)} \Bigg\}, \\
            &L_{N-1}^G > -\frac{2}{\log{(1 - \eta_G m)}} \log{\left( \frac{\lVert G^{(N-1)} - G^{(N),*} \rVert_{L_B^2}}{ \omega_{N-1}} \right)}.
        \end{split}
    \end{equation*}
Recall that     $M_\vartheta$ and $M^\vartheta$ are the covariance bounds from Lemma \ref{lemma:BOUND_VAR} and $C_0^K$ is the uniform upper bound for the slope parameter from Proposition \ref{prop:BOUND_K}. $\eta_K$ and $\eta_G$ are the stepsizes for the gradient descent scheme in $K$ and $G$, respectively (see Theorems \ref{theorem:K_CONVERGENCE} and \ref{theorem:G_CONVERGENCE}). $m$ is the strong convexity parameter from Lemma $\ref{lemma:J22_stronglyconvex}$ and $C_1^K, C_3^K$ are defined in \eqref{eq:constant_convergence_K}. $M_1$ and $M_2$ are introduced in Assumption \ref{assumption:convergence}. Furthermore, define
    \begin{equation}
    \label{eq:M_G}
        \begin{split}
            M_G &\coloneqq 1 + \frac{M_2 (1 - T (\lVert A \rVert_{L^\infty} + \lVert B \rVert_{L^\infty} \lVert K^* \rVert_{L^\infty} + \lVert \Bar{A} \rVert_{L^\infty} \lVert W \rVert_{L^2(I^2)}))}{\sqrt{T} \lVert B \rVert_{L^\infty}} \\
            &= 1 + \frac{\sqrt{T} \lVert B \rVert_{L^\infty} \left( \lvert \Bar{Q} \rvert \lvert \Bar{H} \rvert + T \left( \lVert P^* \rVert_{L^\infty} \lVert \Bar{A} \rVert_{L^\infty} + \lVert Q \rVert_{L^\infty} \lVert H \rVert_{L^\infty} \right) \right) \lVert W \rVert_{L^2(I^2)}}{\underline{\lambda}^R (1 - T (\lVert A \rVert_{L^\infty} + \lVert B \rVert_{L^\infty} \lVert K^* \rVert_{L^\infty} + \lVert \Bar{A} \rVert_{L^\infty} \lVert W \rVert_{L^2(I^2)}))}.
        \end{split} 
    \end{equation}

\begin{remark}[\textbf{Bounds for number of inner iterations}]
    A more conservative bound for $L_n^K$ can be found by noting that $J_1^\alpha(K^{(n)}) - J_1^\alpha(K^*) \le J_1^\alpha(K^{(n)}) \le J_1^\alpha(K^{(0)})$ for all $\alpha \in I$. Moreover, for $L_n^G$ due to the strong convexity, $\lVert G^{(n)} - G^{(n+1), *} \rVert_{L_B^2} \le (2/m)^{1/4} \lVert J_2^\cdot(G^{(n),\cdot}) - J_2^\cdot(G^{(n+1),*,\cdot}) \rVert_{L^2(I)}^{1/2}$ and $J_2^\alpha(G^{(n), \alpha}) - J_2^\alpha(G^{(n+1),*,\alpha}) \le J_2^\alpha(G^{(n), \alpha})$ for all $\alpha \in I$. However, this monotonicity conditions cannot be extended to the initial cost functional $J_2^\alpha(G^{(0), \alpha})$ due to the graphon aggregate updates. $C_{\Hat{\mu}}^{(n+1),*}$ and $C_{\Hat{\mu}, \Bar{\mu}_E}^{(n+1),*}$ can be upper bounded in terms of $C_0^K$, $P^*$ and $Z^{(n)}$ using Proposition \ref{proposition:optimG_fixKZ} and Gronwall's inequality.
\end{remark}
    
\begin{proof}[Proof of Theorem \ref{theorem:CONVERGENCE_FINAL}]
    Let $n \in \{0,1, \dots, N-1 \}$. By the triangle inequality, 
    \begin{equation*}
        \begin{split}
            \lVert \mu^{(n+1)} - \mu^* \rVert_{\mathcal{C}} &\le \lVert \mu^{(n+1)} - \mu^{(n+1),*} \rVert_{\mathcal{C}} + \lVert \mu^{(n+1),*} - \Hat{\mu}_E^{(n+1),*} \rVert_{\mathcal{C}} \\
            &\quad+ \lVert \Hat{\mu}_E^{(n+1),*} - \Bar{\mu}_E^{(n+1),*} \rVert_{\mathcal{C}} + \lVert \Bar{\mu}_E^{(n+1),*} - \mu^* \rVert_{\mathcal{C}}.
        \end{split}
    \end{equation*}
    Hence, by \eqref{eq:mu_estimator_assumption} and Lemmas \ref{lemma:BOUND_MUFPITOLOCAL}, \ref{lemma:BOUND_HATTOBAR} and \ref{lemma:BOUND_MULOCALTOOPTIMAL},
    \begin{equation}
        \label{eq:conv_proof_mu_1}
        \begin{split}
            &\lVert \mu^{(n+1)}  - \mu^* \rVert_{\mathcal{C}} \\
            &\quad \le \frac{\omega_n}{3} + \frac{\sqrt{T} \lVert B \rVert_{L^\infty}}{1 - M_1} \left( \lVert G^{(n+1)} - G^{(n+1),*} \rVert_{L_B^2} + \lVert G^{(n+1),*} - \Bar{G}^{(n+1),*} \rVert_{L_B^2} \right) \\
            &\qquad+ \frac{\sqrt{T} \lVert B \rVert_{L^\infty}}{1 - T(\lVert A \rVert_{L^\infty} + \lVert B \rVert_{L^\infty} \lVert K^* \rVert_{L^\infty} + \lVert \Bar{A} \rVert_{L^\infty} \lVert W \rVert_{L^2(I^2)})} \lVert \Bar{G}^{(n+1),*} - G^* \rVert_{L_B^2} \\
            &\qquad+ \frac{\sqrt{T} \lVert B \rVert_{L^\infty}}{1 - M_1} \lVert \Bar{\mu}_E^{(n+1),*} \rVert_{\mathcal{C}} \lVert K^{(n+1)} - K^* \rVert_{L^2([0,T])}.
        \end{split}
    \end{equation}
    Recall the error bound from Lemma \ref{lemma:BOUND_GUPDATETOLOCAL},
    \begin{equation}
        \label{eq:bound_G_GD_tolocaloptim_final}
        \lVert G^{(n+1),*} - \Bar{G}^{(n+1),*} \rVert_{L_B^2} \le C_{\Hat{\mu}}^{(n+1),*} \lVert K^{(n+1)} - K^* \rVert_{L^2([0,T])}.
    \end{equation}
    Additionally, by Lemma \ref{lemma:BOUND_GLOCALTOPOPTIMAL},
    \begin{equation}
        \begin{split}
            &\lVert \Bar{G}^{(n+1),*} - G^* \rVert_{L_B^2} \\
            &\le \frac{\lVert B \rVert_{L^\infty} \sqrt{T} \left( \lvert \Bar{Q} \rvert \lvert \Bar{H} \rvert + T \left( \lVert P^* \rVert_{L^\infty} \lVert \Bar{A} \rVert_{L^\infty} + \lVert Q \rVert_{L^\infty} \lVert H \rVert_{L^\infty} \right) \right) \lVert W \rVert_{L^2(I^2)}}{\underline{\lambda}^R (1 - T(\lVert A \rVert_{L^\infty} + \lVert B \rVert_{L^\infty} \lVert K^* \rVert_{L^\infty}))} \lVert \mu^{(n)} - \mu^* \rVert_{\mathcal{C}}.
        \end{split}
        \label{eq:bound_G_local_optim_final}
    \end{equation}
    Hence, using \eqref{eq:bound_G_GD_tolocaloptim_final} and \eqref{eq:bound_G_local_optim_final} in \eqref{eq:conv_proof_mu_1} yields
    \begin{equation}
        \label{eq:conv_proof_mu_2}
        \begin{split}
            &\lVert \mu^{(n+1)} - \mu^* \rVert_{\mathcal{C}} \\
            &\le \frac{\omega_n}{3} + \frac{\sqrt{T} \lVert B \rVert_{L^\infty}}{1 - M_1} \lVert G^{(n+1)} - G^{(n+1),*} \rVert_{L_B^2} \\
            &\quad+ \frac{\sqrt{T} \lVert B \rVert_{L^\infty} }{1 - M_1} \left( \lVert \Bar{\mu}_E^{(n+1),*} \rVert_{L_B^2} + C_{\Hat{\mu}}^{(n+1),*} \right) \lVert K^{(n+1)} - K^* \rVert_{L^2([0,T])} \\
            &\quad+ \frac{T \lVert B \rVert_{L^\infty}^2 \left( \lvert \Bar{Q} \rvert \lvert \Bar{H} \rvert + T \left( \lVert P^* \rVert_{L^\infty} \lVert \Bar{A} \rVert_{L^\infty} + \lVert Q \rVert_{L^\infty} \lVert H \rVert_{L^\infty} \right) \right) \lVert W \rVert_{L^2(I^2)}}{\underline{\lambda}^R (1 - T (\lVert A \rVert_{L^\infty} + \lVert B \rVert_{L^\infty} \lVert K^* \rVert_{L^\infty} + \lVert \Bar{A} \rVert_{L^\infty} \lVert W \rVert_{L^2(I^2)}))^2} \lVert \mu^{(n)} - \mu^* \rVert_{\mathcal{C}}.
        \end{split}
    \end{equation}
    Now, observe that by Theorem \ref{theorem:K_CONVERGENCE},
    \begin{equation*}
        \begin{split}
            \lVert K^{(n+1)} - K^* \rVert_{L^2([0,T])} \le \sqrt{C_3^K} \left(1 - \eta_K C_2^K \right)^{L_n^K/2} \lVert J_1^\cdot(K^{(n)}) - J_1^\cdot(K^*) \rVert_{L^2(I)}^{1/2},
        \end{split}
    \end{equation*}
    and by Theorem \ref{theorem:G_CONVERGENCE},
    \begin{equation*}
        \lVert G^{(n+1)} - G^{(n+1),*} \rVert_{L_B^2} \le (1 - \eta_G m)^{L_n^G/2} \lVert G^{(n)} - G^{(n+1),*} \rVert_{L_B^2}.
    \end{equation*}
    Applying those error bounds to \eqref{eq:conv_proof_mu_2} yields,
    \begin{equation*}
        \begin{split}
            \lVert \mu^{(n+1)} - \mu^* \rVert_{\mathcal{C}} &\le  \frac{\omega_n}{3} + \frac{\sqrt{T} \lVert B \rVert_{L^\infty}}{1 - M_1} (1 - \eta_G m)^{L_n^G/2} \lVert G^{(n)} - G^{(n+1),*} \rVert_{L_B^2} \\
            &\quad+ \sqrt{C_3^K} C_{\Hat{\mu}, \Bar{\mu}_E}^{(n+1),*} \left(1 - \eta_K C_2^K \right)^{L_n^K/2} \lVert J_1^\cdot(K^{(n)}) - J_1^\cdot(K^*) \rVert_{L^2(I)}^{1/2} \\
            &\quad+ M_2 \lVert \mu^{(n)} - \mu^* \rVert_{\mathcal{C}}.
        \end{split}
    \end{equation*}
    By the choice of $L_n^K$ and $L_n^G$,
    \begin{equation*}
        \lVert \mu^{(n+1)} - \mu^* \rVert_{\mathcal{C}} \le \omega_n + M_2 \lVert \mu^{(n)} - \mu^* \rVert_{\mathcal{C}}.
    \end{equation*}
    Applying this inequality inductively gives the desired error bound for the population state mean,
    \begin{equation}
    \label{eq:mu_convergenceresult_inproof}
        \begin{split}
            \lVert \mu^{(N)} - \mu^* \rVert_{\mathcal{C}} &\le \sum_{n=0}^{N-1} M_2^{N-1-n} \omega_n + M_2^N \lVert \mu^{(0)} - \mu^* \rVert_{\mathcal{C}} \le \frac{\varepsilon}{2} + M_2^N \lVert \mu^{(0)} - \mu^* \rVert_{\mathcal{C}} \le \varepsilon,
        \end{split} 
    \end{equation}
    where the last inequality holds by the choice of $N$ and Assumption \ref{assumption:convergence}. It remains to analyze the error bounds for the policy parameter.
    Firstly, by the choice of $L_{N-1}^K$,
    \begin{equation*}
        \begin{split}
            \lVert K^{(N)} - K^* \rVert_{L^2([0,T])} &\le \sqrt{C_3^K} \left(1 - \eta_K C_2^K \right)^{L_{N-1}^K/2} \lVert J_1^\cdot(K^{(N-1)}) - J_1^\cdot(K^*) \rVert_{L^2(I)}^{1/2} \le \varepsilon.
        \end{split}
    \end{equation*}
    Secondly, using the triangle inequality, and applying the inequalities \eqref{eq:bound_G_GD_tolocaloptim_final} and \eqref{eq:bound_G_local_optim_final},
    \begin{equation*}
        \begin{split}
            \lVert G^{(N)} - G^* \rVert_{L_B^2} &\le \lVert G^{(N)} - G^{(N),*} \rVert_{L_B^2} + \lVert G^{(N),*} - \Bar{G}^{(N),*} \rVert_{L_B^2} + \lVert \Bar{G}^{(N),*} - G^* \rVert_{L_B^2} \\
            &\le \lVert G^{(N)} - G^{(N),*} \rVert_{L_B^2} + C_{\Hat{\mu}}^{(N),*} \lVert K^{(N)} - K^* \rVert_{L^2([0,T])} + (M_G - 1) \lVert \mu^{(N)} - \mu^* \rVert_{\mathcal{C}} \\
            &\le 2 \omega_{N-1} + (M_G - 1) \varepsilon \\
            &\le \varepsilon \left(1 + M_G - 1 \right) = \varepsilon M_G,
        \end{split}
    \end{equation*}
    where the second to last inequality follows by the choice of $L_{N-1}^K$ and $L_{N-1}^G$, and \eqref{eq:mu_convergenceresult_inproof}. The last inequality follows by the choice of $(\omega_{n})_{n \in \{0,1, \dots, N-1\}}$. This concludes the proof.
\end{proof}

\section*{Acknowledgments}
Philipp Plank is supported by the Roth Scholarship by Imperial College London and the Excellence Scholarship by Gesellschaft f\"ur Forschungsf\"orderung Nieder\"osterreich (a subsidiary of the province of Lower Austria).

\bibliographystyle{plain}
\bibliography{literature.bib}
\newpage

\appendix

\section{Policy gradient estimation using zeroth-order method}
\label{appendix:gradient_est}

Here we  present  the 
model-free 
gradient estimation of the policy parameters used in Section \ref{sec:numerical_model_free}. 
In the experiment,
we set the number of trajectories  for the gradient estimation to be  $10$ and the magnitude of  policy perturbation  to be $\sigma_\epsilon = 0.25$.

\begin{algorithm}[!ht]
\caption{Model-free gradient estimation for Algorithm \ref{algo}}
\label{algo_modelfree}
\begin{algorithmic}[1]
\STATE \textbf{Input:} Number of trajectories  $N$ for gradient estimation, number of policy parameters $N_{\mathrm{policy}} + 1$, number of players $N_{\mathrm{player}}$, policy parameters $(K, G)$ and graphon aggregate $Z$
\STATE \textbf{Output:} $\widehat{\nabla_K} J_1(K)$, $\widehat{\nabla_G} J_2^{\alpha_j}(K, G^{\alpha_j}, Z^{\alpha_j})$ for all $j \in \{1,2,\dots,N_{\mathrm{player}} \}$
\FOR{$n = 0, \dots, N - 1$}
    \FOR{$i = 0, \dots, N_{\mathrm{policy}}$}
        \STATE Simulate $\epsilon^{(n),(i)} \sim \mathcal{U}(\{-\sigma_\epsilon, \sigma_\epsilon\})$
        \STATE Define the perturbed policy parameter: $K^{(n), (i)} = (K_{\tau_0}, \dots, K_{\tau_i} + \epsilon^{(n),(i)}, \dots, K_{\tau_{N_{\mathrm{policy}}}})$
        \STATE Simulate the cost at timestep $\tau_i$ using the oracle: $(\hat{J}_1^{(n)}(K^{(n), (i)}))_{\tau_i}$
    \ENDFOR
\ENDFOR
\STATE Estimate the gradients for each $i \in \{0,1,\dots, N_{\mathrm{policy}}\}$ as
\begin{equation*}
    \widehat{\nabla_{K_{\tau_i}}} J_1(K) = \frac{1}{N \sigma_\epsilon^2} \sum_{n = 0}^{N-1} (\hat{J}_1^{(n)}(K^{(n), (i)}))_{\tau_i} \epsilon^{(n),(i)}.
\end{equation*}
\FOR{$j = 1, \dots, N_{\mathrm{player}}$}
    \FOR{$n = 0, \dots, N - 1$}
        \FOR{$i = 0, \dots, N_{\mathrm{policy}}$}
            \STATE Simulate $\epsilon^{(n),(i)} \sim \mathcal{U}(\{-\sigma_\epsilon, \sigma_\epsilon\})$
            \STATE Define the perturbed policy parameter: $G^{(n), (i), \alpha_j} = (G_{\tau_0}^{\alpha_j}, \dots, G_{\tau_i}^{\alpha_j} + \epsilon^{(n),(i)}, \dots, G_{\tau_{N_{\mathrm{policy}}}}^{\alpha_j})$
            \STATE Simulate the cost at timestep $\tau_i$ using the oracle: $(\hat{J}_2^{(n), \alpha_j}(K, G^{(n), (i), \alpha_j}, Z^{\alpha_j}))_{\tau_i}$
        \ENDFOR
    \ENDFOR
    \STATE Estimate the gradients for each $i \in \{0,1,\dots, N_{\mathrm{policy}} \}$ as
    \begin{equation*}
        \widehat{\nabla_{G_{\tau_i}}} J_2^{\alpha_j}(K, G^{(n), (i), \alpha_j}, Z^{\alpha_j}) = \frac{1}{N \sigma_\epsilon^2} \sum_{n = 0}^{N-1} (\hat{J}_2^{(n), \alpha_j}(K, G^{(n), (i), \alpha_j}, Z^{\alpha_j}))_{\tau_i} \epsilon^{(n),(i)}.
    \end{equation*}
\ENDFOR
\end{algorithmic}
\label{algorithm:gradest}
\end{algorithm}

\end{document}